\documentclass[11pt, usenames, dvipsnames]{article}
\usepackage{amsthm}
\usepackage{amsmath,amsthm,amsfonts,amssymb,mathrsfs,bm,graphicx,stmaryrd,dsfont, enumerate}
\usepackage{booktabs}
\usepackage{dsfont}
\usepackage[colorlinks=true,linkcolor=blue,citecolor=blue]{hyperref}

\usepackage{todonotes}

\usepackage{fullpage}
\usepackage[american]{babel}
\usepackage[varg]{pxfonts}
\usepackage{microtype}
\usepackage{tikz}
\usepackage{pgfplots}
\usepackage{pgf,tikz,pgfplots}
\pgfplotsset{compat=1.14}
\usepackage{mathrsfs}
\usetikzlibrary{arrows}
\pgfplotsset{compat=newest}
\usepackage[font=small,labelfont=bf]{caption}
\usepackage[all]{xy}

\newcommand{\myeq}{:=}

\newtheorem{theorem}{Theorem}[section]
\newtheorem{problem}[theorem]{Problem}

\newtheorem{lemma}[theorem]{Lemma}
\newtheorem{remark}[theorem]{Remark}
\newtheorem{proposition}[theorem]{Proposition}
\newtheorem{corollary}[theorem]{Corollary}
\newtheorem{conjecture}[theorem]{Conjecture}

\newtheorem{definition}[theorem]{Definition}
\newtheorem{example}[theorem]{Example}
\newtheorem{observation}[theorem]{Observation}

\newcommand{\edge}{\Delta}

\newcommand{\C}{\mathbb{C}}

\newcommand{\B}{\mathcal{B}}
\newcommand{\A}{\mathcal{A}}

\newcommand{\eps}{\epsilon}

\newcommand{\Tr}{\mathrm{Tr}}

\newcommand{\G}{\mathcal{G}}

\newcommand{\T}{\mathcal{T}}

\newcommand{\bZ}{\mathbb{Z}}

\newcommand{\diredge}{f}

\newcommand{\che}{\check{\diredge}}

\newcommand{\Spec}{\mathrm{Spec}}
\newcommand{\spr}{\mathrm{spr}}
\newcommand{\im}{\mathrm{Im}}
\newcommand{\calH}{\mathcal{H}}
\newcommand{\calB}{\mathcal{B}}
\newcommand{\calA}{\mathcal{A}}
\newcommand{\calG}{\mathcal{G}}
\newcommand{\calV}{\mathcal{V}}

\newcommand{\calK}{\mathcal{K}}
\newcommand{\calC}{\mathcal{C}}
\newcommand{\calT}{\mathcal{T}}
\newcommand{\calU}{\mathcal{U}}
\newcommand{\calE}{\mathcal{E}}
\newcommand{\calTf}{\mathcal{T}_{\mathrm{full}}}

\newcommand{\bR}{\mathbb{R}}
\newcommand{\bC}{\mathbb{C}}
\newcommand{\bF}{\mathbb{F}}

\newcommand{\spa}{\mathrm{Span}}

\newcommand{\func}{\varphi}

\newcommand{\dos}{\mu_{A_{\mathcal{T}}}}

\newcommand{\cred}{C_{\mathrm{red}}^*}

\title{ Spectra of Infinite Graphs via Freeness with Amalgamation}
\author{Jorge Garza-Vargas \\ jgarzavargas@berkeley.edu \\ UC Berkeley \and Archit Kulkarni \\ akulkarni@berkeley.edu\\ UC Berkeley}
\date{\today}

\begin{document}

\maketitle

\begin{abstract}
%\todo{Draft} 
We use tools from  free probability  to study the spectra of Hermitian  operators on infinite graphs.  Special attention is devoted  to universal covering trees of finite  graphs. For operators on these graphs we derive a new variational formula for the spectral radius and provide new proofs of  results due to Sunada and  Aomoto using free probability.

With the goal of extending the applicability of free probability techniques beyond universal covering trees, we introduce a new combinatorial product operation on graphs and show that, in the non-commutative probability context, it corresponds to the notion of freeness with amalgamation.   We show that  Cayley graphs of amalgamated free products of groups, as well as universal covering trees, can be constructed using our graph product.  %and can therefore be analyzed using our framework.  

%Finally, we discuss the special case of Cayley graphs of amalgamated free products of groups.
\end{abstract}

\pagebreak

\tableofcontents

\pagebreak

\section{Introduction}

In the present work, by a \emph{graph} we mean a locally finite undirected graph allowing loops and multi-edges, unless otherwise specified. Given a graph $\mathcal{G}$, we denote its vertex set by $V(\mathcal{G})$ and its edge set by $E(\mathcal{G})  $\marginpar{$V(\mathcal{G}), E(\mathcal{G})$}. Although we will only be considering undirected graphs it will be convenient to view $E(\calG)$ as a set of directed edges, where each $\diredge \in E(\calG)$ will have a \emph{source} $\sigma(\diredge)\in V(\calG)$\marginpar{$\sigma, \tau$} and a \emph{target} $\tau(\diredge)\in V(\calG)$,  and $E(\calG)$ will be equipped with the involution $\,\check{} : E(\calG)\to E(\calG)$ which to each edge $\diredge$ assigns its reversed edge $\check{\diredge}$.\marginpar{$\diredge, \che$} Given any vertex $u\in V(\calG)$ it will prove useful  to define the sets
$$\sigma(u) := \{\diredge\in E(\calG): \sigma(\diredge)=u\} \quad \text{and} \quad \tau(u):= \{\diredge\in E(\calG) : \tau(\diredge)=u\}.$$
Oftentimes we will be working with graphs $\calG$ that are endowed with \emph{edge weights} and a \emph{vertex potential}. Formally, the edge weights will be given by a function $a:E(\calG) \to \bR$\footnote{We are slightly deviating from the standard definition of edge weight, which requires the $a_\diredge$  to be strictly positive, since in this work  allowing negative coefficients sometimes simplifies our exposition. Moreover, in the important case of universal covering trees  this is not an actual discrepancy, since  gauge invariance allows one to ``turn" negative weights into positive ones without affecting the spectrum.} that is symmetric in the sense that $a_\diredge=a_{\che}$ for all $\diredge\in E(\calG)$\marginpar{$a, b$}, and the vertex potential will be given by a function $b: V(\calG)\to \bR$. 

The adjacency matrix, the graph Laplacian matrix, and transition matrices for symmetric random walks all fall under the umbrella of so-called \emph{Jacobi matrices on graphs} \cite{avni2020periodic}.  These are bounded Hermitian operators associated to weighted graphs with vertex potential. More specifically, if $\calG$ is a graph with edge weights $a$ and vertex potential $b$, then the Jacobi matrix on $\calG$ is the bounded operator $A_\calG$ acting on $\calH:= \ell^2(V(\calG))$ by 
\begin{equation}\label{eqn:jacobi} (A_\mathcal{G} \xi)(u) = b_u \xi(u) + \sum_{\substack{ \diredge \in \sigma(u)}} a_\diredge \xi(\tau(\diredge)). \end{equation}
When working with Jacobi matrices on graphs (which will be the main focus of this paper) it is convenient to consider the \emph{spectral measures} associated to the vertices of the graph. To be precise, given a vertex $u \in V(\mathcal{G})$ one can define the state $\func_u: B(\mathcal{H}) \to \mathbb{C}$ \marginpar{$\func_u$} as
\begin{equation}
\label{eq:rootedstate}
 \func_u(X) \myeq \langle \delta_{u}, X \delta_{u} \rangle, \quad \forall X \in B(\mathcal{H}), 
\end{equation}
where $\delta_u \in \ell^2(V(\mathcal{G}))$ \marginpar{$\delta_u$} denotes the indicator function of the singleton $\{u\}$. Then, since $A_\mathcal{G}$ is self-adjoint, its spectrum  $\Spec(A_\mathcal{G})$\marginpar{$\Spec(\cdot)$} is contained in $\mathbb{R}$, and  by the functional calculus  $\func_u$ induces a probability measure $\mu_{A_\mathcal{G}, u}$\marginpar{$\mu_{A_\mathcal{G}, u}$} supported on $\Spec(A_\mathcal{G})$ with the property that 
$$\int_\mathbb{R} x^k d\mu_{A_\mathcal{G}, u}(x) = \func_u\left(A_\G^k\right)$$
for all $k \in \mathbb{Z}_{\geq 0}$. Throughout this work we  refer to $\mu_{A_\mathcal{G}, u}$ as the spectral measure of $A_\mathcal{G}$ with respect to the root $u \in V(\mathcal{G})$.  In terms of functional analysis, $\mu_{A_\mathcal{G}, u}$ is the composition of the usual projection-valued spectral measure of $A_{\mathcal{G}}$ with $\func_u$. 

The spectra and spectral measures of operators on  infinite graphs  have been extensively studied in the last several decades in different contexts. But despite significant progress in the area, current mathematical tools are still unable to answer simple fundamental questions. 

Since our goal is to provide a new perspective on  problems of current interest, the content of this paper is a combination of new results, new proofs of existing results, and new tools for the analysis of infinite graphs, all through the lens of free probability. 

Our discussion includes the following families of graphs.

\paragraph{Cayley Graphs.} Let $G$ be a finitely generated group and $S\subset G$ be a finite symmetric generating set; i.e. we assume that if $s \in S$ then $s^{-1}\in S$. We denote by $\Gamma = \Gamma(G, S)$ \marginpar{$\Gamma(G,S)$} the left  Cayley graph of $G$ with respect to $S$. Note that since $S$ is symmetric $\Gamma$ is undirected and by definition of $\Gamma$, any symmetric weighting $a: S \to \mathbb{R}^+$ (i.e. $a(s)=a(s^{-1})$) induces a symmetric weighting on the edges of $\Gamma$ in the obvious way. Typically, in this context vertex potentials are not considered; that is, one takes the constant function $b\equiv 0$.  The canonical  measure associated to $A_\Gamma$ is the spectral measure of $\Gamma$ with respect to the identity $e\in G$. Both $\Spec(A_\Gamma)$ and $\mu_{\Gamma, e}$ have been studied thoroughly in the context of random walks on groups \cite{kesten1959symmetric, woess1986nearest, woess1987context,mclaughlin1988random, cartwright1986random}. However, several basic spectral questions about some natural Cayley graphs remain open \cite{kollar2019line}.     

The amalgamated free product of groups inspired the notions of free independence and freeness with amalgamation \cite{voiculescu1985symmetries}, and since Voiculescu's seminal work it became apparent that tools from free probability can provide important insights into the spectral theory of certain Cayley graphs. In this work we focus on a new connection, namely the role of freeness with amalgamation in the study of universal covering graphs, which are not always Cayley graphs.

\paragraph{Universal Covering Graphs.} Let $\mathcal{G}$ be a connected graph with $n$ vertices.  One can form its \emph{universal covering graph}\footnote{See Section \ref{subsecuniversalcovers} for a precise definition and a subtlety involving covers of loops.} (also called universal covering tree or universal cover), 
denoted here by $\mathcal{T}(\mathcal{G})$\marginpar{$\calT(\calG)$}, or simply by $\calT$ when $\calG$ is clear from the context.   Recall that $\mathcal{T}$ is an infinite tree if $\mathcal{G}$ has at least one cycle or loop and is $\mathcal{G}$ itself when $\mathcal{G}$ is a tree.  In this context we will often refer to $\calG$ as the \emph{base graph.}

The universal covering graph comes with a covering map $\Xi: \mathcal{T}\to \mathcal{G}$. So, when $\calG$ is equipped with edge weights $a$ and  vertex potential $b$, one can use $\Xi$\marginpar{$\Xi$} in the obvious way to lift $a$ and $b$ to functions on $E(\calT)$ and $V(\calT)$ respectively, and with this equip $\calT$ with (periodic) edge weights and vertex potential.   The corresponding Jacobi matrix $A_{\calT}$\marginpar{$A_\calT$} can then be viewed as a \emph{pullback} of $A_\calG$, and it is referred to as a \emph{periodic Jacobi matrix with period $n$} \cite{avni2020periodic} or a \emph{pulled-back local operator} \cite{angel2015non}. 

It is standard to  associate  to $A_\mathcal{T}$ a set of spectral measures as follows. For any $u \in V(\mathcal{G})$ we fix any representative $\tilde{u}\in V(\mathcal{T})$ in the fiber $\Xi^{-1}(u)$ and consider the spectral measure $\mu_u:= \mu_{A_{\calT}, \tilde{u}}$\marginpar{$\mu_u$}. We can then associated the following canonical measure to $A_{\calT}$:\marginpar{$\dos$}
$$\dos := \frac{1}{n} \sum_{u\in V(\calG)} \mu_u, $$
which is referred to  as the \emph{density of states} of $A_\mathcal{T}$ in accordance with the physics and spectral theory literature. 

%From recent work of Bordenave and Collins \cite{bordenave2019eigenvalues}, we know that for any finite graph $\mathcal{G}$, both $\Spec(A_\mathcal{T})$ and $\mu_{\mathcal{T}}$ strongly dictate the spectral behaviour of large random lifts of $\mathcal{G}$, even to the extent of controlling all outliers in the spectrum. In turn, random lifts and their asymptotic behaviour are  of interest in the study of expander graphs \cite{marcus2013interlacing, hall2018ramanujan}. This is one of our many motivations to understand $\mu_{A_\T}$ and we refer the reader to \cite{avni2020periodic} for a detailed exposition of other interesting connections in spectral theory. 

It is a well known fact that $\T$ is the limit, in the Benjamini-Schramm sense, of random lifts of $\G$ (see Section \ref{subsecrandlifts}). Hence, $A_\T$  can be regarded as a limit of random matrices. A bit of thought from the free probability perspective shows that in fact $A_\T$ can be viewed as an operator-valued matrix with free entries (see Section \ref{secnumberofbands}). This is the starting point of the present work. 

Previous results give an explicit description of  $\Spec(A_\mathcal{T})$ and $\dos$ when $\mathcal{G}$ has a particular structure \cite{kesten1959symmetric, mckay1981expected, godsil1988walk, figa1994harmonic}. Others have made some progress in the case when $\mathcal{G}$ is an arbitrary graph \cite{aomoto1988algebraic,aomoto1991point, sy1992discrete,  sunada1992group,  avni2020periodic}. However, as we discuss in the last section, many fundamental questions remain open (also see \cite{avni2020periodic}).    

\paragraph{Amalgamated Free Product for Graphs.} As mentioned above, since generic universal covering graphs are not Cayley graphs,  the emergence of freeness with amalgamation in this context might seem somewhat fortuitous. Here, we clarify this connection  by introducing a graph product that corresponds to the notion of freeness with amalgamation. 

 Here is some context.  Inspired by the combinatorial description of Cayley graphs of free products of groups, Quenell \cite{quenell1994combinatorics} introduced the free product of graphs. Only later it was understood by Accardi, Lenczewski, and Sa{\l}apata \cite{accardi2007decompositions} that this graph product is equivalent Voiculescu's free product of Hilbert spaces \cite{voiculescu1992free} and that free probability can be used to compute the spectra of graphs arising from Quenell's graph product. We refer the reader to Section \ref{sec:graph-products} for a detailed and more precise discussion. 
 
In the spirit of \cite{quenell1994combinatorics} and \cite{accardi2007decompositions}, in this paper we define a combinatorial graph product and  show that the machinery of freeness with amalgamation can be used to compute the spectra of graphs arising from this product. In particular, universal covering trees and Cayley graphs of amalgamated products of groups can be constructed using our product.

\subsection*{Bibliographic Note}

After posting the first version of this article to the arXiv, we became aware of the work of Avni, Breuer and Simon \cite{avni2020periodic} posted six weeks prior.  Upon reading their work, we have revised our article in a few ways.  First, in place of our previous notion of ``adjacency operators on weighted graphs,'' we have adopted their terminology of Jacobi matrices on graphs, both to provide consistency in the literature and because it provides a useful distinction between diagonal elements $b_v$ and loops (which behave differently upon taking covers).  Second, our theorem on the number of bands in the spectrum was demonstrated in \cite{avni2020periodic} to be implicit in work of Sunada \cite{sunada1992group}, which we had not realized.  Both our independent  proof and the proof given in \cite{avni2020periodic} argue that the Jacobi operator of a universal cover can be viewed as a specific element of a matrix $C^*$-algebra, and then use a standard K-theory argument which relies on a theorem of Pimsner and Voiculescu. However, these two proofs differ in the way in which the connection with the $C^*$-algebra is made.  The proof we present here uses random lifts and the fact that independent random permutation matrices are asymptotically free (an idea that has previously been exploited in \cite{bordenave2019eigenvalues} for different purposes), while the proof in  \cite{avni2020periodic} deals directly with the Jacobi operator on a concrete Hilbert space. Given the relevance of this result, we have decided to keep our alternative argument in this work, but we no longer state the result as ours.  Finally, as we are able to answer some questions left open in \cite{avni2020periodic}, we include these answers in Appendix \ref{subsec: bandexamples} of this revision. 

To the best of our knowledge, other than what is mentioned in the above paragraph, there is no further overlap between our work and \cite{avni2020periodic}. 

\subsection{Results}

\subsubsection{Universal Covers}

In this section we will consider a finite graph $\G$ with $n$ vertices, universal cover $\T$ and covering map $\Xi: \mathcal{T}\to \mathcal{G}$. The base graph $\calG$ will be equipped with edge weights $a$ and vertex potential $b$, and $\calT$ will be equipped with the induced periodic edge weights and vertex potential. 

Roughly speaking, the proofs of the results in this section use free probability in two different ways. 
\begin{enumerate}[(1)]
    \item  Free probability techniques are used to argue that periodic  Jacobi matrices on $\T$ can be represented as $n\times n$ matrices with entries in a certain $C^*$-algebra. For different applications it is convenient to use different $C^*$-algebras. 
    \item Once a given periodic Jacobi matrix $A_\calT$ is viewed as an element of a matrix $C^*$-algebra, we argue that such an element can be decomposed as a sum of simpler  operators that are free with amalgamation over the algebra $M_n(\C)$ (see Section \ref{secpreliminaries} for precise definitions). This decomposition allows the use of  tools  from   free probability  to understand the spectrum of $A_\calT$.     
\end{enumerate}
Regarding step (1), it is worth noting that once one has represented the Jacobi operator as an specific element of a matrix $C^*$-algebra, one might be able to find elementary proofs that show that such a representation is correct. %, and then the initial free probability argument might seem depreciated.
However, one should appreciate that it is not clear a priori that this connection with $C^*$-algebras exists, and neither is it easy to 
"guess"  what the correct representation of the Jacobi operator on a given $C^*$-algebra is. It is for this reason that we decided to include the free probability arguments for step (1), since we believe they provide a conceptual way to arrive at the $C^*$-algebra representations. %conceptual understanding of this phenomenon. 

\paragraph{Spectral Radius.} Let $m$ be the number of undirected edges in $\G$  and let $\Gamma_m$ be the discrete group obtained by taking the free product of $m$ copies of $\mathbb{Z}_2$. Using our combinatorial graph product, which is discussed below in Section \ref{subsub:amalgamgraphprod}, we show that periodic Jacobi operators on $\T$ can be represented as elements in $M_n(\C)\otimes \cred(\Gamma_m)$  (see Section \ref{secpreliminaries} for definitions). We then use a theorem of Lehner  \cite{lehner1999computing} to prove the following min-max  formula for the spectral radius (actually, both spectral edges) of $A_\calT$. 

\begin{theorem}[Formula for the spectral radius]
\label{thm:radius}
Let $\calG$ be a graph with vertex set $[n]$,  edge weights $a$, and vertex potential $b$. Let $\calT$ be its universal cover with the induced periodic edge weights and vertex potential.  If $\rho_r(A_\T)$ \marginpar{$\rho_r$} denotes the right edge (i.e. maximum element) of $\Spec(A_\T)$ then
\begin{equation}
\label{eq:rho}
\rho_r(A_\mathcal{T}) = \inf_{y_1, \dots, y_n>0} \max_{i \in [n]} \bigg[ b_i + \frac{1}{2y_i} \bigg( 2-\deg(i)  + \sum_{\diredge\in \sigma(i)} \big(1+4a_{\diredge}^2 y_{\sigma(\diredge)} y_{\tau(\diredge)}\big)^{1/2} \bigg) \bigg].    
\end{equation}
where $\deg(i)$ is the degree of the vertex $i$ in $\mathcal{G}$. Moreover, the  infimum can be restricted to those $n$-tuples $(y_1, \dots, y_n)$ for which the $n$ expressions inside the $\max$ symbol are equal to each other. 
\end{theorem}

\begin{remark}[Left edge and spectral radius]
Using the fact that $-A_\calT$ is also a periodic Jacobi matrix on $\calT$, one may obtain a similar expression for the left edge $\rho_l$ of $\Spec(A_\calT)$. And to compute the  spectral radius $\spr(A_{\calT})$ use the trivial observation $\spr(A_{\calT}) = \max\{\rho_r, - \rho_l\}$. 
\end{remark}

\begin{remark}[The symmetric case]
In the case where $b_v = 0$ for all $v \in V(\G)$, the spectrum of $A_\T$ is symmetric about zero, because $\T$ is a bipartite graph. So, in this case the spectral radius equals $\rho_r=\rho_l$. 
\end{remark}
%Applying the formula to $-A_\T$ yields an expression for the left edge of the spectrum. 

\begin{remark}[Algebraic expressions]
One can  use Lagrange multipliers on the above variational problem to find an explicit algebraic description of $\rho_r(\mathcal{T})$ from equation (\ref{eq:rho}); see Corollary \ref{cor:lagrange}.%  \todo{}). 
\end{remark}

\paragraph{Aomoto's Equations.} For $u \in V(\mathcal{G})$, recall that $\mu_u$ denotes the spectral measure of $A_\calT$ associated to a vertex in $\Xi^{-1}[u]$.   We may then form the \emph{Cauchy transform} of this measure\marginpar{$w_u(z)$}:
\begin{equation}
\label{eq:rootedcauchydef}
w_u(z) \myeq  \int_\mathbb{R} \frac{1}{z-x} d \mu_{u}(x). 
\end{equation}
The Cauchy transform is also known as the \emph{Stieltjes transform}, and is closely related to the \emph{Green function} or \emph{resolvent} $(z - A_{\mathcal{T}})^{-1}$, as well as to the walk generating function $Q_{u}(z) = \frac{1}{z} w_{u}(1/z)$ which counts weighted closed walks on $\mathcal{T}$ based at a fixed $\tilde{u} \in \Xi$.  It is a standard fact in analysis that the spectral measure $\mu_{u}$ can be fully recovered from $w_u(z)$ via the \emph{Stieltjes inversion formula}.  

Using the operator-valued version of Voiculescu's $R$-transform,  we recover Aomoto's system of  equations for the $w_u(z)$ presented below in Theorem \ref{thm:system}. %Most of the computations that we needed were already done by Lehner in \cite{lehner1999computing}. 

\begin{theorem}[Aomoto \cite{aomoto1988algebraic}] \label{thm:system}
Using the above notation,  the following system of equations holds: 
\begin{equation}
\label{eqn:system}
\begin{cases}w_u(z) = \dfrac{1}{2(z-b_u)} \left(2-\deg(u)+\displaystyle\sum_{\diredge \in \sigma(u)} \big(1+4a_{\diredge}^2 w_{\sigma(\diredge)}(z)w_{\tau(\diredge)}(z)\big)^{1/2}\right) & \forall u \in V(\calG)
\end{cases}
\end{equation}
for all $z \in \mathbb{C}$ in a neighborhood of infinity and for all real $z$ outside the convex hull of the spectrum $\Spec(A_\T)$.\footnote{Here $\deg(u)$ denotes the degree of the vertex $u$ in $\calG$, where loops contribute twice in the count. }  
\end{theorem}

\begin{remark}
Since in Theorem \ref{thm:system} we are restricting $z$ to be in  a neighborhood of infinity and Cauchy transforms vanish at infinity, the expressions inside the radicals  of (\ref{eqn:system}) are always on the right side of the complex plane, and hence the square roots are globally defined. Moreover, by analyzing the behaviour of the $w_u(z)$ at infinity, one sees that one should take the principal branch for each square root for the system of equations to hold.
\end{remark}

The above system of equations was first discovered by Aomoto  \cite{aomoto1988algebraic} using techniques from the literature of random walks on groups. See \cite[\S5 and \S6]{keller2013spectral}, \cite[\S4]{keller2014invitation} and \cite[\S 6]{avni2020periodic}   for a survey of similar results related to algebraicity of Cauchy transforms. In particular see \cite[\S 6]{avni2020periodic} for a discussion of the role of algebraicity in showing that  $A_\T$ has no singular continuous spectrum. 

In \cite{aomoto1991point} Aomoto used (\ref{eqn:system}) to obtain necessary conditions on $\mathcal{G}$ for the existence of point spectrum in $\Spec(A_\T)$. Here, in Appendix \ref{sec:usingthesystemofequations}, we take a brief detour to show that (\ref{eqn:system}) can be used to establish a connection between the behavior of the density of states at the right edge of $\Spec(A_\T)$ and the growth rate of $\T$. We prove the following. 

\begin{theorem}[Vanishing density at the right edge]
\label{thm:vanishingdensity}
Use the above notation and assume that $a_\diredge >0$ for all $\diredge \in E(\G)$.   Let $\dos$ be the density of states of $A_\calT$ and $\rho_r$ be  right edge of $\Spec(A_\calT)$.  Then $\dos$ is absolutely continuous in a neighborhood of $\rho_r$ and $\lim_{x\to \rho_r} \frac{d\dos}{dx}(x) =0$.
\end{theorem}

The behavior  of $\frac{d\dos}{dx}$ at the edge is tied to the type of singularity of the Cauchy transforms $w_u(z)$ at $z=\rho_r$. On the other hand, the latter have been extensively studied for many classes of infinite graphs (e.g. \cite{nagnibeda2002random, lalley2001random, woess2000random, gouezel2013random}), as an intermediate step to understand the asymptotic behavior of transition probabilities and escape rates of random walks. In particular, Theorem \ref{thm:vanishingdensity}  can be deduced from the results in \cite[\S 2]{nagnibeda2002random},  where the $w_u(z)$ are related to an auxiliary family of transforms,  and sophisticated methods (that seek to prove stronger results) from the theory of random walks on infinite graphs are used. The argument in  \cite[\S 2]{nagnibeda2002random} analyzes, as $z$ varies, the evolution of the Perron eigenvalue of the non-backtracking matrix of $\G$ \footnote{The non-backtracking matrix of $G$ is a non-Hermitian matrix whose rows and columns are indexed by the directed edges of $G$, see \cite{stark1996zeta}.} with entries weighted as a function of the auxiliary transforms. In contrast, our proof of Theorem \ref{thm:vanishingdensity}  is short and self-contained, and uncovers a nice relation between Aomoto's equations and the Perron eigenvalue of  $A_\G$.

%We believe that the converse of the above theorem is true at least under the assumption that all $a_e=1$ and all $b_v=0$. To illustrate this behavior consider the case when $\G$ is a $d$-regular graph. In this case we know that the spectral measure of $A_\T$ is the Kesten-McKay distribution of parameter $d$, and hence its density vanishes at the edge if and only if $d\geq3$, or equivalently, if and only if $\G$ is $d$-regular and has more than one cycle. 

%\todo{transition}

\bigskip

\paragraph{Sunada's Theorem on the Band Structure.} Let $m$ be the number of undirected edges of $\G$ and let $\mathbb{F}_m$ denote the free group on $m$ generators. In \cite{bordenave2019eigenvalues}, large random lifts of graphs were studied in the context of free probability. In Section \ref{secnumberofbands} we recall from \cite{bordenave2019eigenvalues} how random lifts can be used to obtain a representation of $A_\T$ as elements in $M_n(\C)\otimes \cred(\mathbb{F}_m)$. This representation can be combined with a theorem of Pimsner and Voiculescu \cite{pimsner1982k} to conclude the following result:

\begin{theorem}[Sunada] \label{thm:mass}
  Using the notation above, the density of states of $A_\T$ assigns a positive integer multiple of $1/n$ to any connected component of the spectrum of $A_\T$.  Consequently, the spectrum of $A_\T$ contains at most $n$ connected components.
\end{theorem}

\begin{remark}[Tightness of the bound] \label{rem:tight} As mentioned above, if $\G$ is a tree then $\G \cong \T$, and hence, if $\G$ has distinct eigenvalues (e.g. $\G$ is a path and $A_\G$ is its adjacency matrix), Sunada's bound on the number of components is tight. A more interesting example of tightness is any finite graph with $\calG$ with distinct $b_v$ and $\sum_{\diredge \in E(\G)} a_\diredge < \min_{u \ne v} |b_u - b_v|$; see \cite[\S 10.2]{avni2020periodic}. %and $ when $\G$ is a cycle and the $a_e$ and $b_v$ are not constrained; in this case $A_\T$ is a periodic Jacobi operator and its spectrum can have $n$ bands (see \cite[\S 2]{avni2020periodic}). 
\end{remark}

Once the $C^*$-algebra representation is obtained, the trick needed to reduce the proof of the above theorem to the theorem in \cite{pimsner1982k} is standard in $K$-theory and in the context of graph theory it was first used  by Aomoto and Kato in \cite{aomoto1988green}.  The upper bound on the number of bands of the spectrum of $A_\T$ is implicit in the work of Sunada \cite{sunada1992group}.  It was then highlighted and proved explicitly in  \cite{avni2020periodic}.  This technique was also used in \cite{kollar2019line} to establish an upper bound on the number of bands for certain infinite lattices.

In relation to questions regarding the number of bands in the spectrum of $A_\T$,  in Appendix \ref{subsec: bandexamples} we discuss the phenomenon of spectral splitting and  provide  answers to several questions in \cite{avni2020periodic} regarding possible extensions of theorems of Borg and Borg--Hochstadt. 

\subsubsection{Amalgamated Graph Products}
\label{subsub:amalgamgraphprod}

%The fact that Jacobi operators on  universal covers can be decomposed as a free convolution with amalgamation seems very fortuitous, and it is not clear a priori  how to extend such techniques to other classes of infinite graphs. 

Inspired by a series of results of different authors \cite{obata2004quantum, accardi2004monotone, accardi2007decompositions} in which it is shown that each notion of non-commutative stochastic independence corresponds to a combinatorial graph product, we investigate the possibility of associating a graph product to the notion of freeness with amalgamation \cite[3.8]{voiculescu1992free}.  Recall that freeness with amalgamation is not an independence in the sense of Muraki or Speicher \cite{muraki2002five, speicher1997universal}, but it shares many desired properties with the five independences. 

With this purpose in mind we consider the following setting. Let $\mathcal{G}_1, \dots, \mathcal{G}_n$ be finite rooted graphs, and let $\mathcal{C}_1, \dots, \mathcal{C}_n$ be disjoint sets of colors. Assume each $\mathcal{G}_i$ is equipped  with  an edge coloring $c_i: E(\mathcal{G}_i) \to \mathcal{C}_i$. Let $\mathcal{C} = \bigcup_{i=1}^n \mathcal{C}_i$ and let $\mathcal{G}$ be a finite  graph with an edge coloring $c: E(\mathcal{G})\to \mathcal{C}$. 

In this work  we define a graph called the free product of $\mathcal{G}_1, \dots, \mathcal{G}_n$ with amalgamation over $\mathcal{G}$. This product can be viewed as a procedure to construct an infinite graph by iteratively copying the local structure of the $\mathcal{G}_i$ and where the way in which these neighborhoods are combined is dictated by the structure of the graph $\mathcal{G}$ and by how the colorings $c_i$ relate to $c$. 

Now, if $\calG$ is a weighted graph with vertex potential, then these can be lifted in a natural (periodic) way to the the graph $\mathcal{K}$  constructed through this procedure. The upshot here is that the Jacobi operator  $A_\mathcal{K}$ can be written as an amalgamated free convolution of much simpler non-commutative random variables. Hence, in this situation much of the machinery developed around freeness with amalgamation (e.g. \cite{lehner1999computing,speicher1998combinatorial,voiculescu1995operations, belinschi2019atoms}) can be used to understand the spectral measures of $A_\mathcal{K}$. It turns out that the amalgamated free product of graphs is general enough to be able to construct both
\begin{enumerate}[(i)]
    \item any Cayley graph of an amalgamated free product of groups, and
    \item any universal cover of a graph.
\end{enumerate}

\subsection{Structure of the Paper}

In Section \ref{sec:motivationandrelatedwork} we discuss  related work and motivate the study of the spectral properties of $A_\T$. The latter has already been done impeccably by Avni, Breuer and Simon in \cite{avni2020periodic}. So, to minimize redundancy, we very briefly survey some of the important results mentioned in \cite{avni2020periodic}, and  limit our detailed discussion to topics that were not touched upon in \cite{avni2020periodic}: the relevance of $\Spec( A_\T)$ in  the theory of relative expanders, with particular emphasis on the right and left edges of $\Spec( A_\T)$ and the role of random lifts in this context. We also discuss the connections between non-commutative probability and spectral graph theory and spectral results about Cayley graphs in the context of free probability.   

Section \ref{secpreliminaries} contains all the machinery from the theory of free probability that will be needed in the present work, as it is our hope that this paper be of interest to multiple audiences. For example, the reader who is well-acquainted with free probability may skip Section \ref{secpreliminaries}, but might find Section \ref{sec:motivationandrelatedwork} informative, while a reader coming from a background of spectral graph theory may want to skip Sections \ref{subsecuniversalcovers}-\ref{sec:spectralradii},  review parts of Section \ref{secpreliminaries} and proceed to the sections on the spectra of universal covers (Section \ref{secnumberofbands} and Section \ref{secalgebraicuniversal}). The subsequent sections draw from different parts of Section \ref{secpreliminaries} and are developed in a more or less parallel fashion. For this reason, the reader may jump straight to any of the sections that are of their interest and go back as needed. 

In Section \ref{secnumberofbands} we prove Theorem \ref{thm:mass} using asymptotic freeness of random permutation matrices and an argument from operator $K$-theory. The knowledge from free probability required for this section is contained in Section \ref{subsecfreeprobability}. 

In Section \ref{secgraphproduct} we develop the theory behind our graph product, which we call the amalgamated free graph product. This discussion is based on the construction of Hilbert bimodules given in Section \ref{subsecHilbertProd} and a good understanding of the content of Sections \ref{subsecfreeprobability}-\ref{subsecfreenesswithamalg} is recommended. 

In Section \ref{secalgebraicuniversal} we adopt an algebraic approach to the description of the spectral measure of  $A_\mathcal{T}$. The analysis of this section is based on interpreting $A_\mathcal{T}$ as an operator-valued  matrix with free entries. This interpretation coincides with the framework used by Lehner in \cite{lehner1999computing} and his results are used in the proofs of Theorem \ref{thm:system} and Theorem \ref{thm:radius}. 

In Section \ref{sec:future} future directions are discussed. \\
%In Section \ref{secCayley} we discuss the problem of understanding the spectral measure of Cayley graphs of those groups that can be written as amalgamated free products of simpler groups \footnote{Amalgamated free product groups are a particular case of Serre's notion of amalgam. See \cite{serre1980sl} for a discussion of amalgams and their Cayley graphs. }.  This question was investigated in the '80s in the context of random walks on groups, see \cite{cartwright1986random, picardello1985random, woess1987context} for related work. Despite considerable progress, understanding the spectrum of some of these Cayley graphs is still beyond the reach of the current tools in the area (see Section \ref{secCayley} for details). With this motivation in mind, in Section \ref{secCayley}  we  discuss the bearing of free probability theory on this problem and show that our graph product encodes the combinatorial construction of the Cayley graphs in question. \\

\textbf{Note for the Free Probability Expert.} Our  contribution to the theory of free probability is the definition of the amalgamated free product of graphs presented in Section \ref{secgraphproduct}. The generality of this graph product has the potential to allow a free probability approach in  contexts that go beyond what is discussed here. 

In some sense, our construction provides a combinatorial interpretation of  Voiculescu's amalgamated free product for Hilbert bimodules when the underlying algebra is $M_n(\C)$. Moreover, our discussion from Section \ref{sec:cayley-amalgamated} has the non-trivial implication that for groups $G_1, G_2$ with a  common finite subgroup $H$, operators in $C_{\text{red}}^*(G_1\ast_H G_2)$ that are free with amalgamation over $C_{\text{red}}^*(H)$, have explicit copies in distribution in an algebra of the form $M_n(\C)\otimes B(\mathcal{H})$ where $n$ is the order of $H$. This could facilitate numerical computations of amalgamated free convolutions.  

 Sections \ref{secnumberofbands} and \ref{secalgebraicuniversal} on the other hand contain  applications of well understood tools in free probability.  Their value resides in providing a new perspective to the independently interesting problem of understanding the spectral properties of $A_\T$. 

\subsection{Definitions and Conventions} 
\label{sec:defandconv}

Let us be precise about our notion of universal cover for graphs, which is not identical to the standard topological notion of universal covering space.  For this, let $\G$ be a finite connected graph. 

When it comes to loops in $\calG$ (i.e. edges $\diredge$ for which $\sigma(\diredge)=\tau(\diredge)$) it will be useful to follow Friedman's distinction between \emph{whole-loops} and \emph{half-loops} \cite{friedman1993}. A whole-loop is a pair of distinct directed edges $\diredge_1, \diredge_2$ with $\che_1 = \diredge_2$ and $\sigma(\diredge_1)=\tau(\diredge_1)=\sigma(\diredge_2)=\tau(\diredge_2)$, whereas a half-loop is composed of a single directed edge $\diredge$ satisfying $\che =\diredge$ and $\sigma(\diredge)=\tau(\diredge)$.   Whole-loop corresponds to the standard notion of loop used in graph theory.  We will occasionally allow half-loops, for example in the definition of universal cover below, but unless explicitly stated otherwise, in this work graphs will not be permitted to have half-loops, and ``loops'' will refer exclusively to whole-loops.

A \emph{non-backtracking walk} in $\calG$ is a sequence $\diredge_1, \dots, \diredge_m$ of edges in $E(\calG)$ such that $\tau(\diredge_i)=\sigma(\diredge_{i+1})$  and such that $\diredge_{i+1}\neq \che_i$ when $\diredge_i$ is not a loop (i.e. $\sigma(\diredge_i)\neq \tau(\diredge_i)$), and $\diredge_i\neq \diredge_{i+1}$ when $\diredge_i$ is a whole or half-loop.

\begin{definition}[Universal cover of a graph]
Let $\mathcal{G}$ be a finite undirected graph, possibly with half-loops, and fix a root $v_0 \in V(\mathcal{G})$. The universal cover of $\mathcal{G}$ is the tree $\mathcal{T}(\mathcal{G})$ constructed as follows: 
\begin{enumerate}[(1)]
    \item We place one vertex in $\mathcal{T}(\mathcal{G})$ for every non-backtracking walk in $\mathcal{G}$ starting at $v_0$ (this includes the empty walk). 
    \item We connect two vertices in $\mathcal{T}(\mathcal{G})$ by an edge if the walk corresponding to one of them can be obtained by appending an edge to the walk corresponding to the other. %That is
    %$$E(\mathcal{T}(\mathcal{G})) \myeq \{ \{(v_0, \dots, v_l), (v_0, \dots, v_l, v_{l+1})\} :   v_{i-1}\neq v_{i+1} \text{ for } i=1, \dots, l  \}.$$
\end{enumerate}
\end{definition}

\begin{example}[Bouquets]
  Let $d \ge 0$.  The universal cover of a single vertex with $d$ half-loops is the $d$-regular tree, which is a single vertex for $d=0$, a single edge for $d = 1$, and infinite otherwise. On the other hand, the universal cover of a single vertex with $k$ whole-loops is the $2d$-regular tree.  
\end{example}

% \begin{definition}[Universal cover of a graph]
% Let $\mathcal{G}$ be a finite undirected graph and fix a root $v_0 \in V(\mathcal{G})$. The universal cover of $\mathcal{G}$ is the tree $\mathcal{T}(\mathcal{G})$ constructed as follows: 
% \begin{enumerate}[(1)]
%     \item We place one vertex in $\mathcal{T}(\mathcal{G})$ for every non-backtracking walk in $\mathcal{G}$ starting at $v_0$, that is
%     $$V(\mathcal{T}(\mathcal{G})) \myeq \{ w = (v_0, v_1, \dots, v_l) : w \text{ a walk in $G$, }  v_{i-1}\neq v_{i+1} \text{ for } i=1, \dots, l-1 \}. $$
%     \item We connect two vertices in $\mathcal{T}(\mathcal{G})$ by an edge if the corresponding walk for one of them can be obtained by appending a vertex to the walk of the other. That is
%     $$E(\mathcal{T}(\mathcal{G})) \myeq \{ \{(v_0, \dots, v_l), (v_0, \dots, v_l, v_{l+1})\} :   v_{i-1}\neq v_{i+1} \text{ for } i=1, \dots, l  \}.$$
% \end{enumerate}
% We will say that $\mathcal{G}$ is a base graph of $\mathcal{T}(\mathcal{G})$. 
% \end{definition}

One may also define the notion of covering map for graphs (see e.g. \cite{friedman2019relativized}), and one has that any connected cover of $\mathcal{G}$ is covered by $\mathcal{T}(\mathcal{G})$. 
\bigskip

\section{Motivation and related work}
\label{sec:motivationandrelatedwork}

\subsection{Density of States of Operators on Universal Covers}
\label{subsecuniversalcovers}

 Studying the density of states of periodic Jacobi matrices on universal covers is a hard problem and only in specific cases can explicit results be obtained. Here we provide a quick overview of what we consider the most relevant results in the context of this paper. We refer the reader to \cite{avni2020periodic} for a broader and more detailed discussion.  

The density of states of the adjacency matrix of the universal cover of  $d$-regular graphs (i.e. the  $d$-regular tree with constant edge weights $a\equiv 1$ and constant vertex potential $b\equiv 0$) was first computed by Kesten \cite{kesten1959symmetric} in the context of Cayley graphs, then revisited by McKay \cite{mckay1981expected} in the context of random graphs.  Godsil  derived a formula for the density of states of the adjacency matrix of the universal cover of bipartite biregular graphs \cite{godsil1988walk}. In the case of $d$-regular graphs, but now allowing arbitrary edge weights (but still requiring $b\equiv 0$),  Fig\`a-Talamanca and Steger \cite{figa1994harmonic} provided a complete description of the spectrum of the lift of Jacobi matrices to the infinite $d$-regular tree. 

In the more general setting of arbitrary universal covers,  Aomoto \cite{aomoto1988algebraic} (see Theorem \ref{thm:system} above)  derived a system of coupled equations (with square roots) for the Cauchy transforms of the spectral measures of periodic Jacobi matrices. In a similar spirit,  polynomial coupled equations for certain transforms were then obtained by Avni, Breuer and Simon \cite{avni2020periodic} in the context of periodic Jacobi matrices on trees, and by Keller, Lenz and Warzel \cite{keller2013spectral} in the context of infinite trees of finite cone type (which are a family of infinite trees that contains universal covers \cite{keller2014invitation}).

In \cite{aomoto1991point} Aomoto used the coupled equations for the Cauchy transforms that he obtained in \cite{aomoto1988algebraic} to show that periodic Jacobi matrices on $d$-regular trees have no point spectrum. A result in a similar direction was later obtained in \cite{keller2013spectral}, where it was shown that under some conditions on the  base graph (that allow non-constant degree graphs),  the density of states of the adjacency matrix of the universal cover is absolutely continuous. Certainly, this result does not apply to all universal covers, since the density of states may contain atoms; see Table \ref{table1}. A  general result was obtained in \cite{avni2020periodic}, where it was  shown that periodic Jacobi matrices on trees  have no singular continuous spectrum. Finally, we point out the work of Sunada \cite{sunada1992group}, which was discussed above (see Theorem \ref{thm:mass}). 

 The reasons for studying the spectrum of operators on universal covers have varied  from author to author.   Our main motivation is that, as shown in \cite{bordenave2019eigenvalues}, the spectra of these infinite objects govern the behaviour of large random lifts of finite graphs, and on the other hand, random lifts have been instrumental for obtaining groundbreaking results on the existence of optimal (\emph{Ramanujan}) expanders \cite{marcus2013interlacing, hall2018ramanujan}. We refer the reader to Section \ref{subsecrandlifts} for a definition and discussion of random lifts. %\todo{discuss friedman's work here and in sec \ref{subsecrandlifts}}

That said,  the complexity and variety of features  that one can observe when looking at the spectrum of different universal covers is appealing in its own right, and it is the source of many interesting spectral problems. To exemplify this, below we include some simulations and observations. 
\begin{table}[h]
    \centering
    \begin{tabular}{c c}
    \begin{tabular}{c c}
         \includegraphics[scale=3.6]{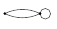} & \includegraphics[scale=.22]{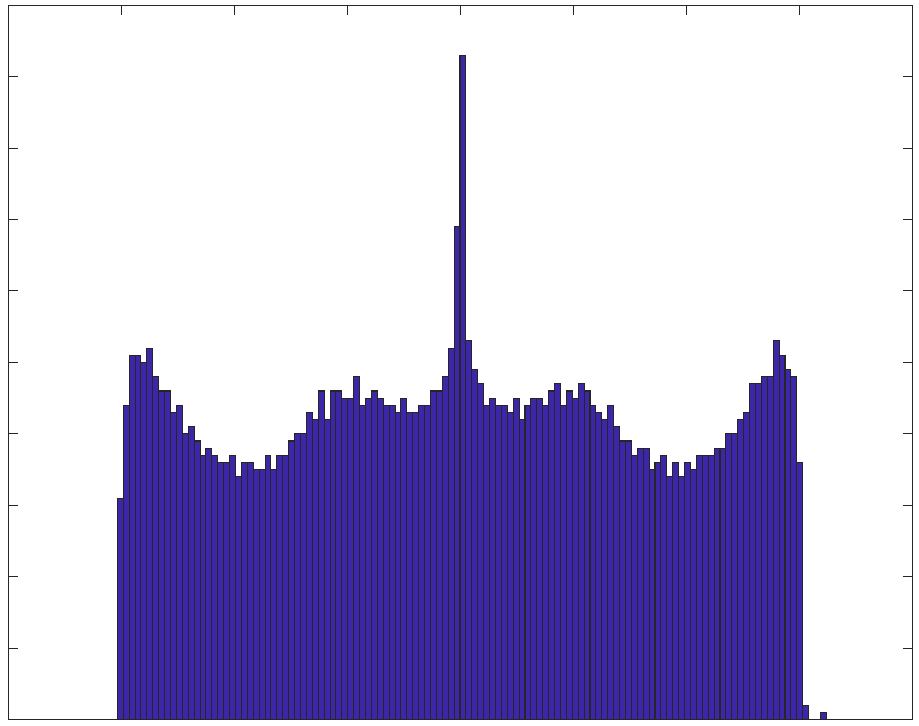}\\
           $\mathcal{G}_1$ &   $\mu_{\mathcal{T}_1}$  \\ \includegraphics[scale=.6]{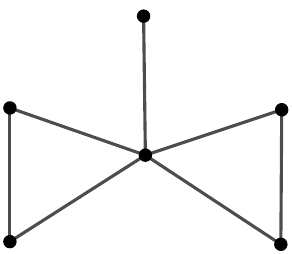} & \includegraphics[scale=.21]{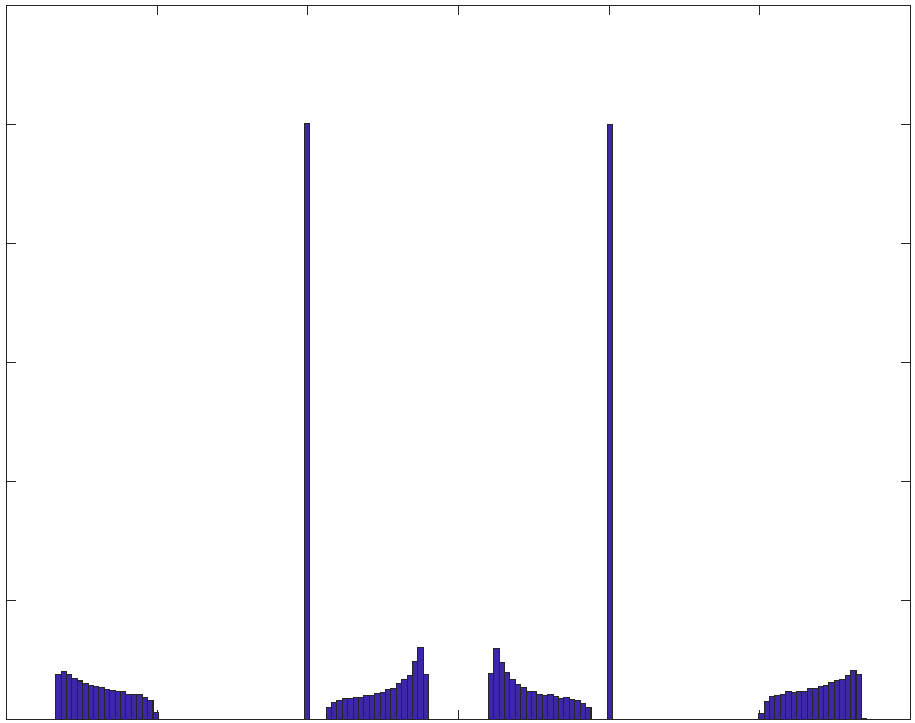}\\ 
             $\mathcal{G}_3$ &  $\mu_{\mathcal{T}_3}$ 
             
           \\ \includegraphics[scale=.58]{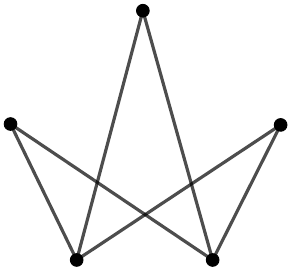} & \includegraphics[scale=.22]{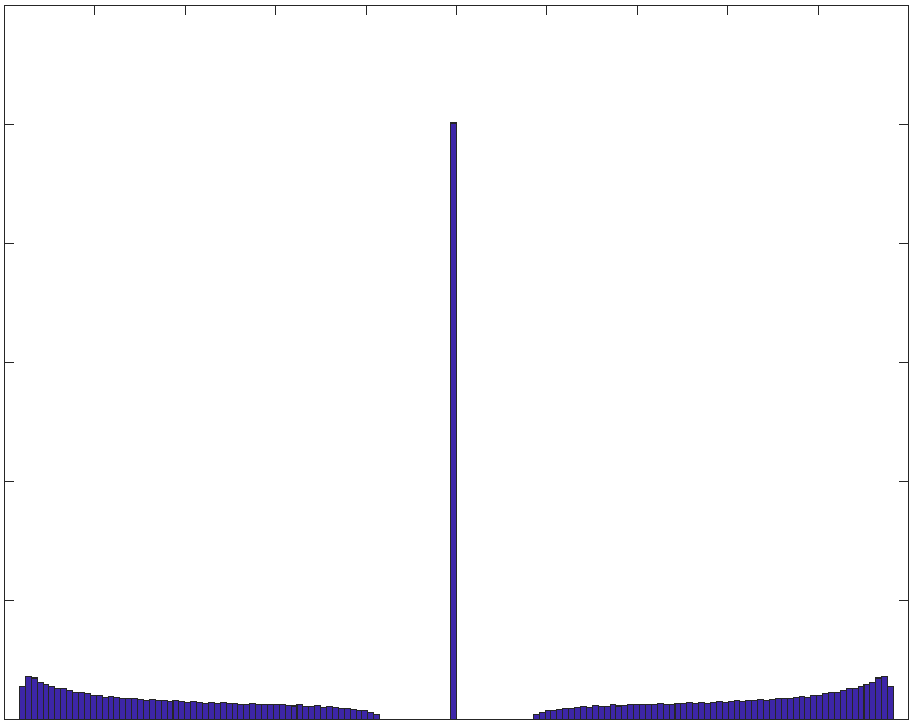} \\   $\mathcal{G}_5$ &   $\mu_{\mathcal{T}_5}$ 
            %\\ \includegraphics[scale=.63]{G_4.pdf} & \includegraphics[scale=.21]{Hist4.pdf} \\
        %  $\mathcal{G}_4$ &   $\mu_{\mathcal{T}_4}$   
        % \\ \includegraphics[scale=.71]{G_5.pdf} & \includegraphics[scale=.21]{Hist5.pdf} \\   $\mathcal{G}_5$ &   $\mu_{\mathcal{T}_5}$   
        % \\  \includegraphics[scale=.6]{G_6.pdf} & \includegraphics[scale=.21]{Hist6.pdf}
        % \\  $\mathcal{G}_6$ &   $\mu_{\mathcal{T}_6}$  
    \end{tabular} &
     \begin{tabular}{c c}
        %  \includegraphics[scale=.68]{G_1.pdf} & \includegraphics[scale=.21]{Hist1.pdf}\\
        %   $\mathcal{G}_1$ &   $\mu_{\mathcal{T}_1}$  \\ 
        %   \includegraphics[scale=.59]{G_2.pdf} &
        %  \includegraphics[scale=.21]{Hist2.pdf} \\  
        %   $\mathcal{G}_2$ &  $\mu_{\mathcal{T}_2}$  \\ 
        %  \includegraphics[scale=.63]{G_3prime.pdf} & \includegraphics[scale=.21]{Hist3.pdf}\\ 
        
        %      $\mathcal{G}_3$ &  $\mu_{\mathcal{T}_3}$    \\ 
        \includegraphics[scale=.63]{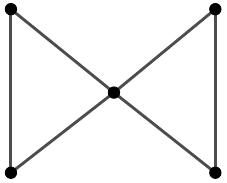} &
         \includegraphics[scale=.21]{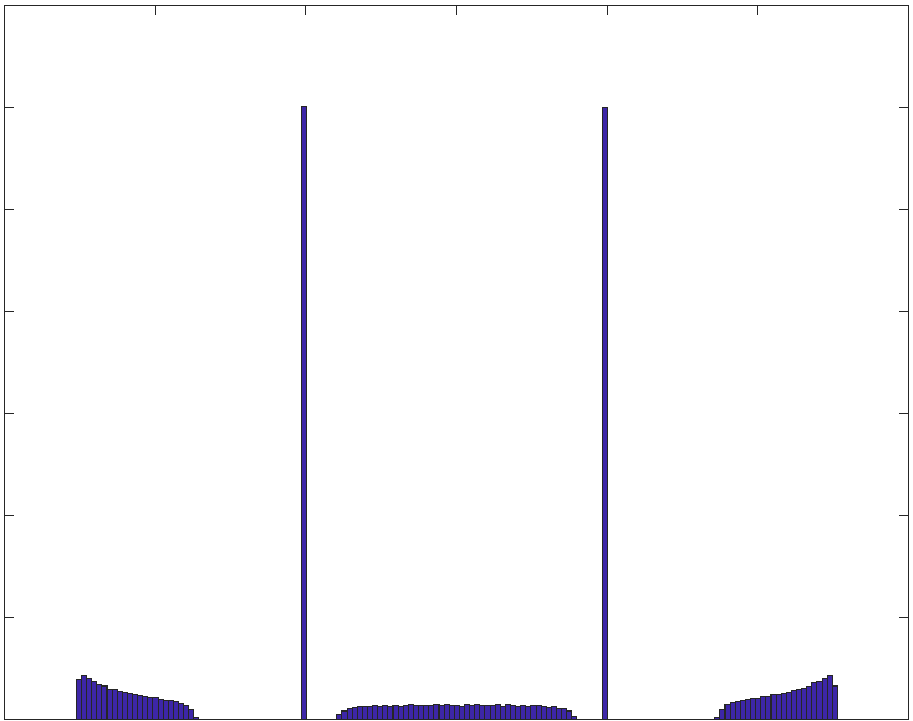} \\  
          $\mathcal{G}_2$ &  $\mu_{\mathcal{T}_2}$  
        \\ \includegraphics[scale=.6]{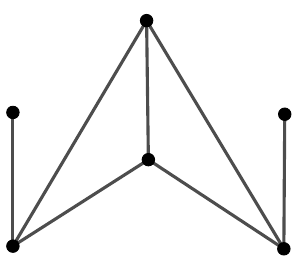} & \includegraphics[scale=.21]{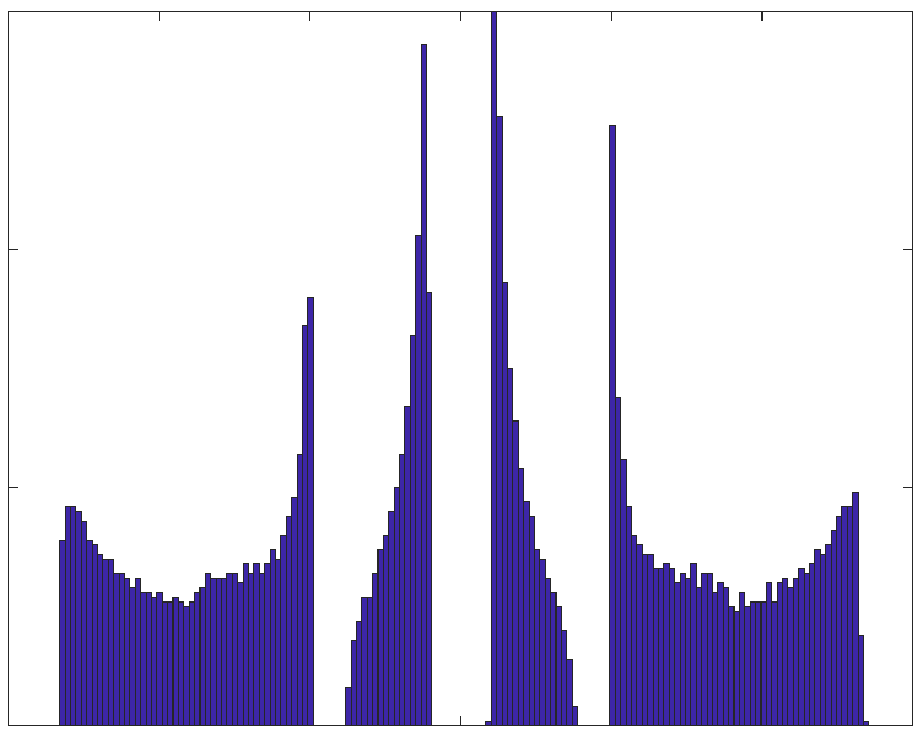} \\
         $\mathcal{G}_4$ &   $\mu_{\mathcal{T}_4}$    
        \\  \includegraphics[scale=.6]{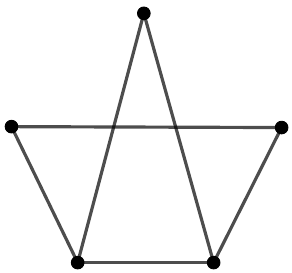} & \includegraphics[scale=.22]{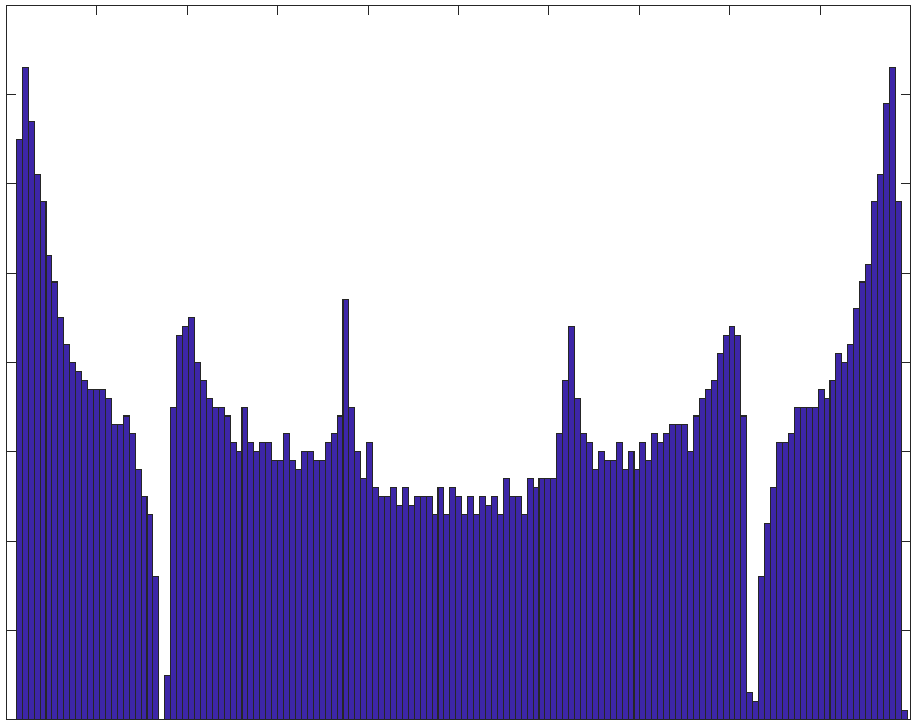}
        \\  $\mathcal{G}_6$ &   $\mu_{\mathcal{T}_6}$  
    \end{tabular}
    \end{tabular}
    \caption{Six graphs $\mathcal{G}_i$ alongside the respective approximation to $\mu_{\T_i}$, the density of states of the adjacency operator $A_{\T_i}$ (that is, $a\equiv 1$ and $b\equiv0$). The approximations presented are (scaled) histograms of the eigenvalues of a large random lift (taken as in Section \ref{subsecrandlifts}) of the respective graph. See Table \ref{tableradii} for the spectral radii of the $\T_i$.}
    \label{table1}
\end{table}

\begin{observation}
For simplicity let us consider only the adjacency operator (so $b_v = 0$ and $a_e = 1$ for all vertices $v$ and edges $e$.)  Table \ref{table1} shows that even if the base graphs are similar in some sense, the spectra of their universal cover may be quite different. In particular, we make the following remarks: 
\begin{enumerate}[i)]
    \item (Topological equivalence) $\mathcal{G}_1$ and $\mathcal{G}_2$ are homeomorphic and, in particular, they have the same fundamental group. Yet the spectra of the universal covers differ.
    \item (Eigenvalues of base graph) The graphs $\mathcal{G}_3$ and $\mathcal{G}_4$ are \emph{cospectral} (i.e. their non-weighted adjacency matrices have the same multiset of eigenvalues); however, the spectra of their universal covers possess very different features. 
    \item (Perturbations) $\mathcal{G}_3$ is obtained by adding a leaf to $\mathcal{G}_2$. In this case a gap in the spectrum around zero is created. From experiments, it seems adding leaves or edges can often cause major changes in the spectrum. 
    \item (Degrees) Graphs $\mathcal{G}_5$ and $\mathcal{G}_6$ have the same degree sequence. 
\end{enumerate}
\end{observation}

\subsection{Random Lifts}
\label{subsecrandlifts}

A random $d$-lift of a graph is a random $d$-fold cover of the graph with a particular choice of randomness. Recently, random $d$-lifts  have been studied in the context of expander graphs. For example, random 2-lifts were used by Marcus, Spielman and Srivastava to show the existence of Ramanujan graphs of every degree  \cite{marcus2013interlacing}; these techniques were later generalized by Hall,  Puder and Sawin  \cite{hall2018ramanujan} to $d$-lifts for $d \ge 2$.  

\begin{definition}[Random lift] \label{def:randlift}
Let $\mathcal{G}$ be a finite graph with $n$ vertices and let $d \ge 1$. Let $\calU(d)$ be the unitary group of dimension $d$ and consider a random function $U: E(\calG) \to \calU(d)$ that satisfies for all $\diredge$ that: \begin{enumerate}[i)]
    \item (Symmetry) $U_{\che} = U_\diredge^*$. 
    \item (Uniformly distributed) $U_\diredge$ is a random matrix distributed uniformly in the space of $d\times d$ permutation matrices. 
    \item (Independence) If $\diredge_1\neq \diredge_2 $ and  $\diredge_1\neq \che_2$, then the random matrices $U_{\diredge_1}$ and $U_{\diredge_2}$ are independent. 
\end{enumerate}
Then, a random $d$-lift of $\G$ is a random graph with $d n$ vertices, whose adjacency matrix is given by 
\[ \sum_{\diredge \in E(\G)} \Delta_{\sigma(\diredge) \tau(\diredge)} \otimes U_{\diredge} \]
where $\edge_{\sigma(\diredge)\tau(\diredge)}$ denotes the $n\times n$ matrix with a 1 in the $(\sigma(\diredge), \tau(\diredge))$ entry and 0 everywhere else. Moreover, if $\calG$ has edge weights $a$ and vertex potential $b$, then the Jacobi matrix $A_\calG$ can be pulled back to the (random) Jacobi  matrix $A_{d, \calG}$ on the lift, given by 
$$A_{d, \calG} :=\sum_{u\in V(\calG)} b_u \Delta_{uu} \otimes I_d+ \sum_{\diredge\in E(\calG)} a_\diredge \Delta_{\sigma(\diredge)\tau(\diredge)} \otimes U_\diredge.$$
\end{definition}

% \begin{definition}[Random lift]
% \label{def:randlift}
% Let $\mathcal{G}$ be a graph with $n$ vertices, $m$ edges, no loops and weights $p: E(\mathcal{G}) \to \mathbb{R}^+$. Let $\gamma: E(\mathcal{G}) \to [m]$ be a numbering of the edges (i.e. $\gamma$ is a bijection) and  let $U_1, \dots, U_m$ be random independent uniform $d\times d$ permutation matrices. A weighted random $d$-lift of $\mathcal{G}$ is a random weighted graph whose adjacency matrix is given by
% $$\sum_{\{i, j\}\in E(\mathcal{G})} p\{i, j\} ( \edge_{ij}\otimes U_{\gamma\{i, j\}} + \edge_{ji} \otimes U^*_{\gamma\{i, j\}})$$
% where $\edge_{ij}$ denotes the $n\times n$ matrix with a 1 in the $(i, j)$ entry and 0 everywhere else. \marginpar{$\edge_{ij}$}
% \end{definition}

It  is well known (and easy to show)  that as $d$ goes to infinity, random $d$-lifts of a fixed graph $\mathcal{G}$ converge in the Benjamini-Schramm sense \cite{benjamini2011recurrence} to the universal covering graph $\mathcal{T}$ of $\mathcal{G}$\footnote{The main idea is that if $\alpha$ is a random permutation of $[d]$, then the size of the orbit of any element on which $\alpha$ acts almost surely goes to infinity as $d\to \infty$. This is enough to show that for every fixed $p\geq 1$ the $p$-neighborhood of any given vertex of the random lift is almost surely tree-like as $d\to \infty$. A quantitative version of this argument can then be used to show that this holds for many vertices at the same time. }. In particular, this implies that the density of states of $A_\mathcal{T}$ is the weak limit of the mean eigenvalue distribution of the random $d$-lifts: 

\begin{lemma}[Limits of random lifts] 
\label{lem:graphconvergence}
Using the above notation, for every fixed $p$ one has 
$$\lim_{d\to \infty} \frac{1}{d} \mathbb{E}\big [\mathrm{Tr}\, A_{d, \calG}^p\big] = \int_\mathbb{R} x^p d\dos(x). $$
\end{lemma}

%\begin{proof}
%The intuition is that random lifts contain few short cycles.  For each vertex $v$ of $\G_d$, fix an arbitrary vertex $\overline{v}$ of $\mathcal{T}$ in the fiber over $v$, and let $\overline{V} = \{\overline{v} : v \in V(\G_d)\}.$  We must show that the mean number of closed walks of length $k$ in $\G_d$ based at a uniform random vertex of $\G_d$ converges to the the number of closed walks of length $k$ in $\mathcal{T}$ based at a uniform random element of $\overline{V}$.  The only discrepancy arises from those closed walks of length $k$ in $\G_d$ that do not lift to a closed walk of length $k$ in $\mathcal{T}$. 

%From the proof of Corollary 8 in \cite{bordenave2019eigenvalues}, in particular it follows that the expected number of vertices of $\G_d$ whose neighborhood of radius $k$ contains a cycle is at most $O(|V(\G)|(2|E(\G)|)^{3k})$, which for $k$ fixed is $O(1)$ as $d \to \infty$.  It follows that a $1 - o(1)$ (in expectation) fraction of vertices of $\G_d$ are not the basepoint of any cycle of length $k$, so the number of closed walks of length $k$ based any such $v$ is the same as for $\overline{v}$ in $\mathcal{T}$.  Each remaining vertex in the $o(1)$ fraction is the basepoint of at most $O(1)$ many closed walks of length $k$, which concludes the proof.
%\end{proof}

Recently, using tools from free probability, Bordenave and Collins viewed the limiting operator $A_\calT$ as an element of a certain $C^*$ algebra and showed that the convergence above holds in a much stronger sense, implying convergence of the edges of the spectrum \cite{bordenave2019eigenvalues}. In Section \ref{sec:sunadastheorem} we revisit in detail this $C^*$-algebra representation and use it give a proof of Sunada's theorem.  The theorem of Bordenave and Collins also handles half-loops $\diredge$, for which $U_\diredge$ is defined to be a random matching (see \cite{friedman2019relativized} for a discussion of related models of random lifts), we also revisit this idea in Section \ref{sec:halfloops} and use it obtain an extension of Sunada's theorem in the presence of half-loops. 
\subsection{Spectral Radii of Operators on Universal Covers}
\label{sec:spectralradii}

In this discussion we fix a finite graph $\mathcal{G}$ and let $A_\G$ denote its adjacency matrix (that is $a\equiv 1$ and $b\equiv 0$), and we denote the universal cover of $\G$ by $\mathcal{T}$.  We denote the spectral radius of $A_\mathcal{G}$ and $A_\T$ by $\spr(\mathcal{G})$ and $\spr(\T)$ respectively. \marginpar{$\spr(\mathcal{G})$}  

The  quantity $\spr(\mathcal{T})$ is fundamental in the theory of Ramanujan graphs \cite{lubotzky1988}. Indeed, in general, a graph $\mathcal{G}$ is said to be Ramanujan if $\Spec(\mathcal{G})$ is contained in $[-\spr(\mathcal{T}), \spr(\mathcal{T})]\cup \{-\spr(\mathcal{G}), \spr(\mathcal{G})\}$ \cite{greenberg1995spectrum}; and many of the results in the area are  stated in terms of $\spr(\mathcal{T})$.  For example,  the main result in \cite{marcus2013interlacing} states that any graph $\mathcal{G}$ has a 2-lift $\mathcal{G}'$ such that the new eigenvalues of $A_{\mathcal{G}'}$ are bounded above by $\spr(\mathcal{T})$, while its generalization in \cite{hall2018ramanujan} shows that the same is true for $n$-lifts for every $n\geq 2$. Similar techniques have been used in \cite{mohanty2019x} to obtain analogous results for quotients of a class of infinite graphs that goes beyond universal covers, where the results are also given in terms of the spectral radius of the given infinite graph. We refer the reader to \cite{sy1992discrete, huang2019local} for discussions on the relation  between  $\spr(\G)$ and $\spr(\T)$.

In most cases $\spr(\mathcal{T})$  is mentioned  only as an implicit quantity and no quantitative bounds on it are provided. When $\mathcal{G}$ is $d$-regular, by the work of Kesten \cite{kesten1959symmetric} we know that $\spr(\mathcal{T}) = 2\sqrt{d-1}$. If $\mathcal{G}$ is a $(c, d)$-biregular bipartite graph, by Godsil \cite{godsil1988walk} we know that $\spr(\mathcal{G}) =\sqrt{c-1}+\sqrt{d-1}$. In the case of general graphs, Hoory \cite{hoory2005lower} proved that if $d_{\mathrm{avg}}(\mathcal{G})$ is the average degree of $\mathcal{G}$ then 
\begin{equation}
\label{eq:rholower}
 \spr(\mathcal{T}) \geq 2\sqrt{d_{\mathrm{avg}}(\mathcal{G})-1}.
\end{equation}
On the other hand, an upper bound can be obtained trivially by noting that if $d_{\max}(\mathcal{G})$ is the maximum degree of $\mathcal{G}$, then $\mathcal{T}$ can be embedded in the infinite $d_{\max}(\mathcal{G})$-regular tree and hence 
\begin{equation}
\label{eq:rhoupper}
2\sqrt{d_{\max}(\mathcal{G})-1} \geq \spr(\mathcal{T}).
\end{equation}
Therefore, if $\mathcal{G}$ is a regular graph we have $d_{\mathrm{avg}}(\mathcal{G}) = d_{\max}(\mathcal{G})$, so by putting together  (\ref{eq:rholower}) and (\ref{eq:rhoupper}) the formula of Kesten is recovered. 

It is hard to find any explicit formula for $\spr(\mathcal{T})$ that only depends on the adjacencies in $\mathcal{G}$ and not, for example, on the paths in $\mathcal{G}$ or the powers of $A_\mathcal{G}$. This is due in part to the fact that two similar base graphs may have universal covering trees with fairly different spectral radii. In the table below we used the system of equations in Corollary \ref{cor:lagrange} to compute the spectral radii of the respective universal covering trees in Table \ref{table1}. 
\begin{table}[h] 
    \centering
    \begin{tabular}{lr} \toprule
     Base graph $\G$ & $\spr(\T(\G))$ \\  \midrule
         $\G_1$ & $ \approx 3.0368$ \\
          $\G_2$ & $\sqrt{ \frac{1}{2}(7 + \sqrt{33})} \approx  2.5243$ \\
          $\G_3$ & $\sqrt{\frac{1}{2}\left(6+\sqrt{5}+\sqrt{27+6\sqrt{5}}\right)}\approx 2.7012$ \\
          $\G_4$ & $\approx 2.6589$ \\
          $\G_5$ & $1+\sqrt{2} \approx 2.4142$ \\
          $\G_6$ & $\approx 2.4461$ \\
          \bottomrule
    \end{tabular}
% \begin{tabular}{|c |c| c| c| c| c| c|}
% \hline 
%  & $\mathcal{G}_2$ & $\mathcal{G}_3$ & $\mathcal{G}_4$  & $\mathcal{G}_6$ \\ \hline  $\spr(\mathcal{T}(\mathcal{G}_i))$ & $\sqrt{ \frac{1}{2}(7 + \sqrt{33})} \approx $ 2.5243 &  $\sqrt{\frac{1}{2}\left(6+\sqrt{5}+\sqrt{27+6\sqrt{5}}\right)}\approx 2.7012$ & $\approx 2.65893$ &  $\approx 2.44611$ \\ \hline
% \end{tabular}
\caption{Using Mathematica, the system of equations in Corollary \ref{cor:lagrange} was solved for the graphs in Table \ref{table1}.  In some cases explicit solutions in radicals were output. Previous results on the %regular and
biregular bipartite case mentioned above imply the result for $\G_5$.  }
\label{tableradii}
\end{table}

\subsection{Graph Products and Non-commutative Probability}
\label{sec:graph-products}
In recent years different problems in spectral graph theory have been approached from the perspective of non-commutative probability theory; we refer the reader to the book of Hora and Obata \cite{hora2007quantum} for a unified exposition. Of particular interest for the present work are a sequence of results \cite{obata2004quantum, accardi2004monotone, accardi2007decompositions} which establish a correspondence between different combinatorial graph products and different notions of stochastic independence in non-commutative probability (see Example \ref{excartesianprod} below).  We summarize this correspondence in the  dictionary presented in the table below. As mentioned before, part of the motivation of this project is to extend this dictionary to include the notion of freeness with amalgamation\footnote{Graph products related to the notion of amalgamation in group theory, such as the one appearing in \cite{mohar2006tree}, have appeared in the past. However, these unrelated products were introduced with the purpose of studying symmetries in graphs and it is not clear if they relate to spectral theory.   }. 
\begin{table}[h] 
\centering
\begin{tabular}{ ll } 
\toprule
 Graph product & Notion of independence \\
\midrule
 Cartesian & Classical (tensor) \\ 
 Free & Free \\ 
 Comb & Monotone \\ 
 Star & Boolean \\ 
\bottomrule
\end{tabular}
\caption[Table 1]{On the left, graph products are mentioned (see \cite{accardi2004monotone} for  definitions). On the right, the corresponding notions of stochastic independence are given (see \cite{speicher1993boolean, muraki2001monotonic} for definitions of Boolean and Monotone independence, respectively.)  }
\label{table}
\end{table}

\begin{example}[Cartesian product]
\label{excartesianprod}
Given two graphs $\mathcal{G}_1$ and $\mathcal{G}_2$, one may form their \emph{Cartesian product} $\mathcal{G}_1\, \Box\, \mathcal{G}_2$, with vertex set and edge set as follows:
\[ V( \G_1 \, \Box \, \G_2) = V(\G_1) \times V(\G_2) \]
\[ E( \G_1 \, \Box \, \G_2) = \{\{(v, w_1), (v, w_2)\} : \{w_1, w_2\} \in E(\G_2)\} \cup \{\{(v_1, w), (v_2, w)\} : \{v_1, v_2\} \in E(\G_1)\}.\]
In other words, we have the relation of adjacency matrices $A_{\G_1 \, \Box \, \G_2} = A_{\G_1} \otimes I + I \otimes A_{\G_2}$.  Using this representation, one sees that the spectrum of $\mathcal{G}_1\, \Box\, \mathcal{G}_2$ is the Minkowski sum $\{\lambda_1 + \lambda_2 : \lambda_1 \in \Spec(\G_1), \lambda_2 \in \Spec(\G_2)\}$. In the language of spectral measures, this says that the density of states $\mu$ of $A_{\G_1\,\Box\,\G_2}$ is the (classical) convolution of the density of states for the constituent graphs; that is, $\mu$ is the law of the sum of two independent random variables distributed according to the density of states of $A_{\G_1}$ and of $A_{\G_2}$, respectively. %\footnote{To be precise here in the definition of spectral measure, we must specify that the trace functional $\tau_n$ we use on each space of $n \times n$ matrices is $\frac{1}{n} \Tr(\cdot)$, where $\Tr$ denotes the ordinary matrix trace.}
\end{example}
In this work,  the free graph product  is of particular interest. The connections of this product to free probability were developed by Accardi, Lenczewski and Sa\l{}apata \cite{accardi2007decompositions}. The authors pointed out that Voiculescu's notion of free independence introduced in \cite{voiculescu1985symmetries} could be used when analysing the spectrum of free products of graphs. As an example provided in their paper, one can recover the spectral measure of the $d$-regular infinite tree, as the \emph{free convolution} of $d$ signed Bernoulli distributions (also known as Rademacher distributions).  

However, the roots of the theory of free graph products go back to Kesten \cite{kesten1959symmetric}, who studied the spectra of random walks on free groups.  We survey this area in the next subsection.

Finally, we draw attention to the \emph{additive graph product} recently defined by Mohanty and O'Donnell \cite{mohanty2019x}, who showed that $X$-\emph{Ramanujan} graphs (a generalized notion of Ramanujan graph) can always be obtained by taking certain quotients of any graph constructed via their product. In the aforementioned work the features of the spectrum of the resulting infinite graphs are left as implicit quantities. Hence, a natural complementary line of research would then be that of understanding the spectrum of the infinite graphs arising from the additive graph product.      

\subsection{Random Walks on Groups and Free Probability}

The '80s and '90s saw a flurry of activity on the spectra of random walks on Cayley graphs of free products of groups; see \cite{woess2000random} for a detailed exposition.  The analytic formula relating the spectral measure of the product graph to the spectral measures of its factors  (essentially, Voiculescu's $R$-transform) was discovered independently by McLaughlin \cite{mclaughlin1988random}, Soardi \cite{soardi1986resolvent}, and Woess \cite{woess1986nearest},  but it was Voiculescu's work \cite{voiculescu1985symmetries} that put it into the more general context of non-commutative probability.

Interestingly,  also in \cite{voiculescu1985symmetries}, Voiculescu  laid out a more general theory of \emph{freeness with amalgamation}, which extends the scalars $\mathbb{C}$ to an arbitrary unital algebra $\mathcal{B}$.  On the other hand, in a parallel way, some progress was also made by the graph theory community in the study of random walks on amalgams \cite{picardello1985random}. In particular Cartwright and Soardi \cite{cartwright1986random} developed the combinatorial tools to obtain the Green function of the Cayley graph of an amalgamated free product of finite groups, in the particular case in which the subgroup over which amalgamation is performed is a normal subgroup of the groups in the product. In Section \ref{subsecfreenesswithamalg} we explain why the tools of free probability allow one to approach this problem even if the normality assumption is dropped.

\section{Preliminaries on Free Probability}
\label{secpreliminaries}

In this section we describe the tools from the theory of free probability that will be used throughout this preprint. A basic background on $C^*$-algebras is recommended, but not necessary for all of the following discussion. We refer the reader to \cite{davidson1996c} for an introduction to $C^*$-algebras.

\subsection{Free Probability}
\label{subsecfreeprobability}
 Free probability was introduced by Dan-Virgil Voiculescu in his seminal papers \cite{voiculescu1985symmetries, voiculescu1986addition}. We refer the reader to \cite{voiculescu1992free, nica2006lectures, mingo2017free} for a detailed introduction. This theory is developed in the context of non-commutative probability, in which  random variables are viewed as elements of a non-commutative algebra and the notion of expectation from classical probability is substituted by a linear functional on the algebra. 
 
 \begin{definition}[Non-commutative probability space]
A non-commutative probability space is a pair $(\mathcal{A}, \func)$ where $\mathcal{A}$ is a unital $\mathbb{C}$-algebra and $\func: \mathcal{A} \to \mathbb{C}$ is a unital linear map. 
\end{definition}

In this work the following examples of non-commutative probability spaces are of primary importance. 

\begin{example}
 The pair $\big(M_n(\bC), \frac{1}{n} \Tr(\cdot) \big)$ is a non-commutative probability space.  
\end{example}

\begin{example}
Let $G$ be a discrete group and denote the reduced $C^*$-algebra of $G$ by $C_{\mathrm{red}}^*(G)$\marginpar{$C_{\mathrm{red}}^*(\cdot)$}.\footnote{In other words, $C_{\mathrm{red}}^*(G)$ is the norm closure in $B(\ell^2(G))$ of the left regular representation of $G$ on $\ell^2(G)$. See \cite[\S 7]{davidson1996c}.  } Then the $C^*$-algebra $C_{\mathrm{red}}^*(G)$ has a canonical state given by 
\begin{equation}
\label{eqcanonicaltracegroup}
\func(x) \myeq \langle  \delta_e, x \delta_e \rangle, \quad \forall x \in C_{\mathrm{red}}^*(G),
\end{equation}
where $\delta_e \in \ell^2(G)$ denotes the indicator of the identity element $e\in G$. 
\end{example}

As Voiculescu showed, in non-commutative probability there is not a unique notion of stochastic independence. 

\begin{definition}[Free independence] \label{def:freeindependence}
Let $(\mathcal{A},\func )$ be a non-commutative probability space, and let $\{\mathcal{A}_i\}_{i\in I}$ be a family of unital subalgebras of $\mathcal{A}$.  We say that the algebras $\mathcal{A}_i$ are  \emph{freely independent} (or just \emph{free}) if $\func(a_1 \cdots a_m) =0$ whenever $a_i \in \mathcal{A}_{j_i}$, $j_1\neq j_2\neq \cdots \neq j_m$\footnote{Here and throughout, we use this shorthand to mean $j_i \ne j_{i+1}$ for all $1 \le i < m$.} and $\func(a_i)=0$. Sets of random variables are said to be freely independent if the algebras they generate are free. 
\end{definition}

In the proof of Theorem \ref{thm:mass} we will use the following. 

\begin{proposition}[Voiculescu]
\label{propfreegenerators}
Let $g_1, \dots, g_m $ be the $m$ canonical generators of $\mathbb{F}_m$\marginpar{$\mathbb{F}_m$}\footnote{Here, $\mathbb{F}_m$ denotes the free group on $m$ generators. } and $\lambda$ be the left regular representation of $\mathbb{F}_m$ on $\ell^2(\mathbb{F}_m)$. Then, the random variables $\lambda(g_1), \dots, \lambda(g_m)$ are free in the non-commutative probability space $(C_{\mathrm{red}}^*(\mathbb{F}_m), \func)$, where $\func$ is as in (\ref{eqcanonicaltracegroup}). 
\end{proposition}

Note that in Proposition \ref{propfreegenerators} the random variables $\lambda(g_i)$ are unitaries and satisfy that 
\begin{equation}
\label{eqhaarunitary}
\func(\lambda(g_i)^k) =0 \quad \forall k \in \mathbb{Z}\setminus \{0\}. 
\end{equation}
In non-commutative probability, random variables satisfying (\ref{eqhaarunitary}) are called \emph{Haar unitaries. }

\subsection{Operator-valued Probability Spaces}

We will require the following generalization of the notion of non-commutative probability space.

\begin{definition}[Operator-valued probability space] A triple $(\mathcal{A}, E, \mathcal{B})$ is called an operator-valued probability space if $\mathcal{A}$ is a unital algebra, $\mathcal{B} \subset \mathcal{A}$ is a subalgebra of $\mathcal{A}$ with $1_\mathcal{A}\in \mathcal{B}$,  and $E: \mathcal{A} \to \mathcal{B}$ is a conditional expectation, i.e. $E$ is a linear map satisfying 
\begin{enumerate}[i)]
    \item $E \restriction_{\mathcal{B}} = \mathrm{Id}_\mathcal{B}$ , where $E \restriction_{\mathcal{B}}$ denotes the restriction of $E$ to the subdomain $\mathcal{B} \subset \mathcal{A}$. 
    \item $E[b_1a b_2] = b_1 E[a] b_2$ for every $a \in \mathcal{A}$ and $b_1, b_2 \in \mathcal{B}$.
\end{enumerate}
Often, to emphasize the role of $\mathcal{B}$ we will say that $(\mathcal{A}, E, \mathcal{B})$ is a $\mathcal{B}$-valued probability space. 
\end{definition}

In the present we are interested in the following  situations. 

\begin{example}[Matrices with entries in the algebra]
\label{exmatrixopval}
Let $(\mathcal{A}, \func)$ be a non-commutative probability space. Then, the algebra $M_n(\mathbb{C})\otimes \mathcal{A}$ (i.e. $n\times n$ matrices with entries in $\mathcal{A}$) has a canonical copy of $M_n(\mathbb{C})$ as a subalgebra, namely $M_n(\mathbb{C})\otimes 1_\mathcal{A}$. Hence, $M_n(\mathbb{C})\otimes \mathcal{A}$  can be viewed as an  $M_n(\mathbb{C})$-valued probability space with the conditional expectation $ \mathrm{Id}_{M_n(\mathbb{C})} \otimes \func $. In other words, if $X \in M_n(\mathbb{C})\otimes \mathcal{A}$ is given by $X= (x_{ij})_{i, j=1}^n$ then $E$ defined by
\begin{equation}
\label{eqconditionalexp}
E(X) \myeq (\func(x_{ij}))_{i,j=1}^n \in M_n(\mathbb{C}), 
\end{equation}
is a $M_n(\mathbb{C})$-valued conditional expectation. 

\begin{example}[Group $C^*$-algebras]
\label{exreducedopvalprobaspace}
 Let $H$ be a subgroup of $G$, and denote by $\mathcal{B}$ the canonical copy of $C_{\mathrm{red}}^*(H)$ inside $C_{\mathrm{red}}^*(G)$. Then, we can view $C_{\mathrm{red}}^*(G)$ as a $\mathcal{B}$-valued probability space, by considering the canonical conditional expectation  $E: C_{\mathrm{red}}^*(G) \to \mathcal{B}$, namely  the projection of $C_{\mathrm{red}}^*(G)$ onto $\mathcal{B}$ that is orthogonal with respect to the GNS inner product induced by the canonical trace of $C^*_{\mathrm{red}}(G)$.   
\end{example}

\end{example}

\subsection{Freeness with Amalgamation}
\label{subsecfreenesswithamalg}

One of the main technical ingredients of this paper is Voiculescu's notion of freeness with amalagamation \cite[\S 5]{voiculescu1985symmetries}. See  \cite[\S 3.8]{voiculescu1992free} and \cite[\S 9]{mingo2017free} for an introduction to the subject, and  \cite{jekel2018operator} for a detailed development.

\begin{definition}[Freeness with amalgamation] Consider an operator-valued probability space $(\mathcal{A}, E, \mathcal{B})$. Let $\{\mathcal{A}_i\}_{i\in I}$ be a family of subalgebras of $\mathcal{A}$ such that $\mathcal{B}\subset \mathcal{A}_i$ for every $i\in I$. We say that the algebras $\mathcal{A}_i$ are free with amalgamation over $\mathcal{B}$ if $E(a_1 \cdots a_m) =0$ whenever $a_i \in \mathcal{A}_{j_i}$, $j_1\neq j_2\neq \cdots \neq j_m$ and $E(a_i)=0$. 

Random variables in $\calA$ are said to be free with amalgamation over $\calB$ if the  subalgebras generated by them are so. 
\end{definition}

Observe that in the particular case of $\mathcal{B} = \mathbb{C} 1_\calA$, we recover the usual definition of free independence. In this work we exploit the following relation between freeness and freeness with amalgamation. 

\begin{observation}
\label{obs:matrixfreeness}
Let $(\mathcal{A}, \func)$ be a non-commutative probability space and  $\mathcal{A}_1, \mathcal{A}_2 \subset \mathcal{A}$ be  freely independent subalgebras. Let $X, Y \in M_n(\mathbb{C})\otimes \mathcal{A}$ with $X=(x_{ij})_{i,j=1}^n$ and $Y=(y_{ij})_{i,j=1}^n$. If $x_{ij}\in \mathcal{A}_1$ and $y_{ij}\in \mathcal{A}_2$ for every $i, j=1, \dots, n$, then $X$ and $Y$ are free with amalgamation over $M_n(\mathbb{C})$ with respect to the conditional expectation defined in (\ref{eqconditionalexp}). 
\end{observation}

The connection with the amalgamated free product of groups is the following. 

\begin{proposition}[Voiculescu]
\label{propCalgebrawithamalgamation}
Let $G_1, G_2$ be discrete groups with a common subgroup $H$. Let $G$ be the group free product with amalgamation of $G_1, G_2$ over $H$, i.e. $ G \myeq G_1 \ast_H G_2$. Consider the inclusions $\rho_0 : C_{\mathrm{red}}^*(H) \to C_{\mathrm{red}}^*(G)$ and $\rho_i : C_{\mathrm{red}}^*(G_i)\to C_{\mathrm{red}}^*(G)$ for $i=1, 2$. Now,  as in Example \ref{exreducedopvalprobaspace}  put $\mathcal{B} \myeq \rho_0( C_{\mathrm{red}}^*(H))$ and view $C_{red}^*(G)$ as a $\mathcal{B}$-valued probability space. Then $\rho_1( C_{\mathrm{red}}^*(G_1))$ and $\rho_2( C_{\mathrm{red}}^*(G_2))$ are free with amalgamation over $\mathcal{B}$. 
\end{proposition}

We are now ready to explain in precise terms the way in which free probability appears in the study of the spectra of Cayley graphs. Consider a finitely generated group $G$ and let $S\subset G$ be a finite generating set. Denote by $\Gamma(G,S)$ the left Cayley graph of $G$ with respect to $S$. When it is clear from the context, we will simply write $\Gamma$.  Let $\lambda$ be the left regular representation of $G$, and note that the operator \marginpar{$T_\Gamma$}
$$T_{\Gamma} \myeq \sum_{s \in S} \lambda(s) \in C_{\mathrm{red}}^*(G)$$
is the adjacency operator of $\Gamma$. Indeed, if $G$ is identified with the vertices of $\Gamma$, then for every $g\in G$ 
$$T_{\Gamma}(\delta_g) = \sum_{ s \in S} \delta_{sg}. $$
Moreover, if $S$ is symmetric, meaning $S = S^{-1}$, then clearly $\Gamma$ is undirected and $T_\Gamma$ is self-adjoint. 

\begin{observation}
\label{obs:cayleydecomp}
Let $G_1, G_2$ be two finitely generated groups with a common subgroup $H$. Let $S_1 \subset G_1$ and $S_2 \subset G_2$ be respective finite generating sets. For $j=1, 2$, let $\iota_j : G_j \to G_1 \ast_H G_2$ be the canonical inclusion and let $S = \iota_1(S_1)\cup \iota_2(S_2)$. Let $\lambda$ be the left regular representation of $G_1 \ast_H G_2$ on $\ell^2(G_1\ast_H G_2)$. Let $\Gamma \myeq \Gamma (G_1\ast_H G_2, S ), \Gamma_1 \myeq \Gamma(G_1, S_1)$ and $\Gamma_2 \myeq \Gamma (G_2, S_2)$. Then we have the following decomposition
$$T_\Gamma  = \sum_{s \in S} \lambda(s) = \sum_{s\in \iota_1(S_1)} \lambda(s) + \sum_{r \in \iota_2(S_2)} \lambda(r) = \widetilde{T_{\Gamma_1}}+ \widetilde{T_{\Gamma_2}},$$
where $\widetilde{T_{\Gamma_j}}$ denotes the natural inclusion of $T_{\Gamma_j}\in C_{\mathrm{red}}^*(G_j)$ into $C_{\mathrm{red}}^*(\ell^2(G_1\ast_H G_2))$. 
From Proposition \ref{propCalgebrawithamalgamation} we know that $\widetilde{T_{\Gamma_1}}$ and $\widetilde{T_{\Gamma_2}}$ are free with amalgamation over $C_{red}^*(H)$ with respect to the conditional expectation $E: C_{\mathrm{red}}^*(G_1\ast_H G_2) \to C_{\mathrm{red}}^*(H)$ defined as in Example \ref{exreducedopvalprobaspace}. 
\end{observation}

\subsection{The Amalgamated Free Product of Hilbert $\mathcal{B}$-modules}
\label{subsecHilbertProd}

 In this section we will assume that $\mathcal{B}$ is a unital $C^*$-algebra and we will focus on  $\mathcal{B}$-valued probability spaces given by operators on Hilbert modules. The theory of Hilbert modules is delicate and we refer the reader to \cite{jekel2018operator} for a comprehensive treatment. Below, we content ourselves with  providing a quick summary of the objects considered here, referring the reader to specific places in \cite{jekel2018operator} for details. 

A right Hilbert $\calB$-module, say $\calH$, is a right $\calB$-module with a $\calB$-valued inner product 
$$\langle\cdot , \cdot  \rangle: \calH \times \calH \to \calB,$$ 
for which $\calH$ is complete with respect to a  norm determined by $\langle \cdot , \cdot \rangle$ in a suitable way (see  \cite[\S 1.2]{jekel2018operator} for definitions and basic properties of these objects). Then, a linear map $T : \calH \to \calH$ (when viewing $\calH$ as a vector space) is said to be right-$\calB$-linear if $T(\eta b)=T(\eta)b$ for all $b\in \calB$ and $\eta\in \calH$. The $\calB$-valued inner product on $\calH$ allows one to define the adjoint for right-$\calB$-linear operators (which may not always exist), and the norm induced by the inner product leads to a natural definition of bounded right-$\calB$-linear operators (see Section 1.2, Definitions 1.2.8 and 1.2.9 in \cite{jekel2018operator}).  

In what follows, for $\calH$ a right Hilbert $\calB$-module, we will denote the $\ast$-algebra of bounded, adjointable (i.e. operators for which the adjoint exists), right $\mathcal{B}$-linear operators on $\mathcal{H}$ by $\widetilde{B}(\mathcal{H})$ \marginpar{$\widetilde{B}(\mathcal{H})$}. Then, a Hilbert $\calB$-bimodule $\calH$ is defined as a right Hilbert $\calB$-module together with a $\ast$-homomorphism $\pi: \calB \to \widetilde{B}(\calH)$. Intuitively $\pi$ provides a left action of $\calB$ on $\calH$ (complementing the existing right action provided by the right $\calB$-module structure), so for $b\in \calB$ and $\eta\in \calH$ we will use $b \eta$ as a shorthand notation for $\pi(b) \eta$.

In Section \ref{secgraphproduct} we will work with the following type of operator-valued probability space,  introduced by Voiculescu in \cite[\S 5.1]{voiculescu1985symmetries}.

\begin{example}[$\calB$-valued conditional expectation on $\widetilde{B}(\calH)$] 
\label{exhilbmods}
Assume that $\mathcal{B}$ is a unital $C^*$-algebra and that $\mathcal{H}$ is a  Hilbert $\mathcal{B}$-bimodule. Furthermore, assume that $\xi \in \mathcal{H}$ is a unit central vector, that is $\langle \xi, \xi \rangle_\mathcal{H} = 1_\mathcal{B}$ and $b\xi = \xi b$ for all $b\in \B$, and consider the decomposition $\mathcal{H} = \xi \mathcal{B} \oplus \overset{\circ}{\mathcal{H}}$, where $\overset{\circ}{\mathcal{H}}$ denotes the orthogonal complement of $\xi \mathcal{B}$ with respect to the $\B$-valued inner product of $\mathcal{H}$ (this is well defined by Lemma 4.3.2 in \cite{jekel2018operator}). It is easy to see that because $\xi$ is a unit central vector, $\pi: \calB \to \widetilde{B}(\calH)$ induces a  $\calB$-bimodule isomorphism between $\calB$ and  $\xi \calB$, and hence $\calB$ can be viewed as a subalgebra of $\widetilde{B}(\calH)$ (see Lemma 4.3.2 in \cite{jekel2018operator}), moreover, the map $E: \widetilde{B}(\calH) \to \calB$ given by
$$E[X] \myeq \langle \xi, X \xi \rangle, \quad \forall X \in \widetilde{B}(\calH),$$
is a $\calB$-valued conditional expectation (see Lemma 1.5.4 in \cite{jekel2018operator}) and hence  $\big(\widetilde{B}(\calH), E, \calB\big)$ is an operator-valued probability space. 
\end{example}

Now consider a family of Hilbert  $\mathcal{B}$-bimodules $\{\mathcal{H}_i\}_{i\in I}$, and for every $i\in I$ suppose that $\xi_i \in \mathcal{H}_i$ is a unit central vector. Then, as in Example \ref{exhilbmods},  consider the decomposition $\mathcal{H}_i = \xi_i \mathcal{B} \oplus \overset{\circ}{\mathcal{H}_i}$ and define the $\calB$-valued conditional expectation $E_i(X) \myeq \langle \xi_i, X\xi_i \rangle_{\mathcal{H}_i}$, so that we can view each $\widetilde{B}(\mathcal{H}_i)$ as a $\mathcal{B}$-valued probability space. We will now explain how to construct a bigger $\mathcal{B}$-valued probability space inside which we can find copies of each of the $\widetilde{B}(\mathcal{H}_i)$ that are free with amalgamation over $\mathcal{B}$. We refer the reader to \cite[\S 5]{voiculescu1985symmetries} for the original source and \cite[\S 4]{jekel2018operator} for an extended exposition. 

Define the amalgamated free product over $\mathcal{B}$ of the $\mathcal{H}_i$ as 
 $$(\mathcal{H}, \xi) \myeq\underset{ i \in I}{\ast_{\mathcal{B}}} (\mathcal{H}_i, \xi_i) \myeq \xi\mathcal{B} \oplus \bigoplus_{i_1\neq i_2 \neq  \cdots \neq i_n} \overset{\circ}{\mathcal{H}_{i_1}} \otimes_\mathcal{B} \cdots \otimes_\mathcal{B} \overset{\circ}{\mathcal{H}_{i_n}} .$$ 
 Next, for every $i \in I$ we consider the subspace of $\mathcal{H}$ 
 $$\mathcal{H}(i) \myeq \xi \mathcal{B} \oplus \bigoplus_{i \neq i_1,  i_1\neq i_2 \neq  \cdots \neq i_n} \overset{\circ}{\mathcal{H}_{i_1}} \otimes_\mathcal{B} \cdots \otimes_\mathcal{B} \overset{\circ}{\mathcal{H}_{i_n}},$$ and define the identifications $V_i:  \mathcal{H}_i\otimes_\mathcal{B} \mathcal{H}(i) \to \mathcal{H} $  by  sending $\xi_i \otimes \xi$ to $\xi$, and in the obvious way identifying $\overset{\circ}{\mathcal{H}_i} \otimes \xi$ with $\overset{\circ}{\mathcal{H}_i}$, each $\xi_i \otimes (\overset{\circ}{\mathcal{H}_{i_1}}\otimes_\mathcal{B} \cdots \otimes_\mathcal{B} \overset{\circ}{\mathcal{H}_{i_n}})$ with $\overset{\circ}{\mathcal{H}_{i_1}}\otimes_\mathcal{B} \cdots \otimes_\mathcal{B} \overset{\circ}{\mathcal{H}_{i_n}}$ and each $\overset{\circ}{\mathcal{H}_i} \otimes_\mathcal{B} (\overset{\circ}{\mathcal{H}_{i_1}}\otimes_\mathcal{B} \cdots \otimes_\mathcal{B} \overset{\circ}{\mathcal{H}_{i_n}})$ with $\overset{\circ}{\mathcal{H}_i} \otimes_\mathcal{B} \overset{\circ}{\mathcal{H}_{i_1}}\otimes_\mathcal{B} \cdots \otimes_\mathcal{B} \overset{\circ}{\mathcal{H}_{i_n}}$. Then, for every $i \in I$, define the left actions of the elements in $\widetilde{B}(\mathcal{H}_i)$ on $\calH$ by the inclusion  $\lambda_i : \widetilde{B}(\mathcal{H}_i) \to \widetilde{B}(\mathcal{H})$  given by
 $$\lambda_i(X) \myeq V_i (X \otimes \mathrm{Id}_{\mathcal{H}(i)})V_i^{-1}.$$
Finally,  note that also  $\xi \in \mathcal{H}$ defines a $\mathcal{B}$-valued conditional expectation on $\widetilde{B}(\mathcal{H})$ given by $$E(X) \myeq\langle \xi , X \xi \rangle.$$ 
\begin{theorem}[Voiculescu]
\label{propamalgfreeprod}
The above construction satisfies the following : 
\begin{enumerate}[i)]
    \item For every $i \in I$, $\lambda_i : \widetilde{B}(\mathcal{H}_i) \to \widetilde{B}(\mathcal{H})$ is an injective $\ast$-algebra homomorphism. 
    \item For every $i \in I$ and every $X \in \widetilde{B}(\mathcal{H}_i)$ it holds that $E(\lambda_i(X))=E_i(X)$. 
    \item The family of subalgebras $\left\{\lambda_i\big(\widetilde{B}(\mathcal{H}_i)\big)\right\}_{i\in I}$ of $\widetilde{B}(\mathcal{H})$  are free with amalgamation over $\mathcal{B}$ with respect to $E$. 
\end{enumerate}
\end{theorem}

\subsection{The Operator-Valued $R$-Transform}
\label{subsecRtrans}

In what follows we consider an operator-valued space $(\mathcal{A}, E, \mathcal{B})$ with $\mathcal{A}$ and $\mathcal{B}$ being unital $C^*$-algebras. Moreover, we assume that $E$ is positive, meaning that $E[XX^*]\geq 0$ for all $X\in \mathcal{A}$. 

In the scalar case, namely when $\mathcal{B} \cong \mathbb{C}$,  Voiculescu's $R$-transform is used to compute the distribution of sums of free random variables. See \cite[\S 3]{voiculescu1992free} for an introductory reference. The same machinery extends without many modifications to the operator-valued context. A good introductory reference for this subject is \cite[\S 9]{mingo2017free}. 

The first step towards defining the operator-valued $R$-transform is to define the operator-valued version of the Cauchy transform, which essentially is the conditional expectation applied to the resolvent. %\footnote{The Cauchy transform is also known as the Stieltjes transform, and the resolvent as the Green function. }

\begin{definition}[Operator-valued Cauchy transform]
Let $\mathbb{H}^+ \myeq \{ b\in \mathcal{B} : \mathrm{Im}(b) > \varepsilon 1_\mathcal{A}, \text{ for some } \varepsilon > 0 \}$ and $\mathbb{H}^- \myeq -\mathbb{H}^+$, where $\im(b) \myeq \frac{b-b^*}{2i}$. Then, for any self-adjoint  $X\in \mathcal{A}$, the operator-valued Cauchy transform of $X$ is the function $G_X: \mathbb{H}^+ \to \mathbb{H}^-$ given by \marginpar{$G_X(\cdot)$}
\begin{equation}
\label{eqopvalcauchy}
G_X(b) \myeq E[(b1_\mathcal{A}-X)^{-1}].
\end{equation}

\end{definition}

\begin{observation}
\label{obsCauchy}
We can recover the scalar Cauchy transform from the operator-valued Cauchy transform. More specifically, consider a non-commutative probability space $(\mathcal{A}, \func)$ and a unital subalgebra $\mathcal{B}\subset \mathcal{A}$ together with a conditional expectation $E: \mathcal{A} \to \mathcal{B}$. Furthermore, assume that $\func$ and $E$ are compatible in the sense that $\func(X) = \func(E(X))$ for every  $X\in \mathcal{A}$. 

Let $X\in \mathcal{A}$ be self-adjoint and let $\mu$ be the probability measure on $\mathbb{R}$ given by the distribution of $X$ with respect to $\func$, i.e., $\mu$ is the unique probability measure whose $k$-th moment equals $\func(X^k)$.  Then, for $z\in \mathbb{C}$ with $\mathrm{Im}(z) > 0$ we have that
$$\func(G_X(z1_\mathcal{B})) = \int_\mathbb{R} \frac{1}{z-t} d \mu(t), $$
where $G_X$ is as in (\ref{eqopvalcauchy}). 
\end{observation}

Without going into much detail we recall that the operator-valued Cauchy transform has a left  inverse when restricted to a proper neighborhood of infinity, so the following definition makes sense. 

\begin{definition}[Operator-valued $R$-transform]
\label{defRtrans}
Let $X \in \mathcal{A}$ be self-adjoint. In a suitable neighborhood of $0_\mathcal{B} \in \mathcal{B}$, the operator-valued $R$-transform of $X$, denoted $R_X$, is well defined by the following equation \marginpar{$R_X(\cdot)$}
\begin{equation} \label{eqn:defrtrans}
    bG_X(b) = 1+R_X(G_X(b)) G_X(b),
\end{equation}
where $G_X(b)$ is as in (\ref{eqopvalcauchy}). 
\end{definition}

Many things are known about the $R$-transform (both the scalar and operator-valued). Here we will only need the following fact.

\begin{theorem}[Voiculescu]
Let $X, Y \in \mathcal{A}$ be self-adjoint random variables that are free with amalgamation over $\mathcal{B}$, then 
$$R_{X+Y} = R_X+R_Y, $$
where $R_X$ and $R_Y$ are as in Definition \ref{defRtrans}. 
\end{theorem}

\section{Band Structure of Universal Covers}
\label{secnumberofbands}

 In this section $\calG$ will be a graph with $n$ vertices, edge weights $a$ and vertex potential $b$. And as above, $\calT$ will denote its universal cover (with the inherited periodic weights and potential)  and $A_\calT$ will be the corresponding periodic Jacobi matrix on $\calT$. 
 
 \subsection{Sunada's Theorem via Aymptotic Freeness}
 \label{sec:sunadastheorem}

Here we give an alternate route to a theorem of Sunada \cite{sunada1992group, avni2020periodic} (i.e. Theorem \ref{thm:mass}) upper-bounding the number of connected components of the spectrum of the universal cover, and below we point out a slight generalization to the case where the base graph is allowed to contain half-loops. But for now let us assume that every loop in $\calG$ is a whole-loop. 

For every $d$ let 
$$A_{d, \calG} =\sum_{u\in V(\calG)} b_u \Delta_{uu} \otimes I_d+ \sum_{\diredge\in E(\calG)} a_\diredge \Delta_{\sigma(\diredge)\tau(\diredge)} \otimes U_\diredge^{(d)}$$
  be the Jacobi matrix of a random $d$-lift, as specified in Definition \ref{def:randlift}.  Now, if $\calG$ has $m$ undirected edges, say $\{\diredge_1, \che_1\}, \dots, \{\diredge_m, \che_m\}$, for every $i\in [m]$ we can use $U_i^{(d)}$ as a shorthand notation of $U_{\diredge_i}^{(d)}$. Then, by construction  $U_1, \dots, U_m$ are independent random uniform permutation matrices and $\{U_\diredge^{(d)}\}_{\diredge\in E(\calG)}= \{U_i^{(d)}\}_{i\in [m]} \cup \{(U_i^{(d)})^*\}_{i\in [m]}$. It easy to see that each $U_i$ converges in distribution to a Haar unitary random variable (see (\ref{eqhaarunitary}) for a definition). Moreover, the asymptotic joint distribution is given by the following classical result of Nica. 
  
  \begin{theorem}[Nica \cite{nica1993asymptotically}]
  \label{thm:nica}
Let $\{U_1^{(d)}, \dots, U_m^{(d)}\}$ be a family of independent random uniform $d\times d$ permutation matrices. Then, with respect to the state $\frac{1}{d} \mathbb{E} \circ \mathrm{Tr}(\cdot)$, as $d$ goes to infinity, the family $\{U_1^{(d)}, \dots, U_m^{(d)}\}$ converges in distribution to a family  $\{u_1, \dots, u_m\}$ of free Haar unitaries. 
\end{theorem}
  
 On the other hand, recall that the result in Proposition \ref{propfreegenerators} gives a construction for a family of free Haar unitaries in $C_{\mathrm{red}}^*(\mathbb{F}_m)$, while  Lemma \ref{lem:graphconvergence}  says that random lifts converge in distribution to $A_\T$. These arguments put together yield the following lemma, which appeared implicitly in \cite{bordenave2019eigenvalues}.  
  
 \begin{lemma}
 \label{lem:repofATinfreegroup}
 Let $\{\diredge_1, \che_1\}, \dots, \{\diredge_m, \che_m\}$ be a numbering of the undirected edges of $\calG$, and $g_1, \dots, g_m$ be the canonical free generators of $\bF_m$, $\func$ the canonical trace of $\cred(\bF_m)$ and $\lambda$ the left regular representation of $\bF_m$ on $\ell^2(\bF_m)$. Then, the density of states of $A_{\calT}$ is the same as the spectral measure of 
 \begin{equation}
 \label{eq:repofATinfreegroup}
 \sum_{u\in V(\calG)} b_u\Delta_{uu} \otimes 1+ \sum_{i=1}^m a_{\diredge_i}\big( \Delta_{\sigma(\diredge_i)\tau(\diredge_i)}\otimes  \lambda(g_{i}) + \Delta_{\tau(\diredge_i)\sigma(\diredge_i)} \otimes \lambda(g_i^{-1}) \big)  
 \end{equation}
 in the non-commutative probability space $(M_n(\mathbb{C})\otimes C_{\mathrm{red}}^*(\mathbb{F}_m), \frac{1}{n}\mathrm{Tr}\otimes \func)$.
 \end{lemma} 
 
Now  the problem has been reduced to studying the spectral measure of a particular operator that lives in a well-studied $C^*$-algebra.\footnote{A similar approach is used by \cite{avni2020periodic} using a different $C^*$-algebra.  %Probably not useful to try to define this in one sentence, so maybe just remove it: They show that $A_\T$ is unitarily equivalent to the Jacobi matrix of the base graph with one edge for each element in the fundamental group of $\G$ replaced with a free group generator.  
Their approach does not use asymptotic freeness or any other free probability concept.} In order to study the band structure of the spectrum of this random variable  a standard trick in operator $K$-theory will be used. In the context of graph theory, this technique was first used by Aomoto  \cite{aomoto1988algebraic} and has been used in the study of different infinite graphs by different authors (e.g. \cite{sunada1992group, kollar2019line}).   

The main technical result that one needs from the theory of $C^*$-algebras is the following. 

\begin{proposition}
\label{propprojectionsKtheory}
Let $m$ and $n$ be positive integers and $\func$ be the canonical trace in $C_{\mathrm{red}}^*(\mathbb{F}_m)$. Then, if $P \in M_n (\mathbb{C})\otimes C_{\mathrm{red}}^*(\mathbb{F}_m)$ is a projection, $(\mathrm{Tr}\otimes \func )(P)$ is a non-negative integer. 
\end{proposition}

Using a standard $K$-theory argument it can be seen that the above proposition follows from a deep and technical result of Pimsner and Voiculescu \cite{pimsner1982k}. A self-contained proof appears in \cite{avni2020periodic}. We can now show Sunada's theorem.

\begin{proof}[Proof of Theorem \ref{thm:mass}] 
Use $T_\G$ to denote the operator in (\ref{eq:repofATinfreegroup}), and let $I$ be a connected component of $\Spec(A_\mathcal{T}) = \Spec(T_\G)$. Now let $ \chi_I(x)$  be the indicator function of the set $I$, and note that since $I$ is a connected component of the spectrum, $\chi_I$ is a continuous function on $\Spec(T_{\mathcal{G}})$ and hence $\chi_I(T_\mathcal{G})\in M_n(\mathbb{C}) \otimes C_{\mathrm{red}}^*(\mathbb{F}_m)$. Moreover, since the range of $\chi_I$ is contained in $\mathbb{R}$ and $\chi_I^2 = \chi_I$, by the continuous functional calculus we have that $\chi_I(T_\G)$ is a projection. 

On the other hand, if $\dos$ is the density of states of $A_{\mathcal{T}}$, and $\mu_{T_\calG}$ is the spectral measure of $T_\calG$ with respect to the function $\frac{1}{n}\Tr\otimes \func$,  we have that
$$ \dos(I) = \mu_{T_\G}(I) = \int_{\Spec(T_\G)} \chi_I(x) d \mu_{T_\mathcal{G}}(x) = \left( \frac{1}{n} \mathrm{Tr} \otimes \func \right) (\chi_I(T_\G)).$$
So by  Proposition \ref{propprojectionsKtheory}, the quantity above is $k/n$ for some integer $k \ge 1$, as desired. 
\end{proof}

\begin{remark}[Spectral splitting]
The $C^*$-algebra representation of $A_\calT$ given in (\ref{eq:repofATinfreegroup}) can also be used to answer other questions about the number of bands in $\Spec(A_{\calT})$. We refer the reader to Appendix \ref{subsec: bandexamples} for a discussion of some problems about the number of bands in $\Spec(A_{\calT})$ that were left open in \cite{avni2020periodic}. 
\end{remark}

\subsection{Allowing Half-Loops}
\label{sec:halfloops}
When the base graph $\G$ has half-loops a modification of Sunada's result still holds. In this case, the random permutations in the lift associated to half-loops should be taken to be random matchings. To be precise, if $L(\calG)$ denotes the set of half-loops of $\calG$, the Jacobi matrix on a random $2d$-lift is given by 
$$A_{d, \calG} =\sum_{u\in V(\calG)} b_u \Delta_{uu} \otimes I_d+ \sum_{\diredge\in L(\calG)} \Delta_{\sigma(\diredge)\tau(\diredge)}\otimes P_\diredge^{(d)}+ \sum_{\diredge\in E(\calG)\setminus L(\calG)} a_\diredge \Delta_{\sigma(\diredge)\tau(\diredge)} \otimes U_\diredge^{(d)},$$
where the $P_\diredge^{(2d)}$ are independent random permutations distributed uniformly in the space of random  matchings of $[2d]$ (i.e.  matrices whose corresponding permutation $\alpha$ has no fixed points and satisfies  $\alpha^2 = \mathrm{Id}$). It is not hard to see that by the way we defined the universal cover a graph with half-loops (see Section \ref{sec:defandconv}) we again have Benjamini-Schramm convergence to $\calT$ when defining random lifts in this way, and hence Lemma \ref{lem:graphconvergence} still holds.  

Now note  that each $P_\diredge^{(2d)}$ is self-adjoint and distributes as a Rademacher random variable with respect to the state $\frac{1}{2d}\Tr(\cdot)$. On the other hand, by  the work of Nica \cite{nica1993asymptotically} we  know that asymptotic freeness also holds for independent random matchings and independent random uniform permutations. That is, Theorem \ref{thm:nica} admits the following extension. 

\begin{theorem}[Nica \cite{nica1993asymptotically}]
Let $\{P_1^{(2d)}, \dots, P_{m_1}^{(2d)}, U_1^{(2d)}, \dots, U_{m_2}^{(2d)}\}$ be a family of independent random $d\times d$ permutation matrices, with the $P_i^{(2d)}$ distribute as uniform matchings and the $U_i^{(2d)}$ as uniform permutation matrices. Then, as $d$ goes to infinity, this family converges in distribution (with respect to the state $\frac{1}{2d}\Tr(\cdot)$) to the free family $\{p_1, \dots, p_{m_1}, u_1, \dots, u_{m_2}\}$, where the $p_i$ are Rademacher random variables and the $u_i$ are  Haar unitaries. 
\end{theorem}

On the other hand, a copy in distribution of the free family $\{p_1, \dots, p_{m_1}, u_1, \dots, u_{m_2}\}$ mentioned above is given by the variables $\{\lambda(h_1), \cdots \lambda(h_{m_1}), \lambda(g_1), \dots, \lambda(g_{m_2})\}\subset \cred(\Gamma_{m_1}\ast \bF_{m_2})$, where $h_1, \dots, h_{m_1}, \dots , g_1, \dots, g_{m_2}$ are the canonical generators of the group $\Gamma_{m_1}\ast \bF_{m_2}$ (recall that $\Gamma_{m_1}$ denotes the free product of $m_1$ copies of $\mathbb{Z}_2$), and $\lambda $ denotes the left regular representation. The $K$-theory of this $C^*$-algebra is also well understood (see for example \cite{sunada1992group}), and hence one can get the corresponding analog of Proposition \ref{propprojectionsKtheory}. In this case, one obtains the following result.

\begin{proposition}[Analog of Sunada's theorem]
\label{prop:generalizedsunada}
If $\calG$ contains half-loops, then the density of states of $A_{\calT}$ assigns a positive integer multiple of $1/2n$ to any connected component of $\Spec(A_{\calT})$. Hence $\Spec(A_{\calT})$ has at most $2n$ components. 
\end{proposition}

The bound of $2n$  is again tight since $\G$ can be a finite tree with a half-loop added to some vertex, in which case $\T$ will be a finite tree with $2n$ vertices, and it is then easy to construct examples where the (in this case finite) Jacobi matrix  $A_{\calT}$ has $2n$ distinct eigenvalues.  Moreover, if one allows the coefficients $b_v$ to be nonzero, then one can obtain many more examples with $2n$ bands. In particular we have the following example, which was pointed out to us by Barry Simon.  Roughly speaking, one may take $n$ disjoint copies $\G_1, \dots, \G_n$ of a fixed finite graph $\G$ (and its graph Jacobi matrix), add different large multiples of the identity to the Jacobi matrix on each $\G_i$, and then connect the $\G_i$ by some edges with very small weights $a_\diredge$.  This produces a connected graph $\G'$ whose universal cover has $n$ times as many spectral bands as $A_\G$.  If $\G$ is a single vertex with a whole-loop and a half-loop with coefficients chosen so that the spectrum of $A_{\T(\G)}$ has $2$ bands by Theorem \ref{thm:fts}, then $A_{\T(\G')}$ has $2n$ bands in its spectrum.%, \todo{include Simon example?} 

\begin{remark}[Avoiding asymptotic freeness]
Alternatively, using the  amalgamated free product of graphs,  in Section \ref{sec:rademacherrepresentation} we will show that $A_{\calT}$ can be represented as an element in $M_n(\bC)\otimes \cred(\Gamma_m)$, from where  Proposition \ref{prop:generalizedsunada} can be easily deduced from $K$-theory results as explained above. 
\end{remark}

\section{The  Amalgamated Free Product for Graphs}
\label{secgraphproduct}
In Section \ref{subsecproductcombdesc} we give a combinatorial description of our graph product and show that it extends both the free product of graphs given in \cite{quenell1994combinatorics} and the notion of universal cover of a graph. At the end of Section \ref{subsecproductcombdesc} we state Theorem \ref{thmrefmain},  the main result of this section, which explains the connection with freeness with amalgamation and the use of having such a product.   

In Section \ref{secoperatorproduct} we delve into the technicalities of the product, showing the connection between the combinatorial description given in Section \ref{subsecproductcombdesc} and the construction of the amalgamated free product of Hilbert modules described in Section \ref{subsecHilbertProd}. 

To lighten notation, in the remaining of this section for a graph $\mathcal{G}$ we will write $\mathcal{G} = (\mathcal{V}, \mathcal{E}, e)$ to denote that  $\mathcal{V}(\mathcal{G})=\mathcal{V}$,   $E(\mathcal{G})=\mathcal{E}$, and $e$ is the root of $\mathcal{G}$. \marginpar{$(\mathcal{V}, \mathcal{E}, e)$}

\subsection{Combinatorial Description}
\label{subsecproductcombdesc}

%Motivated by the construction of Cayley graphs of free products of groups, Quenell introduced the notion of free product for general rooted graphs in \cite{quenell1994combinatorics}. Later, Accardi, Lencewksi and Sa\l{}apata \cite{accardi2007decompositions} made explicit the relation between this product and free probability.  We review this construction below.

  First, following the presentation in  \cite{accardi2007decompositions} we  recall  Quenell's free product of rooted graphs. Let $\mathcal{G}_1=(\mathcal{V}_1, \mathcal{E}_1, e_1),\, \dots,\, \mathcal{G}_n=(\mathcal{V}_n, \mathcal{E}_n, e_n)$ be rooted graphs and for every $i\in [n]$ denote $\overset{\circ}{\mathcal{V}_i} \myeq \mathcal{V}_i \setminus \{e_i\}$\marginpar{$\overset{\circ}{\mathcal{V}_i}$}. The \emph{free product of rooted graphs} will be denoted by $\ast \{\mathcal{G}_i\}_{i=1}^n$. 

The vertex set of $\ast \{\mathcal{G}_i\}_{i=1}^n$  will be the following set of words: \marginpar{$\ast \{\mathcal{V}_i\}_{i=1}^n$}
\begin{equation}
\label{eqfreeprodvert}
    \ast \{\mathcal{V}_i\}_{i=1}^n \myeq \{e \} \cup \{v_m\cdots v_1 : v_k \in \overset{\circ}{\mathcal{V}_{i_k}} \text{and } i_k \neq i_{k+1} \text{ for } 1 \le k \le m-1, m\geq 1 \}. 
\end{equation}
As in \cite{accardi2007decompositions} we think of $e$ as the empty word and  we allow the roots $e_i$ to appear in words also playing the role of an empty word, that is, if $w \in \ast \{\mathcal{V}_i\}_{i=1}^n$ then we have the convention $we_i = e_i w = w$. The set of edges  in the free product is defined as follows \marginpar{$\ast \{\mathcal{E}_i\}_{i=1}^n $}
\begin{equation}
\label{eqfreeprodedge}
\ast \{ \mathcal{E}_i\}_{i=1}^n = \left\{(vw, v'w) : (v, v') \in \cup_{i\in [n]} \mathcal{E}_i \text{ and } w, vw, v'w \in \ast \{ \calV_i\}_{i=1}^n \right\}. 
\end{equation}
 In summary  $\ast \{\mathcal{G}_{i}\}_{i=1}^n = (\ast \{\mathcal{V}_i\}_{i=1}^n, \ast \{\mathcal{E}_i\}_{i=1}^n , e)$\marginpar{$\ast \{\mathcal{G}_i\}_{i=1}^n$}. See Figure \ref{fig:freeproduct}  for an example.

 \begin{figure}
  \centering
     \includegraphics[scale=.22]{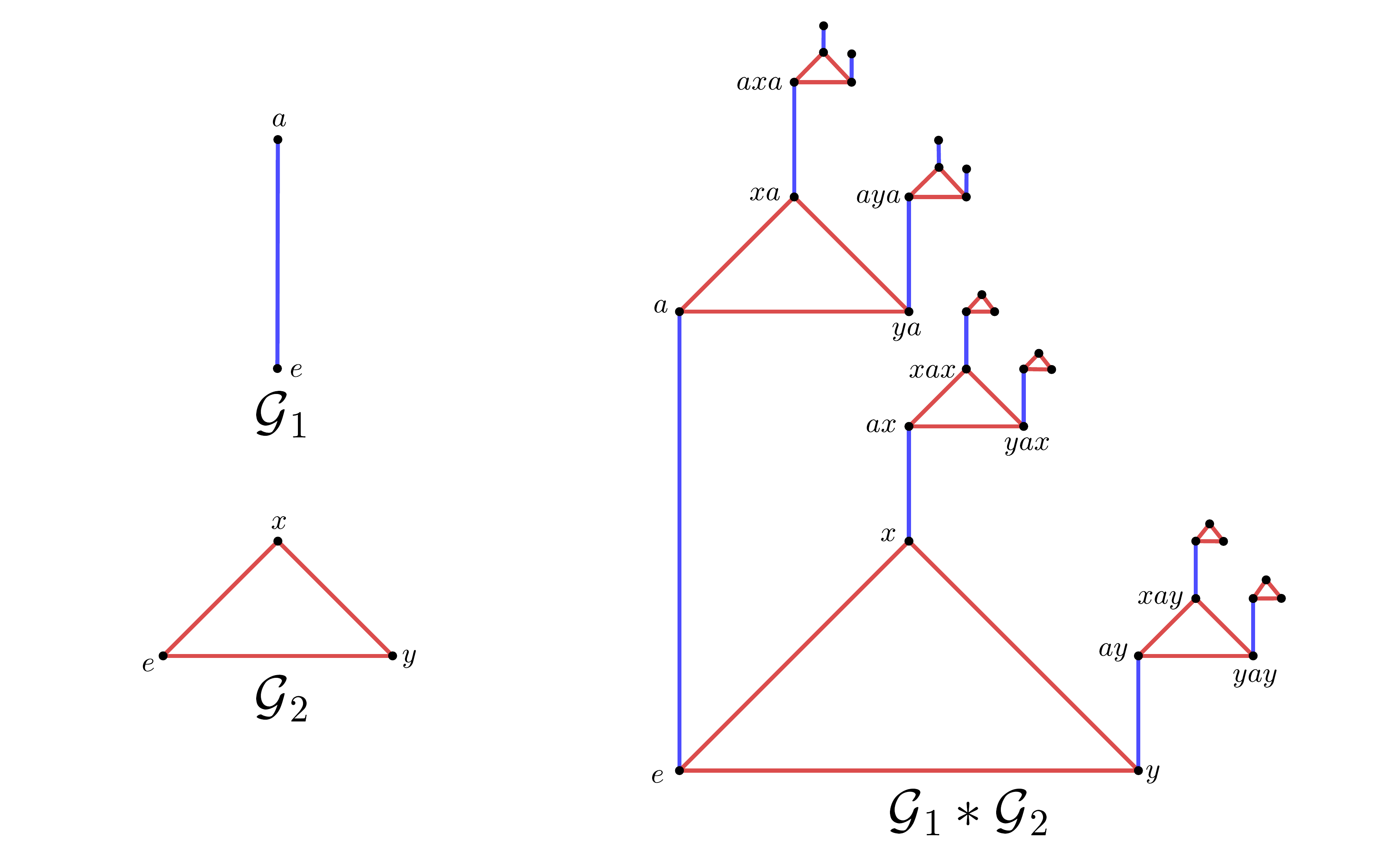}
 \caption{The free product of two graphs $\mathcal{G}_1$ and $\mathcal{G}_2$.  }
 \label{fig:freeproduct}
 \end{figure}

 We now expand the notion of free product of graphs by introducing the concept of \emph{relator graph}. 

 \begin{definition}[Relator graph]
Let $\mathcal{C}$ be a finite set. A relator graph with colors in $\mathcal{C}$ is a finite  graph $\mathcal{G}= (\mathcal{V}, \mathcal{E})$   together with an edge coloring $c: \mathcal{E} \to \mathcal{C}$. 
\end{definition}

Our \emph{amalgamated free product of rooted graphs}  $\mathcal{G}_1=(\mathcal{V}_1, \mathcal{E}_1, e_1), \dots, \mathcal{G}_n=(\mathcal{V}_n, \mathcal{E}_n, e_n)$ will be defined relative to a relator graph $(\mathcal{G}, c)$ and a coloring of the $\mathcal{G}_i$ compatible with the coloring of the relator graph. The colorings together with the graph  structure of the  relator graph  indicate how the $\mathcal{G}_i$ will be combined to form an infinite (possibly disconnected) multi-rooted graph. Thus, one may obtain different results even when the $\mathcal{G}_i$ are fixed, if the relator $(\mathcal{G}, c)$ is modified. 

\begin{definition}[Amalgamated free product for graphs]
\label{defmain}
First we describe the input data for the product: 
\begin{itemize}
    \item (Colored rooted graphs) Let $\mathcal{G}_1 = (\mathcal{V}_1, \mathcal{E}_1, e_1),  \dots , \mathcal{G}_n = (\mathcal{V}_n, \mathcal{E}_n, e_n)$ be finite rooted  graphs. Assume that each $\mathcal{G}_i$ comes equipped with an edge coloring $c_i : \mathcal{E}_i\to \mathcal{C}_i$ such that $\mathcal{C}_i \cap \mathcal{C}_j = \emptyset$ for every $i\neq j$.
    \item (Relator graph) Let $\mathcal{C} \myeq \bigcup_{i=1}^n \mathcal{C}_i$ and let $(\mathcal{G}, c)$ be a relator graph with colors in $\mathcal{C}$ and $k$ vertices. Denote $\mathcal{G} = (\mathcal{V}, \mathcal{E})$ and without loss of generality set $\mathcal{V} = [k]$.  
\end{itemize}
Then, the free product of the $(\mathcal{G}_i, c_i)$ with amalgamation over $(\mathcal{G}, c)$, which we denote by $*_{(\mathcal{G}, c)} \left\{(\mathcal{G}_i, c_i)\right\}_{i=1}^n$, is the $k$-rooted  graph  defined as follows: 
\begin{itemize}
    \item (Vertex set) For the vertex set we take $k$ copies of the free product of the $\calV_i$, that is 
    $$\ast_{(\mathcal{G}, c)} \{\mathcal{V}_i\}_{i=1}^n \myeq  [k] \times \ast \{\mathcal{V}_i\}_{i=1}^n.$$ 
    \item (Edge set) The colorings $c_i$ and $c$ are used to know which edges should be added:
    $$\ast_{(\mathcal{G}, c)} \{(\mathcal{E}_i, c_i)\}_{i=1}^n \myeq \left\{ \big( (l, vw), (l', v'w)\big) : (v, v')\in \mathcal{E}_j \text{ for  some } j \in [n], (l, l') \in \mathcal{E} \text{ and } c(l, l') = c_j(v, v')  \right\}, $$
    where of course we are assuming that $w, vw, v'w\in \ast \{\calV_i\}_{i=1}^n$. 
    \item (Roots) The root set of $\calG$ is $\{(l, e)\}_{l=1}^k$. 
\end{itemize}
\end{definition}

To lighten notation sometimes we will use $\ast_{(\calG, c)}\{\mathcal{G}_i\}_{i=1}^n$ and  $\ast_{(\calG, c)}\{\mathcal{E}_i\}_{i=1}^n$, as shorthand notations for $*_{(\mathcal{G}, c)} \left\{(\mathcal{G}_i, c_i)\right\}_{i=1}^n$ and $*_{(\mathcal{G}, c)} \left\{(\mathcal{E}_i, c_i)\right\}_{i=1}^n$, respectively. 

We first observe that this product generalizes different relevant constructions.

\begin{example}[Universal cover of a graph]
\label{exunivcover}
Let $\mathcal{G}=(\mathcal{V}, \mathcal{E})$ be a finite graph with $k$ vertices. 

First we consider the case where $\G$ has no loops. Let $m$ be the number of undirected edges in $\calG$ and let $c: \mathcal{E} \to [m]$ be a coloring satisfying $c(\diredge)=c(\che)$ and $c(\diredge_1)\neq c(\diredge_2)$ whenever $\diredge_1 \neq \diredge_2$ and $\che_1\neq \diredge_2$. Then choose $(\mathcal{G}, c)$ as the relator graph, and for every $i\in [m]$  let  $\mathcal{G}_i$  be a rooted graph consisting of two vertices (one of them being the root $e_i$) connected by a single  undirected edge with color $i$.   

Now, if $\G$ has loops one should replace every whole-loop for two half-loops. Then $m$ is defined as the number of undirected non-loop edges in $\calG$ plus the number of half-loops. Once again one considers a coloring $c: \mathcal{E} \to [m]$ with $c(\diredge_1)=c(\diredge_2)$ if and only if $\diredge_1=\diredge_2$ or $\che_1=\diredge_2$.  The rest is done as in the no-loop case.

In any case, with this construction, note that $\calTf:= \ast_{(\calG, c)} \{\calG_i\}_{i=1}^m$ has $k$ roots $(1, e), \dots, (k, e)$ and that the connected component of $\calTf$ containing $(i, e)$, say $\calT_i$, is isomorphic to $\calT(\calG)$. Moreover, $\calT_i$ and $\calT_j$ are disjoint whenever $i\neq j$ and the unique root $(i, e)$ in $\calT_i$ is in the fiber of  $i\in V(\calG)$. See Figure \ref{fig:universalcover} for an example. 
\end{example}

 \begin{figure}[h]
  \centering
\includegraphics[scale=.29]{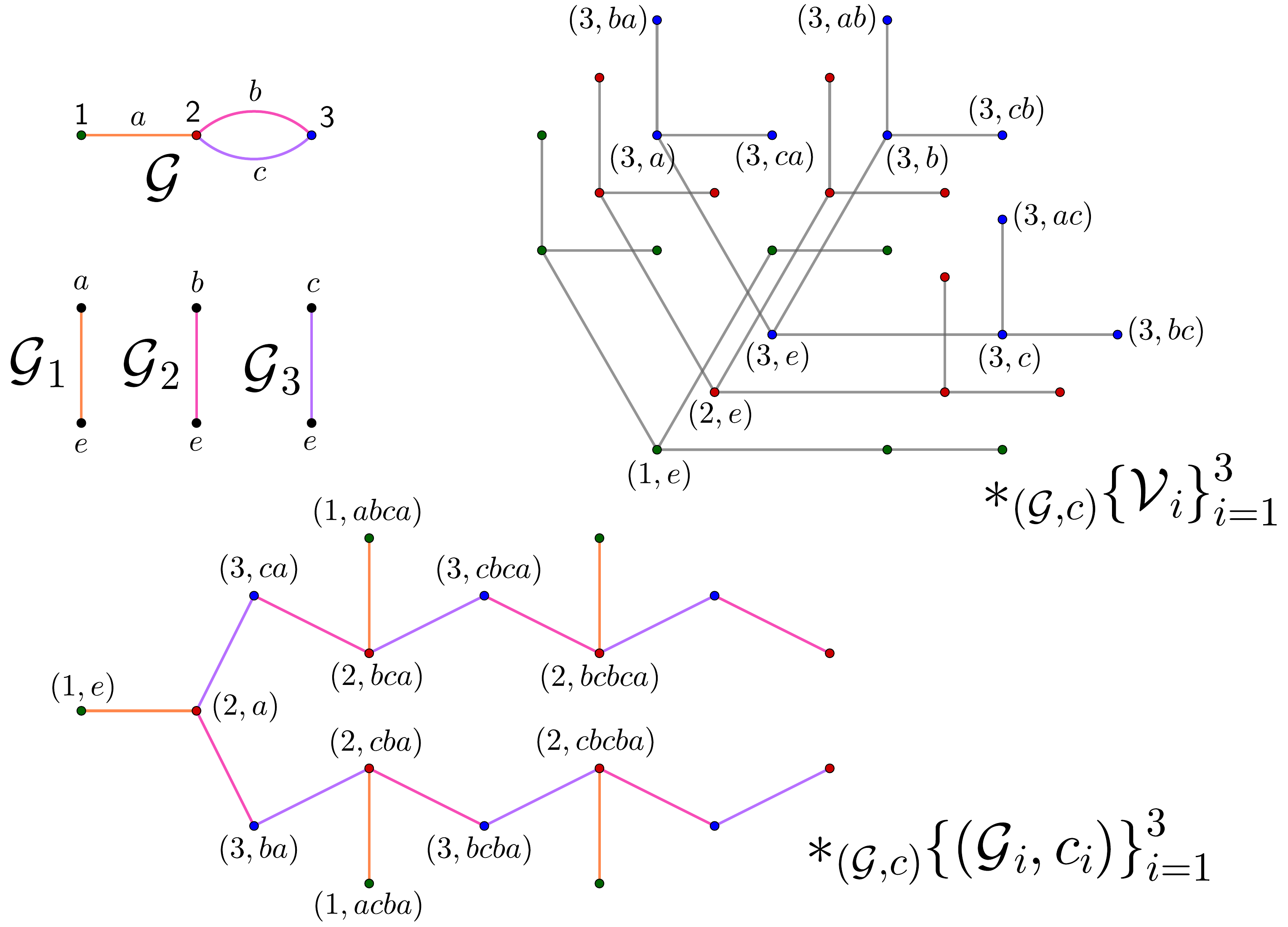}
\caption{Here we illustrate the construction in Example \ref{exunivcover} in the particular case of a graph $\calG$ with three vertices and three edges. On the top-left we have the graph $\calG$ and the graphs $\calG_1, \calG_2, \calG_3$ defined in the construction. On the top-right a part of the vertex set $\ast_{(\calG, c)}\{\calV_i\}_{i=1}^n$ is shown, which in general is the union of the vertex sets of $k$ disjoint $m$-regular trees, where $m$ is the number of undirected edges in $\calG$. The vertices of these $m$-regular trees correspond to sequences of edges in $\calG$ containing all non-backtracking walks, and the connected components of the amalgamated graph product $\ast_{(\calG, c)}\{\calG_i\}_{i=1}^n$ (which are copies of $\calT$) move across the different $k$ levels in the vertex set, as it can be seen in the lower part of the image, which is part of the connected component containing the root $(1, e)$.   }
 \label{fig:universalcover}
\end{figure}

\begin{example}[Free products of graphs]
Given $n$ finite rooted graphs $\mathcal{G}_1 = (\mathcal{V}_1, \mathcal{E}_1, e_1),  \dots , \mathcal{G}_n = (\mathcal{V}_n, \mathcal{E}_n, e_n)$, one can construct its free product by taking $\calG=(\calV, \calE)$ to be the graph that has only one vertex and $n$ half-loops around it. Then, for any $i\in [n]$, let $\calC_i$ be the singleton $\{i\}$, that is $c_i$ will color all edges in $\calG_i$ with color $i$. And, $c: \calC \to [n]$ will color each half-loop in $\calG$ with a different color.  

With this setup it is easy to see that 
\[\ast_{(\mathcal{G}, c)} \{(\mathcal{G}_i, c_i)\}_{i=1}^n =  \ast \{\mathcal{G}_i\}_{i=1}^n.\] 
In this case, because the relator graph has only one vertex the amalgamated product has only one root.  Hence, this construction generalizes Quenell's free product of graphs. 
\end{example}

\begin{example}[Cayley graphs of amalgamated free products of groups]
 In Section \ref{sec:cayley-amalgamated} we will show that if $G_1, \dots, G_n$ are finite groups with symmetric generating sets $S_1, \dots, S_n$ and a common subgroup $H$, then, the left Cayley graph $\Gamma(\ast_H \{G_i\}_{i=1}^n, S)$ (where $S=\bigcup_{i=1}^nS_i$) can be constructed  explicitly as an amalgamated free product of graphs. In our construction, $H$ is taken to be the vertex set of the relator graph, and for every $i\in [n]$ a set of representatives for the right cosets of $H$ in $G_i$ is taken as the vertex set of the  graph $\calG_i$. Then, the edges in $\calG$ and $\calG_i$,  and the colorings $c$ and $c_i$ encode the way in which the generators interact with the elements in $H$ and the cosets of $H$ in $G_i$. %\todo{Add comment about normal subgroups.}
 \end{example}

So far all the discussion in this section has been at the level of combinatorics; however, our main result here is about Jacobi operators. It is then necessary to clarify how  edge weights and vertex potentials can be incorporated in our framework.

\begin{definition}[Lifting coefficients]
\label{def:liftweights}
If the relator graph $\calG=(\calV, c)$ has edge weights $a$ and vertex potential $b$, then the graph $\ast_{(\calG, c)} \{\calG_i\}_{i=1}^n$ can be endowed with ``periodic" edge weights $\tilde{a}$ and vertex potential $\tilde{b}$ in a natural way as follows: 
$$\tilde{a}_{\big( (l, vw), (l', v'w)\big)} := a_{(l, l')} \quad \text{and} \quad \tilde{b}_{(l, w)}:= b_l,$$
for all $\big( (l, vw), (l', v'w)\big)\in \ast_{(\calG, c)} \{\calE_i\}_{i=1}^n$ and $(l, w)\in \ast_{(\calG, c)} \{\calV_i\}_{i=1}^n$. 
\end{definition}

When working with Jacobi operators on amalgamated free products of graphs it will be convenient to have the following notation. 

\begin{definition}[Notation]
\label{def:notation}
Fix $\mathcal{G}= (\mathcal{V}, \mathcal{E})$  a graph and $c: \mathcal{E} \to \mathcal{C}$  a coloring of the edges. For each $\alpha \in \mathcal{C}$ we denote by $\mathcal{G}[\alpha]$  \marginpar{$\G[\alpha]$} the graph with the same vertex set as $\mathcal{G}$ but that only includes those edges of color $\alpha$, i.e. $\mathcal{G} = (\mathcal{V}, c^{-1}(\alpha))$. If $\calG$ has edge weights and vertex potential, then $\calG[\alpha]$ inherits these in the obvious way. 
\end{definition}

We are now ready to state the main theorem. 

\begin{theorem}
\label{thmrefmain}
 Use the notation in Definition \ref{defmain} and Definition \ref{def:notation}, and assume that the relator graph $\calG$ is equipped with edge weights $a$ and  vertex potential $b$. Let $J$ denote the Jacobi operator on $\ast_{(\calG, c)} \{(\calG_i, c_i)\}_{i=1}^n$ where the edge weights and vertex potential are defined as in Definition \ref{def:liftweights}.

 For every $i\in[n]$   and every $\alpha \in \mathcal{C}_i$ let $X_\alpha$ and $A_\alpha$ be the Jacobi matrix of  $\mathcal{G}[\alpha]$ and the adjacency matrix of $\mathcal{G}_i[\alpha]$ respectively. Then, there exists an operator-valued probability space $(\mathcal{A}, E, M_k(\C))$, and random variables $D, T_1, \dots, T_n \in \mathcal{A}$ with the following properties: 
\begin{enumerate}[i)]
    \item \label{mainthm:pathcounting} (Path counting) If $T:=D+T_1+\cdots +T_n$, then for every natural number $p$ and $i, j \in [k]$, we have that 
    $$E[T^p](i, j) = J^p\big((e, i), (e, j)\big).$$
    \item \label{mainthm:freenesswithamalgam} (Freeness with amalgamation) The $T_i$ are free with amalgamation over $M_k(\C)$ with respect to $E$. Moreover $D$ is free with amalgamation over $M_k(\C)$ from any other element in $\calA$. 
    \item \label{mainthm:equalindist} (Equal in distribution) Fix $i \in [n]$, let $m_i \myeq |\mathcal{V}_i| $ and note that $X_\alpha \otimes A_\alpha \in M_k(\mathbb{C})\otimes M_{m_i}(\mathbb{C})$ for every $\alpha \in \mathcal{C}_i$. Then, if we view $ M_{m_i}(\mathbb{C})$ as a non-commutative probability space with the functional induced by the root of $\mathcal{G}_i$, and from there construct the $M_k(\mathbb{C})$-valued probability space $(M_k(\mathbb{C})\otimes M_{m_i}(\mathbb{C}), E_i, M_k(\mathbb{C}) )$ as in Example \ref{exmatrixopval}, we have that 
    $$\sum_{\alpha\in \mathcal{C}_i} X_\alpha \otimes A_\alpha$$
    distributes with respect to $E_i$ as $T_i$ with respect to $E$. Moreover, $D$ distributes with respect to $E$ as the matrix $\mathrm{diag}(b_1, \dots, b_k)$ distributes in the space $(M_k(\bC), \mathrm{Id}, M_k(\bC))$. 
\end{enumerate}
\end{theorem}

\begin{remark}[Scalar-valued distributions]
Note that item (\ref{mainthm:pathcounting}) in the above theorem tells us that the $k\times k$ matrices obtained as the operator-valued moments of $T$ encode all of the information provided by walks between any two roots of the amalgamated free graph product. Hence, for every $i\in [k]$, if one composes the conditional expectation $E$ with the functional $\func_i: M_k(\bC)\to \bC$ given by $\func_i(X) :=X(i, i)$, then one can understand the spectral measure of the Jacobi matrix $J$ on $\ast_{(\calG, c)}\{(\calG_i, c_i)\}$ associated to the root $(e, i)$.
\end{remark}
 To see Theorem \ref{thmrefmain} in action we refer the reader to Section \ref{sec:rademacherrepresentation}, where we combine Example \ref{exunivcover} and Theorem \ref{thmrefmain} to obtain a $C^*$-representation of periodic Jacobi operators on universal covers. 

\subsection{Proof of Theorem \ref{thmrefmain}}
\label{secoperatorproduct}

Since we care about the explicit action of the operators in question we think of the  $X_\alpha$ as elements in $B(\ell^2(\mathcal{V}))$ and of the $A_\alpha$ as  elements of $B(\ell^2(\mathcal{V}_i))$ if $\alpha \in \mathcal{C}_i$. To lighten notation let $$\mathcal{B} \myeq B(\ell^2(\mathcal{V})) \cong M_k(\C). $$
and because of the last equivalence let $I_k$ denote the unit of $\calB$. In what follows the notion of tensor will be used in different ways, we will use $\otimes_\mathcal{B}, \otimes_{\mathbb{C}}$ and $\otimes$  to denote the tensor of $\mathcal{B}$-modules, $\mathbb{C}$-algebras and $\mathbb{C}$-vector spaces respectively. Although, to avoid overloading, the subscript in $\otimes_\mathcal{B}$ will be dropped when using it to denote pure tensor elements in a $\mathcal{B}$-module tensor product.

For clarity of exposition we divide the proof in several steps. The first one consists in defining the building blocks $\mathcal{H}_i$ to which we apply the construction of the amalgamated free product for Hilbert modules described in Section \ref{subsecHilbertProd}. 
\bigskip

\noindent \emph{Step 1: Construction of the operator-valued probability spaces.}  For every $i\in [n]$ let $$\mathcal{H}_i \myeq \mathcal{B} \otimes \ell^2(\mathcal{V}_i),$$ and note that this is a  $\calB$-bimodule with the right and left actions of $\mathcal{B}$ on pure tensors given by
\begin{equation}
\label{eqdefrightaction}
(X\otimes \eta) Y \myeq XY \otimes \eta \quad \text{and} \quad Y(X\otimes \eta) = YX \otimes \eta \quad \forall X, Y \in \mathcal{B}, \, \forall \eta \in \ell^2(\mathcal{V}_i),  
\end{equation}
and then extended by linearity. Moreover, we can turn each $\mathcal{H}_i$ into a right   Hilbert $\mathcal{B}$-module by using the $\mathcal{B}$-valued inner product determined by
\begin{equation}
\label{eqdeftensor}
\langle X\otimes \eta, Y\otimes \zeta \rangle_{\mathcal{H}_i} \myeq \langle \eta, \zeta \rangle_{\ell^2(\mathcal{V}_i)} X^* Y, \quad \forall X, Y \in \mathcal{B}, \, \forall \eta, \zeta \in \ell^2(\mathcal{V}_i).
\end{equation}
In fact, because we are working with finite-dimensional spaces, operators are automatically bounded, so this turns $\calH_i$ into a Hilbert $\calB$-bimodule where, as expected, the representation $\pi_i : \calB \to \widetilde{B}(\calH_i)$ is given by the left action of $\calB$ on $\calH_i$. 

Now recall that for every $i\in [n]$,  $e_i \in \mathcal{V}_i$ denotes the root of $\mathcal{G}_i$, so $\delta_{e_i} $ is the distinguished vector in the Hilbert space $\ell^2(\mathcal{V}_i)$. Following the construction outlined in Section \ref{subsecHilbertProd},  we set $\xi_i \myeq I_k \otimes \delta_{e_i}$ to be the distinguished vector of $\mathcal{H}_i$. It is easy to see that $\xi_i$ is a unit central vector,  and moreover, by (\ref{eqdefrightaction}),  $\xi_i \mathcal{B} = \mathcal{B}\otimes \delta_{e_i}$. Hence, the decomposition $\mathcal{H}_i= \xi_i \mathcal{B} \oplus \overset{\circ}{\mathcal{H}_i}$ satisfies $$\overset{\circ}{\mathcal{H}_i} = \mathcal{B} \otimes \ell^2(\overset{\circ}{\mathcal{V}_i} ),$$ where as before $\overset{\circ}{\mathcal{V}_i} \myeq \mathcal{V}_i\setminus \{e_i\}$. 

As mentioned in Example \ref{exhilbmods}, the $\ast$-algebra of bounded, adjointable, right $\mathcal{B}$-linear operators $\widetilde{B}(\mathcal{H}_i)$ is a $\mathcal{B}$-valued probability space with the conditional expectation defined by $E_i(S) \myeq \langle \xi_i, S\xi_i\rangle_{\mathcal{H}_i}$ for every $S\in \widetilde{B}(\mathcal{H}_i)$. But since $\mathcal{B}$ and $B(\ell^2(\mathcal{V}_i))$ are finite-dimensional, we can use the identification $$\widetilde{B}(\mathcal{H}_i) \cong \mathcal{B} \otimes_\mathbb{C} B(\ell^2(\mathcal{V}_i)),$$ by letting each $X \otimes A \in \mathcal{B}\otimes_{\mathbb{C}} B(\ell^2(\mathcal{V}_i))$ act on $\mathcal{H}_i$ by 
\begin{equation}
\label{eqdefactionoftensor}
X\otimes A (X' \otimes \eta) = XX' \otimes A \eta, \quad \forall X' \in \mathcal{B}, \forall \eta \in \ell^2(\mathcal{V}_i), 
\end{equation}
and extending this definition by linearity to all elements in $\mathcal{B}\otimes_{\mathbb{C}} B(\ell^2(\mathcal{V}_i))$. Now note that under this identification, the $E_i$ we just defined also coincide with the construction given in Example \ref{exmatrixopval}, since for a pure tensor $X \otimes A \in \mathcal{B} \otimes_{\mathbb{C}} B(\ell^2(\mathcal{V}_i)) $ we have, by (\ref{eqdeftensor}) and (\ref{eqdefactionoftensor}), that 
$$E_i( X \otimes A) = \langle \xi_i , X \otimes A (\xi_i) \rangle_{\mathcal{H}_i} = \langle I_k \otimes \delta_{e_i}, X \otimes A (I_k\otimes \delta_{e_i}) \rangle_{\mathcal{H}_i} = \langle I_k\otimes \delta_{e_i} , X \otimes A \delta_{e_i} \rangle_{\mathcal{H}_i} = \langle \delta_{e_i}, A \delta_{e_i} \rangle_{\ell^2(\mathcal{V}_i)} X.  $$
\bigskip

\noindent \emph{Step 2: Defining $\A$, $\func$, and the $T_i$.} Note that $\B$ and the Hilbert  $\B$-bimodules $\mathcal{H}_i$ satisfy all the conditions required in Section \ref{subsecHilbertProd}, so we can go ahead with the construction of the amalgamated free product.  Define
\begin{align*}
\mathcal{H} &\myeq \xi\mathcal{B} \oplus \bigoplus_{i_1\neq i_2 \neq  \cdots \neq i_m} \overset{\circ}{\mathcal{H}_{i_1}} \otimes_\mathcal{B} \cdots \otimes_\mathcal{B} \overset{\circ}{\mathcal{H}_{i_m}},
\end{align*}
where $\xi = I_k \otimes \delta_e$ and $\delta_e$ is the vector resulting from identifying the $\delta_{e_i}$, and let $\calA \myeq  \widetilde{B}(\calH)$.  Then, for every $i \in [n]$,  set $\lambda_i : \widetilde{B}(\mathcal{H}_i) \to \calA$ to be the inclusions described in Section \ref{subsecHilbertProd} and define
$$T_i \myeq \sum_{\alpha \in \mathcal{C}_i} \lambda_i(X_\alpha \otimes A_\alpha).$$

%We will  take the $\mathcal{B}$-valued conditional expectation on $\calA$ given by $E[S] = \langle S, S\xi \rangle_{\calH}$, and  $\tau_r$ will denote the state on $\mathcal{B}$ given by $\tau_r(X) \myeq \langle \delta_r, X \delta_r \rangle_{\ell^2(\mathcal{V})}$. 
%We can then lift $\tau_r$ to a state $\tau$ on $\calA$ by defining 
%$$\tau \myeq \tau_r\circ E.$$

With these definitions, from Theorem \ref{propamalgfreeprod} it is clear that the $T_i$ are free with amalgamation over $\B$ with respect to the conditional expectation $E$, and that each $T_i$ has the same distribution as
$$\sum_{\alpha\in \mathcal{C}_i} X_\alpha \otimes A_\alpha.$$
This proves the claims about the $T_i$ made in items \ref{mainthm:freenesswithamalgam}) and \ref{mainthm:equalindist}) in Theorem \ref{thmrefmain}. Before completing the proof of the theorem  we will understand better the structure of $\calH$. 
\bigskip

 \noindent \emph{Step 3: Rewriting $\calH$ and defining $D$.}  First recall that
\begin{align*}
\mathcal{H} &= \xi\mathcal{B} \oplus \bigoplus_{i_1\neq i_2 \neq  \cdots \neq i_m} \overset{\circ}{\mathcal{H}_{i_1}} \otimes_\mathcal{B} \cdots \otimes_\mathcal{B} \overset{\circ}{\mathcal{H}_{i_m}} 
\\ &= \mathcal{B}\otimes \delta_{e} \oplus \bigoplus_{i_1\neq i_2 \neq \cdots \neq i_m} (\mathcal{B}\otimes \ell^2(\overset{\circ}{\mathcal{V}}_{i_1})) \otimes_\mathcal{B} \cdots \otimes_\mathcal{B} (\mathcal{B}\otimes \ell^2(\overset{\circ}{\mathcal{V}}_{i_m})) 
\nonumber    
\end{align*}
and note that as $\B$-bimodules we have the identification
$$(\mathcal{B}\otimes \ell^2(\overset{\circ}{\mathcal{V}}_{i_1})) \otimes_\mathcal{B} \cdots \otimes_\mathcal{B} (\mathcal{B}\otimes \ell^2(\overset{\circ}{\mathcal{V}}_{i_m}))  \cong \mathcal{B} \otimes \ell^2(\overset{\circ}{\mathcal{V}}_{i_1})\otimes \cdots \otimes \ell^2(\overset{\circ}{\mathcal{V}}_{i_m}),$$
for every $i_1\neq i_2 \neq \cdots \neq i_m$. So we can view $\calH$ as follows 
\begin{align*}
    \mathcal{H} & \cong \mathcal{B} \otimes \delta_e \oplus \bigoplus_{i_1\neq i_2 \neq  \cdots \neq i_m} \mathcal{B} \otimes \ell^2(\overset{\circ}{\mathcal{V}}_{i_1})\otimes \cdots \otimes \ell^2(\overset{\circ}{\mathcal{V}}_{i_m})
\\ & \cong \calB \otimes \calK 
\end{align*}
where $\calK := \ast_{\C} \{\ell^2(\calV_i)\}_{i=1}^n$, and the last equivalence (which is clearly true at the level of $\calB$-bimodules for the algebraic direct sum) holds at the level of Hilbert $\calB$-bimodules because $\calB$ is finite-dimensional. We can then use this identification to define 
$$D:= \text{diag}(b_1, \dots, b_k)\otimes \text{Id}_{\calK}.$$
It is clear that $D\in \calA$ and that it is free with amalgamation over $\calB$ from any other element in $\calA$, as stated in \ref{mainthm:freenesswithamalgam}).  
\bigskip

\noindent \emph{Step 4: Interpreting $T$ as the Jacobi operator.} Define $T:=D+T_1+\cdots +T_n$, and for every $i, j\in [k]$ let $\edge_{ij}$ denote the operator in $\calB$ corresponding to the matrix in $M_k(\mathbb{C})$ with a 1 in the entry $(i, j)$ and 0 everywhere else. Observe that
\begin{equation}
\label{eq:matrixmult}
\edge_{ij} \edge_{lm} = \begin{cases} \edge_{im} & \text{if }j =l, \\ 0 & \text{otherwise}, \end{cases} 
\end{equation}
and hence the $l$-th column space $C_l := \spa \{\Delta_{i l} : i\in [k]\} $ is a left ideal of $\calB$ for every $l \in [k]$ (where the span is taken at the level of $\C$-vector spaces). 

 The main idea in this step is to find a $\C$-vector subspace $W$ of $\mathcal{H}$  that is $T$-invariant and such that a $\C$-basis of $W$ can be bijected with $\ast_{(\G, c)} \{\calV_i\}_{i=1}^n$ so that it is clear that $T$ acts on $W$ in the same way that the Jacobi operator acts on $\ell^2(\ast_{(\G, c)} \{\calV_i\}_{i=1}^n)$.  With this end, recall that $\calK := \ast_{\C} \{\ell^2(\calV_i)\}_{i=1}^n$ and for every $l\in [k]$ define 
 $$W_l:= C_l \otimes \calK,$$
 and consider the $\C$-basis for $W_l$ given by 
$$\Theta_l \myeq \{ \edge_{il} \otimes \delta_e  : i \in [k]\} \cup \{ \edge_{il} \otimes \delta_{v_m}\otimes \cdots \otimes \delta_{v_1} : i \in [k], v_1\in \overset{\circ}{\mathcal{V}}_{i_1}, \dots, v_m\in \overset{\circ}{\mathcal{V}}_{i_m},    i_1 \neq i_{2} \neq \cdots \neq i_m\}.  $$
Now, because $C_l$ is a left ideal of $\calB$ it is clear that $W_l$ is left-invariant under the action of each $T_j$ and $D$, and hence it is invariant under the action of $T$. Then, we can decompose $$I_k \otimes \delta_e = \sum_{i=1}^k  \Delta_{ii} \otimes \delta_e, $$
and it will follow that $ T^p(\Delta_{ll}\otimes \delta_e)\in W_l$ for every $l$ and every $p$. On the other hand
$$E[T^p] = \langle I_k\otimes \delta_e , T^p (I_k\otimes \delta_e) \rangle_{\calH} = \sum_{l=1}^k \sum_{i=1}^k \langle \Delta_{ii}\otimes \delta_e, T^p(\Delta_{ll}\otimes \delta_e) \rangle_{\calH},$$
and because of the construction of $\langle \cdot , \cdot \rangle$, and since $C_l$ is a left ideal, it is easy to see that $$\sum_{i=1}^k \langle \Delta_{ii}\otimes \delta_e, T^p(\Delta_{ll}\otimes \delta_e) \rangle_{\calH}\in C_l, \quad \forall l\in [k].$$  Similarly we can define the row space $R_i= \spa\{\Delta_{ij} : j\in [k]\}$ which is a right ideal, and use the same reasoning to conclude that $\langle \Delta_{ii}\otimes \delta_e, T^p(\Delta_{ll}\otimes \delta_e) \rangle_{\calH} \in R_i$ for every $i$. Hence, if we view $E[T^p]$ as a $k\times k$ matrix we get that
\begin{equation}
\label{eq:entryil}
E[T^p](i, l) = \langle \Delta_{ii} \otimes \delta_e, T^p(\Delta_{ll}\otimes \delta_e) \rangle_\calH.
\end{equation}
To analyze the above expression, for any word $w=v_m \cdots v_1\in \ast\{\calV_i\}$ denote $\delta_w := \delta_{v_m} \otimes \cdots \otimes \delta_{v_1}$ and note that
$$\Delta_{sl} \otimes \delta_w \mapsto (s, w)$$
is a bijection between $\Theta_l$ and the vertex set $\ast_{(\calG, c)} \{\calV_i\}$. Moreover, for every $j\in [n]$, any  $w\in \ast\{\calV_i\}$ and any  $v\in \calV_j$ it is clear that
\begin{align*}
T_j(\Delta_{sl}\otimes \delta_{vw}) & = \sum_{\alpha\in \calC_j}\, \sum_{\diredge_1\in \sigma_{\calE}(s),\, \diredge_2\in \sigma_{\calE_j}(v)} X_{\alpha}\big(\sigma(\diredge_1), \tau(\diredge_1)\big) A_{\alpha} \big( \sigma(\diredge_2), \tau(\diredge_2) \big)\cdot \Delta_{\tau(\diredge_1)l}\otimes \delta_{\tau(\diredge_2)w} 
  \\  & =  \sum_{\substack{s'\in \calV, \, v'\in \calV_j \\ (s, s')\in \calE,\, (v, v')\in \calE_j,\, c(s, v')=c_j(v, v')}} a_{(s, s')} \cdot \Delta_{s'l}\otimes \delta_{v'w} 
\end{align*}
where in the first line we have used the notation $\sigma_\calE$ and $\sigma_{\calE_j}$, to emphasize that $\sigma_{\calE}(s):=\sigma(s)\subset \calE$ and $\sigma_{\calE_j}(v):=\sigma(v)\subset\calE_j$. It is also clear that
$$D (\Delta_{sl}\otimes \delta_w) = b_s \Delta_{sl}\otimes \delta_w. $$
From the above it follows that $T$ acts on $\Theta_l$ in the same way in which $J$ acts on $\ast_{(\calG, c)} \{\calV_i\}_{i=1}^n$ as we wanted to show. 

 Therefore 
$$T^p(\Delta_{ll}\otimes \delta_e) = \sum c_{i, v_1, \dots, v_m} \Delta_{i, l} \otimes \delta_{v_m}\otimes \cdots \delta_{v_1} $$
where the sum ranges over all tuples $(i, v_m, \dots, v_1)$ for which there is a (possibly lazy) walk of length $p$ in $\ast_{(\calG, c)}\{\calG_i\}_{i=1}^n$ between the vertices $(l, e)$ and $(i, v_m \cdots v_1)$, where the coefficients $c_{i, v_1, \dots, v_m}$ are determined by the edge weights and vertex potential. So, recalling (\ref{eq:entryil}) and the definition of the inner product $\langle \cdot , \cdot \rangle_{\calH}$, we can conclude that $E[T^p](i, l)$ is the sum of weighted walks of length $p$ between $(l, e)$ and $(i, e)$, as we wanted to show.    

\subsection{Amalgamated Free Product of Groups as Amalgamated Free Product of Graphs}
\label{sec:cayley-amalgamated}

Here we show that Cayley graphs of amalgamated free products of finite groups can be realized as an amalgamated free product of finite graphs. Note that this is not an obvious fact and requires a proof. In the setup of Observation \ref{obs:cayleydecomp}, freeness with amalgamation appears because we are working in a group $C^*$-algebra of an amalgamated free group product. In contrast, in the case of the amalgamated free graph product, the freeness with amalgamation comes from the fact that we are working in a tensor algebra. 

%The following discussion uses a technique sketched in \cite[Appendix C]{kollar2019line} following Sunada \cite{sunada1992group}, and  allows us to go from the former setup to the latter. 
% \begin{lemma}
% What do we actually need
% \end{lemma}
%Fix a set of representatives $R$ of left cosets of $H$ in $G$. 

In this section we will consider finite groups $G_1, \dots, G_n$\marginpar{$G_i, H$}   with a common subgroup $H$. And we will work with the \emph{right} cosets of $H$ in each of this groups. For every $i$ we will fix $R_i$ with $e_{G_i}\in R_i$ to be a set of representatives of the right cosets of $H$ in $G_i$, and heavily exploit the following structural theorem for amalgamated graph products (see \cite[Section 1.2, Theorem 1]{serre1980sl}). 

\begin{theorem}[Structure of amalgamated free products of groups]
\label{thm:normalform}
Let $H, G_i$ and $R_i$ be as above. Then, every element $g$ of the amalgamated free product $G:= \ast_H \{G_i\}_{i=1}^n$ has a unique representation of the form
$$g=hr_m \cdots r_1$$
where $h\in H$, $r_j\in R_{i_j}\setminus \{e_{G_{i_j}}\}$ and $i_{j+1}\neq i_{j}$. This representation is called the \emph{normal form} of $g$. 
\end{theorem}

Although  we have chosen to consider \emph{right} cosets we will be interested in constructing \emph{left} Cayley graphs (the reason for this discrepancy will become apparent below). As in Section \ref{subsecfreenesswithamalg}, given a group $G$ and a generating set $S$, we denote the \emph{left} Cayley graph of $G$ with respect to $S$ by $\Gamma(G, S)$.   Now, for every $i$ fix $S_i$ a symmetric generating set of $G_i$, that is assume that for every $s\in S_i$ one also has $s^{-1}\in S_i$. The symmetry of the symmetric sets will ensure that the constructed graphs are undirected. 

Below we will show how to construct the Cayley graph $\Gamma(\ast_H \{G_i\}_{i=1}^n, S)$ for $S=\bigcup_{i=1}^n S_i$ as an amalgamated free  product of  graphs $\{\calG_i\}_{i=1}^n$ over some relator graph $(\calG, c)$. But, before describing the construction in detail we provide some intuition on why such construction should exist. First, note that for every $i$, we can view $G_i$ as the amalgamated free product of the singleton $\{G_i\}$ over $H$, and apply Theorem \ref{thm:normalform} with $n=1$ to obtain that every $g\in G_i$ can be uniquely represented as $g=hr$ with $h\in H$ and $r\in R_i$. This induces a bijection between $G_i$ and $H\times R_i$, which in turn induces an isomorphism between the vector spaces $\ell^2(G)$ and $\ell^2(H)\otimes \ell^2(R_i)$. On the other hand, the Jacobi operator on $\Gamma(G, S)$, is constructed via the operators $\lambda_s : \ell^2(G_i)\to \ell^2(G_i)$ given by the left regular representation of $G_i$. The idea is that, because $\ell^2(G_i)\cong \ell^2(H)\otimes \ell^2(R_i)$, we can view $\lambda_s$ as an operator over the latter, and decompose it as a sum of pure tensors in $B(\ell^2(H))\otimes B(\ell^2(R_i))$. Once this is done, one is precisely in the setup of the amalgamated graph product. Below, we exemplify this procedure in the case of two small  Abelian groups, and we later present the general construction. 

\begin{example} \label{ex:sl2z} 
  Here we provide an explicit construction  for the Jacobi operator on the Cayley graph of $SL(2, \mathbb{Z})$ with respect to some canonical generators. It is well-known that this group can be viewed as the free product of $\mathbb{Z}_4$ and $\mathbb{Z}_6$ with amalgamation over $\mathbb{Z}_2$, and has the finite presentation 
\[SL(2, \mathbb{Z}) \cong \langle x, y \mid x^4 = y^6 = 1, x^2 = y^3\rangle. \]  
Let $\lambda_{\mathbb{Z}_4}, \lambda_{\mathbb{Z}_6}$ denote the left regular representations of $\mathbb{Z}_4$ and $\mathbb{Z}_6$, respectively, and note that the cosets of (the inclusion of) $\bZ_2$ in $\bZ_4$ and $\bZ_6$  are $\{\{0, 2\}, \{1, 3\}\}$ and $\{\{0, 3\}, \{1, 4\}, \{2, 5\}\}$ respectively.  Then fix the generating sets $S_1=\{1, -1\}\subset \bZ_4$ and $S_2=\{-1, 1\}\subset \bZ_6$. 

Now consider the algebra inclusions (constructed as above) $\chi_1: \bC[\bZ_4]\to B(\ell^2(\bZ_2))\otimes B(\ell^2(\{0, 2\}, \{1, 3\}))$ and $\chi_2: \bC[\bZ_6] \to  B(\ell^2(\mathbb{Z}_2)) \otimes B(\ell^2(\{0, 3\}, \{1, 4\}, \{2, 5\}))$. We start by evaluating  
$\chi_1 $ at the element $\lambda_{\mathbb{Z}_4}(1)+ \lambda_{\mathbb{Z}_4}(-1)$ (which is the Jacobi operator of $\Gamma(\bZ_4, S_1)$ )  where we use the coset representatives $R = \{0, 1\}$. It is easily seen that
$$
\chi_1(\lambda_{\mathbb{Z}_4}(1)+ \lambda_{\mathbb{Z}_4}(-1)) = \left(\begin{array}{cc} 1 & 0 \\ 0 & 1 \end{array}\right) \otimes \left(\begin{array}{cc} 0 & 1 \\ 1 & 0 \end{array}\right) + \left(\begin{array}{cc} 0 & 1 \\ 1 & 0 \end{array}\right) \otimes \left(\begin{array}{cc} 0 & 1 \\ 1 & 0 \end{array}\right) .
$$

Similarly, we evaluate $\chi_2$ on $\lambda_{\mathbb{Z}_6}(1)+ \lambda_{\mathbb{Z}_6}(-1)$ where the set of coset representatives $R = \{0, 1, 2\}$ is used, to obtain
$$\chi_2(\lambda_{\mathbb{Z}_6}(1)+ \lambda_{\mathbb{Z}_6}(-1)) = \left(\begin{array}{cc} 1 & 0 \\ 0 & 1 \end{array}\right) \otimes \left(\begin{array}{ccc} 0 & 1 & 0\\ 1 & 0 & 1 \\ 0 & 1 & 0  \end{array}\right) + \left(\begin{array}{cc} 0 & 1 \\ 1 & 0 \end{array}\right) \otimes \left(\begin{array}{ccc} 0 & 0 & 1 \\ 0 & 0 & 0 \\ 1 & 0 & 0 \end{array}\right) .$$
With the above, one can  decompose  the Jacobi operator on the Cayley graph of $SL(2, \bZ)$  as a sum of two random variables that are free with amalgamation over $B(\ell^2(\bZ_2))$ by  applying the construction in Section \ref{subsecHilbertProd} to the operators obtained above. Equivalently, this gives a way to construct the Cayley graph of $SL(2, \bZ)$ as an amalgamated free product of finite graphs. 

Note that the above decompositions into pure tensors are very simple, and hence a construction of the Cayley graph via our graph product can be obtained using few colors.  This simplicity is due to the Abelian nature of the groups considered here. In general, the situation is more complicated and as we will see below more colors are needed.  
\end{example}

\paragraph{General construction.} We  produce a family of colored graphs $\{(\mathcal{G}_i, c_i)\}_{i=1}^n$ and a relator graph $(\mathcal{G}, c)$ as follows: 
\begin{enumerate}[i)]
    \item \emph{Vertices:} Set $V(\calG)=H$ and for every $i$ put $V(\calG_i)=R_i$. 
    \item \emph{Roots:} For every $i$ root $\calG_i$ at $e_{G_i}$. 
    \item \label{item:colorsets} \emph{Color sets:} For every $i$ we will choose $$\calC_i =S_i\times H \times R_i/\sim$$ 
    where $(s, h, r)\sim (s', h', r')$ if $(s, h, r) =(s', h', r')$ or if $s'=s^{-1}$ and $shr= h'r'$. 
    \item \label{item:edges} \emph{Edges:} For every $i$ and every $(h, h')\in H^2$ and $(r, r')\in R_i^2$, put a \emph{directed} edge in $\calG$ from $h$ to $h'$, and in $\calG_i$ from $r$ and $r'$, both of color $[(s, h, r)]$, if $shr = h'r'$.  
\end{enumerate}
\bigskip

\begin{proposition}
The graphs $\calG$ and $\calG_1, \dots, \calG_n$ are undirected, the colorings of the edges are well defined, and the graphs $\Gamma(\ast_{H} \{G_i\}_{i=1}^n, S)$ and $\ast_{(\calG, c)} \{\calG_i\}_{i=1}^n$ are isomorphic. 
\end{proposition}

\begin{proof}
We begin by arguing that the color sets are well defined,  and that all of the edges defined are in fact \emph{undirected} (i.e. for any directed edge that appears in the construction, its reverse edge was also added and both have the same color). To do this first we point out that the relation defined in \ref{item:colorsets}) induces a pairing of the elements of $S_i\times H \times R_i$. Indeed, by Theorem \ref{thm:normalform}, for any $(s, h, r)$ there are unique $h'\in H$ and $r'\in R_i$ such that $shr =h'r'$, and hence each triple is related to exactly one other triple. Moreover, if $(s, h, r)\sim (s', h', r')$ for $(s', h', r')\neq (s, h, r)$, then by definition we have that $shr =h'r'$ and $s'=s^{-1}$, so we also have $hr = s'h'r'$ and hence $(s', h', r')\sim (s, h, r)$, that is, $\sim$ is a symmetric relation. Because $\sim$ is symmetric and each triple is only related to one triple (other than itself), we can conclude that $\sim$ is an equivalence relation, and hence $\calC_i$ is well defined. Moreover, because $\sim$ induces a pairing in $S_i\times H\times R_i$, we are also guaranteed that for any directed edge of the form $(h, h')\in H^2$ (in $\calG$) or of the form $(r, r')\in R_i^2$ (in $\calG_i$), its inverse was also added, and both received the same color, so we can think of both of them as an undirected edge. 

Now we prove the isomorphism claim. First note that by the definition of $\ast_{(\calG, c)} V(\calG_i)$ and Theorem \ref{thm:normalform} there is a natural bijection $\phi: \ast_{(\calG, c)} V(\calG_i) \to \ast_{H} \{G_i\}_{i=1}^n$, given by
$$\phi(h, r_m \cdots r_1) = h r_m \cdots r_1. $$
It then just remains to show that $\phi$ preserves edge incidences. For this, suppose that $(h, r r_m \cdots r_1)$ and $(h', r'r_m \cdots r_1)$ are incident in the free amalgamated product of graphs (where $r$ and $r'$ are allowed to be the empty word and $m$ is allowed to be 0). Then, by definition there is some $i$ such that $ r, r' \in V(\calG_i)=R_i$, and there is an edge between $r$ and $r'$ in $\calG_i$ and between $h$ and $h'$ in $\calG$. Moreover, the aforementioned edges have the same color, which in this case means that there exists some $s\in S_i$ such that $shr = h'r'$, and therefore 
$$s\cdot \phi(h, rr_m \cdots r_1) =shrr_m\cdots r_1=h'r'r_m \cdots r_1= \phi(h', r' r_m \cdots r_1). $$
That is, any pair of adjacent vertices in $\ast_{(\calG, c)} \{\calG_i\}_{i=1}^n$ gets sent to a pair of adjacent vertices in $\Gamma(\ast_H\{G_i\}_{i=1}^n, S)$. The same argument can then be used to show that any pair of adjacent vertices in  $\Gamma(\ast_H\{G_i\}_{i=1}^n, S)$ comes from a pair of adjacent vertices in $\ast_{(\calG, c)} \{\calG_i\}_{i=1}^n$, so the proof is concluded.  
\end{proof}

\subsection{Implications for $A_{\mathcal{T}(\mathcal{G})}$}
\label{sec:rademacherrepresentation}

Here we revisit Example \ref{exunivcover} in the context of Theorem \ref{thmrefmain}. As in Example \ref{exunivcover},  we assume without loss of generality that every loop in $\calG$ is a half-loop and furthermore that $V(\calG)= [n]$. Let $a$ and $b$ be the edge weights and vertex potential of $\calG$, and let $\calT$ be its universal cover. 

Label the half-loops of $\calG$ by $\diredge_1, \dots, \diredge_l$ and its non-loop undirected edges by $\{\diredge_{l+1}, \che_{l+1}\}, \dots, \{\diredge_m, \che_m\}$.  Then define the discrete group $\Gamma_m \myeq \mathbb{Z}_2 \ast \cdots \ast \mathbb{Z}_2$ and denote its canonical generators by $g_1, \dots, g_m$. As usual $\lambda$  will denote the left regular representation of $\Gamma_m$ on $\ell^2(\Gamma_m)$. We will now see that the following result follows easily from Theorem \ref{thmrefmain}. 

\begin{proposition}
\label{prop:rademacherepresentation}
Define $X_\calG\in M_n(\bC)\otimes \cred(\Gamma_m)$ by \marginpar{$X_\mathcal{G}$}
$$X_\calG := \sum_{i=1}^n b_i \Delta_{ii} \otimes 1_{\cred(\Gamma_m)}+ \sum_{i=1}^l a_{\diredge_i} \Delta_{\sigma(\diredge_i)\tau(\diredge_i)}\otimes \lambda(g_i)+ \sum_{i=l+1}^m (a_{\diredge_i} \Delta_{\sigma(\diredge_i)\tau(\diredge_i)}+ a_{\che_i}\Delta_{\sigma(\che_i)\tau(\che_i)})\otimes \lambda(g_i) ,$$
 and let $E: M_n(\bC)\otimes \cred(\Gamma_m) \to M_n(\bC)$ be the canonical conditional expectation (see (\ref{eqconditionalexp})). Then:
\begin{enumerate}[i)]
    \item \label{item:diagonality} $E[X_{\calG}^p]$ is a diagonal matrix for all $p\in \mathbb{Z}_{\geq 0}$.  
    \item \label{item:spectralmeasures} For any $r\in [n]$, let $\func_r: M_n(\bC)\to \bC$ be defined by $\func_r(B):=B_{rr}$ for all $B\in M_n(\bC)$. Then, the spectral distribution of $X_\calG$ with respect to $\func_r\circ E$ is equal to $\mu_r$, the spectral measure of $A_{\calT}$ associated to $r$. 
    \item \label{item:DOS} The spectral distribution of $X_{\calG}$ with respect to $\frac{1}{n} \Tr \circ E$ is equal to the density of states of $A_{\mathcal{T}}$. 
\end{enumerate}
\end{proposition}

\begin{proof}
Consider the construction in Example \ref{exunivcover} and recall that for such a construction $\calT_{\text{full}}= \ast_{(\calG, c)} \{\calG_i\}_{i=1}^m$ has $n$ roots $(1, e), \dots, (n, e)$ and that the connected components $\calT_1, \dots, \calT_n$ of each are disjoint copies of $\calT$.  Lift the edge weights and vertex potential of $\calG$ to $\calTf$ as in Definition \ref{def:liftweights}. Then  apply Theorem \ref{thmrefmain} to obtain operators $D, T_1, \dots, T_m$ in an operator-valued probability space $(\calA, E, M_n(\bC))$, so that: $T_i$ distributes with respect to $E$ as  
$$a_{\diredge_i}\Delta_{\sigma(\diredge_i)\sigma(\diredge_i)} \otimes \left( \begin{array}{cc} 0 & 1 \\ 1 & 0 \end{array} \right)\in (M_n(\bC)\otimes M_2(\bC), E_i, M_n(\bC))$$
for $i, \dots l$ (i.e. when $f_i$ is a half-loop), and it distributes as
$$a_{f_i} (\Delta_{\sigma(\diredge_i)\tau(\diredge_i)}+\Delta_{\tau(\diredge_i)\sigma(\diredge_i)} ) \otimes \left( \begin{array}{cc} 0 & 1 \\ 1 & 0 \end{array} \right)\in (M_n(\bC)\otimes M_2(\bC), E_i, M_n(\bC))$$
for $i=l+1, \dots, m$ (i.e. when $\{\diredge_i, \che_i\}$ is an undirected non-loop edge), and $D$ distributes with respect to $E$ as 
$\text{diag}(b_1, \dots, b_n)\in (M_n(\bC), \mathrm{Id}, M_n(\bC))$. Moreover we know that $D, T_1, \dots, T_m$ are free with amalgamation with respect to $E$. 

Now decompose  $X_\calG$ by defining $X_\calG^D := \sum_{i=1}^n b_i\Delta_{ii}\otimes 1_{\cred(\Gamma_m)}$ and $X_\calG^{(i)} := a_{\diredge_i} \Delta_{\sigma(\diredge_i)\sigma(\diredge_i)}\otimes \lambda(g_i)$ for $i=1, \dots, l$, and $X_\calG^{(i)} := a_{\diredge_i}(\Delta_{\sigma(\diredge_i)\tau(\diredge_i)}+ \Delta_{\tau(\diredge_i)\sigma(\diredge_i)} ) \otimes \lambda(g_i)$ when $i=l+1, \dots, m$. Now, from Observation \ref{obs:matrixfreeness} we know that the family  $\{X_\calG^D, X_\calG^{(1)}, \dots, X_\calG^{(m)} \}$ is free with amalgamation over $M_n(\bC)$. Moreover, it is clear that $D$ is equal in distribution (over $M_n(\bC)$) to $X_{\calG}^D$, and similarly each $T_i$ is equal in distribution to $X_{\calG}^{(i)}$. 

Hence, $X_\calG$ is equal in distribution over $M_n(\bC)$ to $T=D+T_1+\cdots T_m$. Then \ref{item:spectralmeasures}) follows directly from Theorem \ref{thmrefmain} \ref{mainthm:pathcounting}). On the other hand, because $\frac{1}{n}\Tr = \frac{1}{n} \sum_{r\in [n]} \func_r$, we get \ref{item:DOS}) from \ref{item:spectralmeasures}).  Finally, to show \ref{item:diagonality}) take $i\neq j$ and use that $E[T^p](i, j) = E[X_\calG^p](i, j)$ combined with Theorem \ref{thmrefmain} \ref{mainthm:pathcounting}), to conclude that $E[X_\calG^p](i, j)$ is equal to the sum of the weighted paths of length $p$ in $\calTf$ from $(i, e)$ to $(j, e)$, but since these vertices are in distinct connected components we can conclude that $E[X_\calG^p](i, j)=0$ as we wanted to show. 
\end{proof}

\begin{remark}[Distinct $C^*$-algebra representations]
Note that we have so far presented two essentially different methods to view $A_\calG$ as an element of a $C^*$-algebra. First, in Section \ref{secnumberofbands} we used asymptotic freeness of random matrices to argue that $A_\calT$ could be viewed as an element in $M_n(\bC)\otimes \cred(\mathbb{F}_m)$. Here, we have used the machinery from amalgamated graph products to show that $A_{\calT}$ can be viewed as an element in $M_n(\bC)\otimes \cred(\Gamma_m)$. Although each representation and each  proof are insightful in its own way, note that having access to distinct $C^*$-algebras is also relevant in applications. The representation in $M_n(\bC)\otimes \cred(\mathbb{F}_m)$ was exploited in Section \ref{sec:sunadastheorem} to prove Sunada's theorem, where the $K$-theory of $\cred(\mathbb{F}_m)$ (which is different from that of $\cred(\Gamma_m)$) played a crucial role. On the other hand in Section \ref{secalgebraicuniversal} we will use the representation in $M_n(\bC)\otimes \cred(\Gamma_m)$, which will allow us to apply verbatim some results of Lehner \cite{lehner1999computing} about the norm and Cauchy transform of elements in this $C^*$-algebra. 
\end{remark}

\section{Universal Covers: Algebraic Description of the Spectrum}
\label{secalgebraicuniversal}

In this section we give two applications of Proposition \ref{prop:rademacherepresentation}, where it was shown that $A_\calT$ can be represented as an element in $M_n(\bC)\otimes \cred(\Gamma_m)$. An important ingredient for our proofs will be the work of Lehner \cite{lehner1999computing}, which has formulas for the operator-valued $R$-transform and norm of certain elements in   $M_n(\bC)\otimes \cred(\Gamma_m)$. In what follows we maintain the same setup and notation defined at the beginning of Section \ref{sec:rademacherrepresentation}.

\subsection{Aomoto's Equations via the $R$-transform}

Here we will recover Aomoto's system of equations (i.e. Theorem $\ref{thm:system}$) using free probability. We begin by noting that if we apply \cite[Proposition 3.1]{lehner1999computing}, and specialize from positive definite matrices to diagonal positive definite matrices, we get the following result: 
\begin{proposition}[Lehner] \label{prop:lehner}
Let $W\in M_n(\mathbb{C})$ be a diagonal  matrix with positive entries and let $w_i \myeq W_{ii}$. Then the $M_n(\mathbb{C})$-valued $R$-transform of $X_{\mathcal{G}}$ at $W$, i.e. $R_{X_{\mathcal{G}}}(W)$, is also diagonal, and the diagonal entries are given by 
$$R_{X_{\mathcal{G}}} (W) (i, i) = b_i + \frac{1}{2w_i} \bigg( \sum_{\diredge\in \sigma(i) } \big(1+4a_\diredge^2 w_{\sigma(\diredge)}w_{\tau(\diredge)}\big)^{1/2} -1 \bigg), \quad \forall i \in [n].   $$
\end{proposition}

In order to use the above proposition we will need the following observation, which is a corollary of Proposition \ref{prop:rademacherepresentation}. 

\begin{observation}
\label{obs:diagonalcauchy}
Let $G_{X_\mathcal{G}}$ denote the $M_n(\mathbb{C})$-valued Cauchy transform of $X_\mathcal{G}$. Then for  $z \in \mathbb{C}$ in a neighborhood of infinity, $G_{X_\mathcal{G}}(zI_n)$ is diagonal. 
\end{observation}

\begin{proof}
We will use the following power series expansion (in a neighborhood of infinity) for the Cauchy transform 
$$G_{X_\mathcal{G}}(B) = \sum_{p\geq 0} E[B^{-1}(X_{\mathcal{G}}B^{-1})^p].$$
For $B= zI_n$, the terms in the right-hand side of the above equation will be of the form $z^{-(p+1)} E\big[X_\mathcal{G}^p\big]$. Proposition \ref{prop:rademacherepresentation} \ref{item:diagonality}) then implies that each term of the series expansion of $G_{X_\calG}$ is a diagonal matrix, so the claim follows.  
\end{proof}

We can now show Theorem \ref{thm:system}.

\begin{proof}[Proof of Theorem \ref{thm:system}]
As in the statement, for every $i\in [n]$ let $w_i(z)$ be the Cauchy transform of $\mu_i$, the spectral measure of $A_\mathcal{T}$ corresponding to $i$. Let $W(z) \myeq G_{X_\G}(zI_n)$, and note that from Observation \ref{obs:diagonalcauchy} we know that $W(z)$ is diagonal. Moreover, from Observation \ref{obsCauchy} we have that 
$$w_i(z) = W(z)(i, i).$$
Moreover, because they are Cauchy transforms, for any $i$ we have that $w_i(z) > 0$ for all sufficiently large real $z$. Using the definition of the $R$-transform in (\ref{eqn:defrtrans}) we obtain 
\[ z w_i(z) = 1 + R_{X_{\mathcal{G}}} (W) (i, i) w_i(z)  \]
for all $1 \le i \le n$.  After substituting in the expression for $R_{X_{\mathcal{G}}}$ from Proposition \ref{prop:lehner} and simplifying, we get that the system of equations in the theorem statement holds for sufficiently large positive real $z$.  Since $w_u(z)$ is holomorphic and $w_u(z) \to 0$ as $|z| \to \infty$, both sides of the equations are holomorphic in a neighborhood of infinity, so by analytic continuation the system of equations holds in a neighborhood of infinity.  Since $w_i(z) w_j(z) > 0$ for all $i,j$ when $z$ is real and outside the convex hull of the spectrum $\Spec(A_\T)$, the system of equations holds for these $z$ as well, as the singularity of the square root is always avoided.
\end{proof}

\begin{remark}[Half-loops are allowed]
Although (for exposition purposes) we stated Theorem \ref{thm:system} for graphs without half-loops note that the above proof does allow half-loops, and the statement is left unchanged.  
\end{remark}

\begin{remark}[Understanding the density of states]
 In \cite{aomoto1991point} Aomoto used this system of equations to prove results about the point spectrum of $A_\calT$ (equivalently the atoms in the density of states). In Appendix \ref{app:densityattheedge} we show that this system of equations can also be used to understand the behavior of the density of states  at the edge of $\Spec(A_\calT)$. Specifically, we prove Theorem \ref{thm:vanishingdensity} stated in the introduction. 
\end{remark}

\subsection{Formula for the Spectral Radius}
\label{sec:spectralradius}

Let $\spr(A_\mathcal{T})$ denote the spectral radius of $A_\mathcal{T}$ and $\rho_r$ denote the right edge of $\Spec(A_\calT)$. This subsection will be devoted to  proving Theorem \ref{thm:radius}. We will build on  \cite[Theorem 1.1]{lehner1999computing} and, as above, we will use $g_1, \dots, g_m$  to denote the canonical generators in $\Gamma_m$.

\begin{theorem}[Lehner]
\label{thm:lehnernorm}
Assume that $m \geq 2$, and let $A_0, \dots, A_m$ be $n\times n$ Hermitian matrices, with $A_0$ positive semidefinite. Then 
\begin{equation}
\label{eq:Lehnerradius}
\bigg\| A_0\otimes 1_{C_{\mathrm{red}}^*(\Gamma_m)}+ \sum_{i=1}^m A_i \otimes \lambda(g_i) \bigg\| = \inf_{Z>0} \bigg\| {2Z}+A_0+ \sum_{i=1}^m Z^{\frac{1}{2}}\left((I_n+(Z^{-\frac{1}{2}}A_i Z^{-\frac{1}{2}})^2)^{\frac{1}{2}}-I_n \right) Z^{\frac{1}{2}}\bigg\|
\end{equation}
where the infimum is taken over all positive definite invertible $n\times n$ matrices $Z$. Moreover, the infimum can be restricted to those $Z$ for which the expression inside the norm sign equals a positive scalar multiple of the identity matrix $I_n$. 
\end{theorem}

We are now ready to prove Theorem \ref{thm:radius}. In short, the proof %combines Lemma \ref{lem:equalspectralrads} and 
uses Aomoto's equations (Theorem \ref{thm:system}) to reduce the expression (\ref{eq:Lehnerradius}).

% \begin{proof} For $i=1, \dots, n$ define \[g_i( y_1,  \dots, y_n) =   \frac{1}{2y_i} \left( 2-\deg(i)  + \sum_{j:\, \{i, j\}\in E(\mathcal{G})} \sqrt{1+4a_{\{i, j\}}^2 y_i y_j} \right). \]
% Let $\rho^+(A_\T)$ denote the right edge of the spectrum of $A_\T$.  We will show that \[\rho^+(A_\T) = \inf_{y_1, \dots, y_n>0} \max_{i \in [n]} (b_i + g_i(y_1, \dots, y_n)).\]

% To see how this implies the theorem, note that $\rho(A_\T) = \max\{\rho^+(A_\T), \rho^+(-A_\T)\}$, and that for each $y_1, \dots, y_n > 0$, we have 
% \[ \max\left\{ \max_{i \in [n]} (-b_i + g_i(y_1, \dots, y_n)), \max_{i \in [n]} (b_i + g_i(y_1, \dots, y_n)) \right\} = \max_{i \in [n]} (|b_i| + g_i(y_1, \dots, y_n)). \]
% %\[ \max\{  |b_i| + \frac{1}{2y_i} \left( 2-\deg(i)  + \sum_{j:\, \{i, j\}\in E(\mathcal{G})} \sqrt{1+4a_{\{i, j\}}^2 y_iy_j} \right),  |b_i| + \frac{1}{2y_i} \left( 2-\deg(i)  + \sum_{j:\, \{i, j\}\in E(\mathcal{G})} \sqrt{1+4a_{\{i, j\}}^2 y_iy_j} \right) \} =  |b_i| + \frac{1}{2y_i} \left( 2-\deg(i)  + \sum_{j:\, \{i, j\}\in E(\mathcal{G})} \sqrt{1+4a_{\{i, j\}}^2 y_iy_j} \right)\]
% %\myeq \frac{1}{2y_i} \left( 2-\deg(i) + \sum_{j:\, \{i, j\}\in E(\mathcal{G})} \sqrt{1+4p\{i, j\}^2 y_iy_j} \right)$. 

\begin{proof} For $i = 1, \dots, n$ define $g_i(y_1, \dots, y_n)$ to be the expression inside the $\max$ symbol in the theorem statement, and let $w_i(z)$ denote the Cauchy transform of $\mu_i$ (the spectral measure of $A_\mathcal{T}$ associated to $i$). Fix $t> \rho_r(T)$, and observe that $\infty >w_i(t)> 0$  for every $i\in [n]$. On the other hand, from Theorem \ref{thm:system} we have $t= g_i(w_1(t), \dots, w_n(t))$ for every $i$. Together, this implies that
$$t \geq \inf_{y_1, \dots, y_n > 0} \max_{i\in[n]} g_i(y_1, \dots, y_n) . $$
Since the above inequality holds for any $t > \rho_r(A_\mathcal{T})$, it holds for $t = \rho_r(A_\T)$. It remains to show the opposite inequality.

From Proposition \ref{prop:rademacherepresentation} we have that $\spr(A_\mathcal{T}) = ||X_{\mathcal{G}}||$. We would like to apply Theorem \ref{thm:lehnernorm} on $X_\G$, but the theorem requires $A_0$ to be positive semidefinite.  To remedy this, take $\lambda \ge 0$ large enough so that $\lambda + b_i \ge 0$ for all $i$; we may now apply Theorem \ref{thm:lehnernorm} on $X_\mathcal{G}+\lambda$ to obtain \begin{equation}
\label{eq:lehnerforthiscase}
\spr(A_\mathcal{T} + \lambda) = \|X_\mathcal{G} + \lambda I_n\| = \inf_{Z>0} \bigg\| 2Z + D + \lambda I_n + \sum_{i=1}^m Z^{\frac{1}{2}}
\Big(\big(I_n + a_{\diredge_i}^2 (Z^{-\frac{1}{2}} A_i  Z^{-\frac{1}{2}})^2\big)^{\frac{1}{2}}-I_n\Big)Z^{\frac{1}{2}}  \bigg\| 
\end{equation}
where $D = \operatorname{diag}(b_1, \dots, b_n)$, $A_i= \Delta_{\sigma(\diredge_i)\sigma(\diredge_i)}$ for $i=1, \dots, l$, and $A_i = \Delta_{\sigma(\diredge_i)\tau(\diredge_i)}+ \Delta_{\tau(\diredge_i)\sigma(\diredge_i)} $ for $i=l+1, \dots, m$.  We will now see that in this case the  infimum is achieved by diagonal matrices.  Let $y_1, \dots, y_n >0$ and take $Y \myeq \mathrm{diag}(1/2y_1, \dots, 1/2y_n)$. Simple computations yield that upon setting $Z = Y$ in (\ref{eq:lehnerforthiscase}), the quantity inside the norm on the right hand side becomes
\begin{equation}
\label{eq:diagforg's}
%2Y+ B+ \sum_{e=\{i, j\} \in E(\mathcal{G})} Y^{\frac{1}{2}} ((I_n + a_e^2(Y^{-\frac{1}{2}} (\edge_{ij}+\edge_{ji}) Y^{-\frac{1}{2}})^2)^{\frac{1}{2}}-I_n)Y^{\frac{1}{2}}= 
\lambda I_n + \mathrm{diag}(g_1(y_1, \dots, y_n), \dots, g_n(y_1, \dots, y_n)).
\end{equation}
Since  $\lambda + g_i(y_1, \dots, y_n)\geq 0$ for all $i$, we have  \[|| \lambda I_n + \mathrm{diag}[g_1(y_1, \dots, y_n), \dots, g_n(y_1, \dots, y_n)] || = \lambda + \max_{i\in [n]} g_i(y_1, \dots, y_m).\] Then, (\ref{eq:lehnerforthiscase}) and (\ref{eq:diagforg's}) yield
\[ \spr(A_\T + \lambda) \le \lambda + \inf_{y_1, \dots, y_n > 0} \max_{i\in[n]} g_i(y_1, \dots, y_n). \]
Since $\spr_r(A_\T) + \lambda = \rho_r(A_\T + \lambda) \le \spr(A_\T + \lambda)$, we have \[ \rho_r(A_\T) \le \inf_{y_1, \dots, y_n > 0} \max_{i\in[n]} g_i(y_1, \dots, y_n) \] as desired.

% $$\inf_{y_1, \dots, y_n > 0} \max_{i\in[n]} g_i(y_1, \dots, y_n)\geq \rho(\mathcal{T}),$$
% which concludes the proof of (\ref{eq:rho}). 

% -------------------------scratch---------------------------

% For each $\lambda$ with $\lambda + b_i \ge 0$ for all $i$, i.e. $\lambda \ge - \min_i b_i$, we get \[\rho(T + \lambda) \le \inf_y \max_i | \lambda + g_i(y)| = \inf_y \max_i ( \lambda + g_i(y)) = \lambda + \inf_y \max_i g_i(y).\]
% So in the case where $\rho(T + \lambda) = \rho(T) + \lambda$, i.e. the case where the spectral radius is achieved on the right side, we're done.
\end{proof}

%\begin{proof}[Proof of Theorem \ref{thm:radius}]
Applying Lagrange multipliers to the optimization problem (\ref{eq:rho}), using the constraint that all $n$ expressions inside the $\max$ symbol are equal, we have the following corollary:

\begin{corollary} \label{cor:lagrange}
With the above setup and notation,  $t=\rho_r$ is the only real number such that, under the constraint $y_1, \dots, y_n >0$,  the following system of $2n+1$ equations in the variables $t, y_1, \dots, y_n, \lambda_1, \dots, \lambda_n \in \mathbb{R}$ has a solution:
$$\begin{cases}
    1 = \displaystyle\sum_{i=1}^n \lambda_i, \\
    \lambda_i (t - b_i) = \displaystyle\sum_{\diredge\in \sigma(i)} \displaystyle a_\diredge^2\, \frac{y_{\tau(\diredge)} \lambda_{\sigma(\diredge)} + y_{\sigma(\diredge)} \lambda_{\tau(\diredge)}}{\big(1+4 a_\diredge^2y_{\sigma(\diredge)} y_{\tau(\diredge)}\big)^{1/2}} & \forall i\in [n], \\
    t = b_i + \displaystyle\frac{1}{2y_i} \bigg( 2-\operatorname{deg}(i) +\displaystyle\sum_{\diredge\in \sigma(i)} \big(1+4a_\diredge^2y_{\sigma(\diredge)}y_{\tau(\diredge)}\big)^{1/2} \bigg)  & \forall i \in [n], 
\end{cases} $$
where half-loops count once towards the count in $\mathrm{deg}(i)$.
\end{corollary}
%\end{proof}

%\TODO remove nocite cmd
%Remove this line to remove non-cited bib entries
%\nocite{*}
%%%%%%%%%%%%%%%%%%%%%%%%%%%%%%%%%%%%%%%%%%%%%%%%

\section{Future Research}
\label{sec:future}

Since the first version of the present paper has been made public, some of the questions discussed in this section have been answered. Below we discuss  relevant recent work and highlight the questions that we still believe to be open and interesting. 

\paragraph{Amalgamated Free Product for Graphs.} In the first version of this paper we pointed out similarities and differences between our graph product and the \emph{additive graph product} introduced by Mohanty and O'Donnell in \cite{mohanty2019x}, and asked for a complete characterization of the graphs that could be constructed using each of these products. This has been answered in \cite{o2020explicit} by O'Donnell and Wu, where a new graph product was introduced (with the purpose of exploiting results in \cite{bordenave2019eigenvalues} to generate relative expanders), and was shown to generalize the amalgamated free product of graphs and the additive product. Moreover, the authors gave a full characterization of the graphs that could be obtained using their product. 

On a different direction, we would like to recall questions about algebraicity and absence of singular continuous spectrum for  general classes of graphs (such as the ones that can be obtained using the amalgamated graph product). For certain operators on different classes of infinite graphs, the respective Green functions (or other related functions)  have been shown to be algebraic \cite{woess1987context,  lalley2001random, nagnibeda2002random, keller2013spectral,  avni2020periodic}. On the other hand, algebraicity is relevant in the context of spectral theory since it provides means to show that the operators in question have no singular continuous spectrum \cite{  avni2020periodic} or in some cases that the spectrum is purely absolutely continuous \cite{keller2013spectral}. We believe that Jacobi operators on any graph defined via our product also have no singular continuous spectrum, and it is possible that the work of Anderson \cite{anderson2014preservation} might lead to showing algebraicity of their Green functions. Proving this would provide a generalization of Theorem 6.7 in \cite{avni2020periodic} and other results in the literature of spectral analysis of Cayley graphs. We state this as a conjecture:

\begin{conjecture}[Algebraicity and absence of singular continuous spectrum]
The Cauchy transforms of the spectral measures of any Jacobi operator on a graph constructed via the amalgamated free product of graphs are algebraic, and the Jacobi operator has no singular continuous spectrum.
\end{conjecture}

The above could also be relevant in relation to numerical computations. In particular, we point out that an operator-valued analog of \cite{rao2008polynomial} can lead to accurate numerical method for computing the spectral measures of any graph obtained via the amalgamated free product of graphs.

\paragraph{Universal Covering Graphs.} In \cite{avni2020periodic} and in previous versions of this paper some questions about the point spectrum of $A_\calT$ were asked. These questions were answered in \cite{banks2020point}, where a full characterization of the point spectrum of $A_{\calT}$ was provided by analyzing the combinatorial structure of its eigenvectors. Later in \cite{arizmendi2021universality}, using free probability tools, a general theory about atoms of spectral measures of polynomials in non-commutative random variables was developed, and some of the results in \cite{banks2020point} can be obtained directly from this general theory.  

On a different direction, we bring to the  attention of the reader that many of the questions about $m$-functions of $A_{\calT}$ raised in \cite{avni2020periodic} remain open (in particular see \cite[Section 10.1]{avni2020periodic}). On the other hand, in the same way in which we deduced Aomoto's systems of equations using the $R$-transform, we believe that results about the $m$-functions of $A_{\calT}$ can be obtained via the theory of subordination in free probability \cite{voiculescu1993analogues, biane1998processes, belinschi2007new, belinschi2017analytic}. Moreover, the subordination approach would yield numerical methods for  computing the density of states of $A_{\calT}$ via fixed point equations (see \cite{belinschi2017analytic, helton2018applications}). 

Finally, we highlight that, despite recent activity in the area, simple questions about the number of bands in the spectrum of $A_\calT$ seem to be out of reach of current tools. For example, we believe that the following toy problems are quite challenging. 

\begin{problem}
Find a characterization of  the base graphs $\G$ and coefficients $a_\diredge, b_v$ for which $A_{\T}$ has connected spectrum. 
\end{problem}

\begin{problem}
Find an infinite sequence of Jacobi matrices $A_{\G_n}$ on graphs $\G_n$  (without half-loops) with $|V(\G_n)| \to \infty$ as $n \to\infty$, and with $b_v^{(n)} = 0$ for every $v \in V(\G_n)$, such that for every $n$, $\G_n$ has at least two cycles and the spectrum of $A_{\T(\G_n)}$ has $|V(\G_n)|$ bands.\footnote{The simpler problem in which the condition $b_v^{(n)} = 0$   is removed has a simple solution \cite{avni2020periodic}; see Remark \ref{rem:tight} above. } 
\end{problem}

\section*{Acknowledgements}

We  thank Nikhil Srivastava and Dan-Virgil Voiculescu for many helpful discussions and suggestions of research directions. We also  thank  Peter Sarnak for a very helpful and encouraging conversation we had at the Simons Institute. We thank Barry Simon for his helpful feedback on a previous version of this paper, for pointing out key references and for helpful discussions.   We are also grateful to Nalini Anantharaman, Octavio Arizmendi, Benson Au, Hari Bercovici, Shai Evra and Sidhanth Mohanty for helpful discussions. 

 \bibliographystyle{alpha}
\bibliography{references}

\appendix

\section{Spectral Splitting}
\label{subsec: bandexamples}

Theorem \ref{thm:mass} provides an upper bound to the number of bands in the spectrum of the universal cover, and we have provided examples of when this bound is tight above.  A natural further question is to determine what properties of a graph result in one band, two bands, and so on, as the number of bands has some relevance in physics (see \cite{avni2020periodic}).  In this section, we discuss some interesting examples yielding various numbers of bands.

One way to vary the number of bands in the spectrum of the universal cover is to fix the base graph $\mathcal{G}$ and vary the edge weights $a_\diredge$.  For the specific case of regular trees, Fig\'a-Talamanca and Steger \cite{figa1994harmonic} gave an explicit description of this phenomenon. Here, to make it compatible with our context, the theorem below paraphrases Lemma 1.4 in Chapter 2 of  \cite{figa1994harmonic}. 

\begin{theorem}[Fig\'a-Talamanca, Steger] \label{thm:fts}
Let $\mathcal{G}$ be the graph with two vertices $u, v$ and $d$ parallel edges $\diredge_1, \dots , \diredge_d$ connecting them. Assume that $a_{\diredge_1} \geq \cdots \geq a_{\diredge_d} > 0$ and    $b_u=b_v =0$. Then, zero is in the spectrum of $A_{\mathcal{T}(\mathcal{G})}$ if and only if 
\begin{equation}
\label{eq: FTSweights}
    a_{\diredge_1}^2 \leq \sum_{i=2}^d a_{\diredge_i}^2. 
\end{equation}
\end{theorem}

Let $\mathcal{G}$ be as in the above theorem. Note that from Theorem \ref{thm:mass} it follows that the spectrum of $A_{\T(\G)}$ has at most two bands. Combining this with the fact that the spectrum is symmetric about zero (as $\T$ is bipartite), we have the following: 

\begin{observation}
Using the notation of Theorem \ref{thm:fts}, $A_{\T(\G)}$ has a connected spectrum if and only if inequality (\ref{eq: FTSweights}) is satisfied.  Moreover, this remains true if we add a constant to $b_u$ and $b_v$.
\end{observation} 

%%%
We pause to mention some examples.  If $d=2k+1$ and the $a_{\diredge_i}$ are chosen to be distinct and in such a way that (\ref{eq: FTSweights}) holds, we get a $(2k+1)$-regular tree with non-constant coefficients and connected spectrum.  On the other hand, if $d =2k$ and the $a_{\diredge_i}$ are chosen with the same characteristics as above, we get that $A_{\mathcal{T}(\G)}$ has connected spectrum\footnote{These two examples answer Conjectures 9.6 and 9.7 in \cite{avni2020periodic} in the negative, and we thank Barry Simon for pointing out that the case $d=3$ negatively answers their Conjecture 10.5.}.
%%%

More can be said about the graph $\G$ from Theorem \ref{thm:fts}:

\begin{proposition}
Using the notation of Theorem \ref{thm:fts}, if $b_u$ and $b_v$ are instead taken to be arbitrary distinct reals, then the spectrum of $A_{\T}$ has two bands.
\end{proposition}
\begin{proof}
Applying Lemma \ref{lem:repofATinfreegroup}, in particular we obtain that the spectrum of $A_{\calT}$ is the spectrum of an operator-valued matrix of the form
\[\begin{pmatrix} b_1 & x \\ x^* & b_2 \end{pmatrix} \in M_2(\mathbb{C}) \otimes C^*_{\text{red}}(F_d).\]  Since $b_1 \ne b_2$ by assumption, let us subtract a suitable multiple of the identity to obtain \[X \myeq \begin{pmatrix} t & x \\ x^* & -t \end{pmatrix}\] for some real $t$; this merely translates the spectrum.  Note that the spectrum of $X$ is symmetric about zero, as $U X U^{-1} = -X$ for the unitary $U = \begin{pmatrix} 0 & -1 \\ 1 & 0 \end{pmatrix}$.\footnote{We thank Barry Simon for this observation.}  Since $X^2 = \begin{pmatrix} y & 0 \\ 0 & y^* \end{pmatrix}$ where $y = xx^* + t^2$ is invertible, 0 is not in the spectrum of $X$. Thus the spectrum has a gap, as desired.
\end{proof}
%Theorem \ref{thm:fts} is interesting for several reasons. First, in the case of $d=3$, the  observation in the previous paragraph implies that the set of parameters (in $\R^5$)  for which the spectrum is connected  is of dimension 4. This gives a negative answer to Conjecture 10.5 in \cite{avni2020periodic}\footnote{We thank Barry Simon for pointing out that Theorem \ref{thm:fts} provided an answer to this conjecture.}, where the aforementioned set was conjectured to always be of codimension 2.  However, it is still possible that this set is always of codimension at least 1. 

%Actually,  the above discussion also provides answers to other  questions raised in \cite{avni2020periodic} regarding the possibility of extending Borg's theorem to Jacobi operators on universal covers. If $d=2k+1$ and the $a_{e_i}$ are chosen to be distinct and in such a way that (\ref{eq: FTSweights}) holds, we get an example of a $(2k+1)$-regular tree with non-constant coefficients and connected spectrum. This gives a negative answer to Conjecture 9.6 in \cite{avni2020periodic}. On the other hand, if $d =2k$ and the $a_{e_i}$ are chosen with the same characteristics as above, we get have that $A_{\mathcal{T}(\G)}$ has connected spectrum and it cannot be viewed as the Jacobi operator of a universal cover of a graph with only one vertex and no half-loops. This disproves Conjecture 9.7 in \cite{avni2020periodic}. 

Now we show that non-constant-degree universal covers with  similar characteristics can be constructed. To this end, we will now use $\mathcal{G}$ to denote the graph \includegraphics{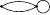} consisting of two vertices $u, v$, with a loop $\diredge_1$ on $u$ and two parallel edges $\diredge_2, \diredge_3$ connecting $u$ and $v$. Put $b_u=b_v =0$ and assume $a_{\diredge_i}>0$. 

 \begin{figure}[h]
  \centering
 \begin{tabular}{ll}
\raisebox{3.67\height}{\includegraphics[scale=3.3]{counterexgraph_cropped.pdf}} \qquad \qquad \qquad  & \includegraphics[scale=0.5]{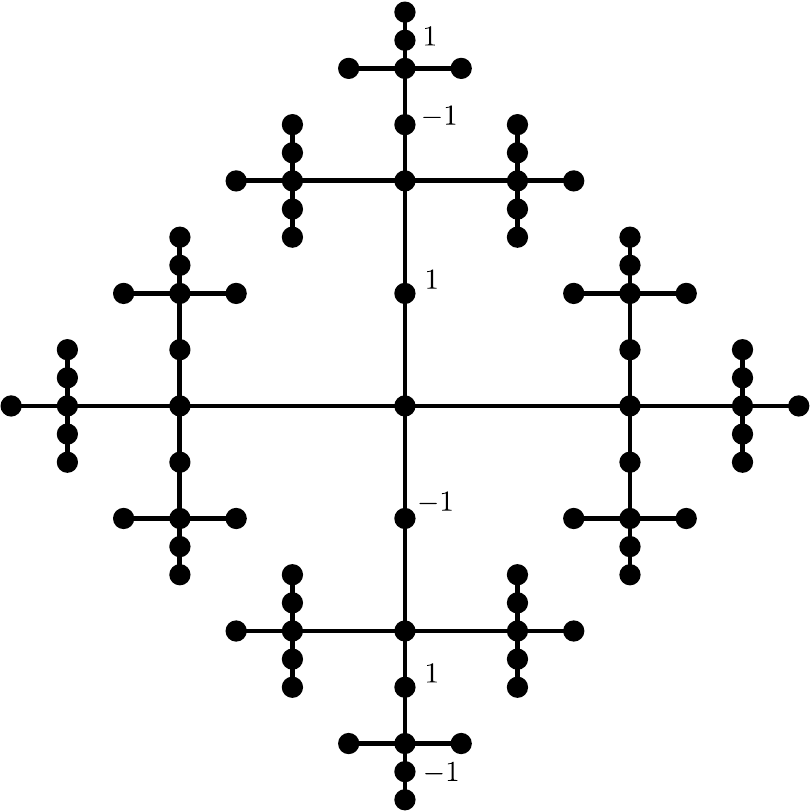} \\
\end{tabular}

 \caption{The graph $\G$ and a finite portion of its universal cover $\calT$, along with an approximate eigenvector for zero.}  
  \label{fig:counterex_tree}

 \end{figure}

As in the example of regular trees, since $\mathcal{G}$ has two vertices,  we also have that the spectrum of $A_{\mathcal{T}(\G)}$ is connected if and only if it contains zero. By Lemma \ref{lem:repofATinfreegroup}, determining the invertibility of $A_{\calT}$ is equivalent to deciding if the following operator-valued matrix is invertible:
$$X \myeq \left(\begin{array}{cc}
    a_{\diredge_1}(\lambda(g_1)+ \lambda(g_1)^*) & a_{\diredge_2}\lambda(g_2)+ a_{\diredge_3}\lambda(g_3)   \\
       a_{\diredge_2}\lambda(g_2)^*+ a_{\diredge_3}\lambda(g_3)^* & 0 
\end{array}\right) \in M_2(\C)\otimes C_{\mathrm{red}}^*(\mathbb{F}_3), $$
where $g_1, g_2$ and $g_3$ are the canonical generators of $\mathbb{F}_3$. 

We separate our analysis into two cases. First assume that $a_{\diredge_2}\neq a_{\diredge_3}$. We can further assume without loss of generality that $a_{\diredge_2} > a_{\diredge_3}$. Note that 
\begin{equation}
\label{eq:invertibility}
 a_{\diredge_2}\lambda(g_2)+ a_{\diredge_3}\lambda(g_3) = a_{\diredge_2} \lambda(g_2) \left(1+ \frac{a_{\diredge_3}}{a_{\diredge_2}} \lambda(g_2^{-1} g_1) \right) . 
\end{equation}
From $a_{\diredge_2} > a_{\diredge_3}$ we have $\|\frac{a_{\diredge_3}}{a_{\diredge_2}} \lambda(g_2^{-1} g_1)\| < 1$, which implies that $1+ \frac{a_{\diredge_3}}{a_{\diredge_2}} \lambda(g_2^{-1} g_1)$ is invertible, and in view of (\ref{eq:invertibility}) this implies that $a_{\diredge_2}\lambda(g_2)+ a_{\diredge_3}\lambda(g_3)$ is also invertible.  Set $x \myeq (a_{\diredge_2}\lambda (g_2)+a_{\diredge_3}\lambda (g_{3}))^{-1}$ and
$$Y = \left(\begin{array}{cc}
   0  & x^* \\
 x    & -a_{\diredge_1} x (\lambda(g_1)+\lambda(g_1)^*) x^*
\end{array} \right).$$
It is easy to see that $X Y = I_2 \otimes 1$ and hence that zero is not in the spectrum of $A_{\calT}$ when $a_{\diredge_2}\neq a_{\diredge_3}$. 

Now assume that $a_{\diredge_2}=a_{\diredge_3}$. In this case one can show that $A_{\calT}$ is not invertible.  Indeed, let $x_n$ be the vector whose entries alternate between $1$ and $-1$ along $n$ consecutive degree-two vertices on the $y$-axis of $\calT$ as depicted in Figure \ref{fig:counterex_tree}.  Then $\Vert x_n \Vert = \sqrt{n}$ while $\Vert A_{\calT} x_n \Vert = \sqrt{2}$, so $\{x_n\}$ is a sequence of approximate eigenvectors for zero.  Alternatively, one can show that $X$ is not invertible by noting that $X(1, 2)$ is a scalar multiple of $\lambda(g_2)+\lambda(g_3)$, which is not invertible since its spectrum is  $\{z \in \C: |z|\leq \sqrt{2}\}$ (see Example 5.5 in \cite{haagerup1999brown}). This means that both $X(1, 2)$  and $X(2, 1)$ are not invertible, and hence $X$ is not invertible. 

This discussion can be summarized as follows. 

\begin{example}
\label{ex:twovertex}
 Let $\mathcal{G}$ be the graph \includegraphics{counterexgraph_cropped.pdf} consisting of two vertices $u, v$, with a loop $\diredge_1$ on $u$ and two parallel edges $\diredge_2, \diredge_3$ connecting $u$ and $v$. Put $b_u=b_v =0$ and assume that $a_{\diredge_i}>0$ for $i=1, 2, 3$. Then, if $a_{\diredge_1} = a_{\diredge_2}$, the spectrum of $A_{\calT}$ is connected; otherwise it has two bands. 
\end{example}

This example disproves Conjecture 9.5 in \cite{avni2020periodic}, since it provides a graph $\G$ of non-constant degree and specific coefficients for which   $A_{\calT}$ has a connected spectrum. 

Finally, in \cite{avni2020periodic} an interesting conjecture was made regarding the possibility of generalizing the Borg-Hochstadt theorem. Roughly speaking, in the language of our work, it was conjectured that for an arbitrary universal cover $\T$,  if the cumulative distribution function of the density of states of $A_\T$  is of the form $j/p$ inside every gap, then there exists a quotient of $\T$, say $\G$, such that $|V(\G)|$ is a divisor of $p$.\footnote{Actually, the conjecture was stated in terms of the notion of period discussed in \cite{avni2020periodic}, where an ultimate definition of period was left open. However,  there does not seem to be a sensible definition of period that rules out the counterexamples presented here.} In some sense, Example \ref{ex:twovertex} is already a counterexample of this conjecture, since in this case when $a_{\diredge_1}=a_{\diredge_2}$ there is only one band in the spectrum, while the smallest quotient of the universal cover has two vertices. However, it is still of interest to find an example with an interior spectral  gap where the extension of  the Borg-Hochstadt theorem does not hold. We do this below\footnote{We should mention that our example consists of a graph with a leaf, while in \cite{avni2020periodic} only leafless graphs were considered.}. 

Let $m\geq 4$ and let $g_1, \dots, g_m$ be the canonical generators of $\mathbb{F}_m$. Take $x \in C_{\mathrm{red}}^*(\mathbb{F}_m)$  self-adjoint. If $x$ is invertible then 
\begin{equation}
\label{eq:3by3matrix}
\left(\begin{array}{ccc}
    x & \lambda(g_1) & 0  \\
    \lambda(g_1)^* & 0 & \lambda(g_2) \\
     0 & \lambda(g_2)^* & 0
\end{array}\right) \left( \begin{array}{ccc}
    x^{-1} & 0 & -x^{-1} \lambda(g_1 g_2)  \\
    0 & 0 & \lambda(g_2) \\ - \lambda(g_2^{-1} g_1^{-1}) x^{-1}  & \lambda(g_2^{-1}) & \lambda(g_{2}^{-1} g_1^{-1}) x^{-1} \lambda(g_1 g_2)
\end{array} \right) = I_3\otimes 1.
\end{equation}
Now consider a graph $\G$ with three vertices $u, v, w$ and edges $\diredge_1, \diredge_2$ connecting $u$ with $v$, and $v$ with $w$, respectively. Assume that $b_u=b_v=b_w=0$. If we add whole-loops on the vertex $u$ then by Lemma \ref{lem:repofATinfreegroup} $A_{\calT}$ will have the form of the first matrix in the left side of (\ref{eq:3by3matrix}). Moreover, $x$ will correspond to the loops on $u$ and it can be made invertible by putting at least two loops on $u$ and varying their Jacobi coefficients as shown in Example \ref{ex:twovertex}. So, for the cases where $x$ is invertible, zero will not be in the spectrum of $A_\T$, but by Theorem \ref{thm:mass} and since the density of states is symmetric about zero, this means that $A_\T$ has exactly two bands, each with mass $1/2$.   This provides an example of a universal cover whose smallest quotient has three vertices, and where the mass of the bands  has the form $j/2$. 

\section{Behavior of the Density of States at the Right Edge of the Spectrum} 
\label{app:densityattheedge}
%The spectral radii of a base graph and its universal cover \todo{should we change the title to "Behavior of the DOS at the edge"?} }

\label{sec:usingthesystemofequations}
Here we use the same notation and setup as the one described at the beginning of Section \ref{sec:rademacherrepresentation}. 

Using Aomoto's systems of equations for the Cauchy transforms, we may prove the following condition on when the spectral radii of a graph and its universal cover match. The proof will use Theorem 6.6 of \cite{avni2020periodic} which states that the Cauchy transforms defined in Theorem \ref{thm:system} are algebraic. Among other things, this result implies that if $\mu$ is the density of states of $A_\T$, and $\rho_r$ is the right edge of $\Spec(A_\T)$, then the (possibly infinite) limit $\lim_{\epsilon\to 0
^+}\mu((\rho_r-\epsilon, \rho_r])/\epsilon$ exists.

\begin{proposition} \label{thm:perron}
   Let $\mu$ and $\rho_r$ denote the density of states of $A_\T$ and maximum element of the spectrum of $A_\mathcal{T}$ respectively.  Suppose that 
 \begin{equation}
 \label{eq:limitpositive}
 \lim_{\eps \to 0^+} \frac{\mu((\rho_r - \eps, \rho_r])}{\eps} > 0.
 \end{equation}
%Let $w(z)$ denote the Cauchy transform of the adjacency operator on $\ell^2(V(\mathcal{T}))$, and let $\rho$ denote its spectral radius.  Then if $\lim_{t \to \rho^+} w(t) = \infty$ (where $t \in \mathbb{R}$), 
Then the maximum eigenvalue of $A_{\G}$ is equal to $\rho_r$.
\end{proposition}
\begin{proof}
%\todo{check huang rahman for general jacobi matrices, or at least weighted adjacency matrices}  
Let $w(z)$ denote the Cauchy transform of $\mu$. We begin by proving that the condition in (\ref{eq:limitpositive}) implies that $\lim_{t\to \rho_r^+} w(t) = \infty$. There are two cases. If $\mu$ has an atom at $\rho_r$ the statement is clearly true. If there is no atom at $\rho_r$, we note that  Theorem 6.6 in \cite{avni2020periodic} implies that $w(t)$ is algebraic (since the sum of algebraic functions is algebraic), and hence $\mu$ is absolutely continuous in a  neighborhood of $\rho_r$ and has a power-like behavior near $\rho_r$ (see Theorem 2.9 in \cite{anderson2008law} for a formal statement). In particular, (\ref{eq:limitpositive}) implies that there is some $C>0$ such that for all $\epsilon > 0$ small enough we have $\frac{d\mu}{dx}(s) > C$ for all $s \in (\rho_r-\epsilon, \rho_r)$. Hence, for any $\epsilon > 0$ 
$$w(\rho_r+\epsilon^2) \geq \int_{\rho_r-\epsilon}^{\rho_r} \frac{1}{\rho_r+\epsilon^2-s} d\mu(s) \geq C \int_{\rho_r-\epsilon}^{\rho_r} \frac{1}{\rho_r+\epsilon^2-s} ds = C \log\left(1+\frac{1}{\epsilon}\right).$$
So taking $\epsilon\to 0^+$ we get \begin{equation} \label{eqn:cauchyblowup} \lim_{t \to \rho_r^+} w(t) = \infty. \end{equation}
Now, for all  $u\in V(\G)$, let $w_u(z)$ be defined as in (\ref{eq:rootedcauchydef}) and  recall that $w(z) = \frac{1}{|V(\mathcal{G})|} \sum_{u \in V(\mathcal{G})} w_u(z)$.  

Note that $w_u(t)$ is analytic, positive and strictly decreasing for $t > \rho_r$, so $\lim_{t \to \rho_r^+} w_u(t)$ exists and lies in $(0, \infty]$.  
 Thus, for all $u, v \in V(\G)$, we have that $\lim_{t \to \rho_r^+} w_v(t) / w_u(t)$ exists and lies in $[0, \infty]$.  In fact, we claim $\lim_{t \to \rho_r^+} w_v(t) / w_u(t) < \infty$ when $u$ is a neighbor of $v$. If not, then setting $z = t$ in Theorem \ref{thm:system} and rearranging, we would have
 \begin{equation} \label{eqn:t} t= b_u +  \frac{2-\operatorname{deg}(u)}{2w_u(t)} + \displaystyle\sum_{\diredge\in \sigma(u)} \left(\frac{1}{4 w_{\sigma(\diredge)}(t)^2}+a_{\diredge}^2 \frac{w_{\tau(\diredge)}(t)}{w_{\sigma(\diredge)}(t)}\right)^{1/2},\end{equation}
 and the right hand side would diverge as $t \to \rho_r^+$ while the left hand side would converge, a contradiction.  

Taking the reciprocal and switching the roles of $u$ and $v$, we in fact obtain \begin{equation} \label{eqn:ratios} 0 < \lim_{t \to \rho_r^+} w_v(t) / w_u(t) < \infty\end{equation} whenever $u$ and $v$ are neighbors. By the connectedness of $\G$, (\ref{eqn:ratios}) actually holds for arbitrary vertices $u ,v$.  Together with (\ref{eqn:cauchyblowup}), this implies  $\lim_{t \to \rho_r^+} w_u(t) = \infty$ for all $u\in V(\G)$.

By (\ref{eqn:ratios}), there exist positive real numbers $\{\widetilde{w}_u\}_{u\in V(\mathcal{G})}$ with the property 
$$\frac{\widetilde{w_u}}{\widetilde{w}_v} = \lim_{t\to \rho_r^+} \frac{w_u(t)}{w_v(t)}.$$ 
Taking the limit as $t \to \rho_r^+$ of (\ref{eqn:t}), we get
%Then, from our system of equations, by letting $t$ go to $\rho_r$ and since all $w_i(t)$ blow up we get
$$\rho_r \sqrt{\widetilde{w}_u} =  b_u +  \sum_{ \diredge\in \sigma(u)} a_\diredge \sqrt{\widetilde{w}_{\tau(\diredge)}} .$$
Recalling the definition (\ref{eqn:jacobi}) of $A_\G$, the above equation explicitly shows that $\rho_r$ is an eigenvalue of $A_\G$ with an eigenvector with positive entries $\left(\sqrt{\widetilde{w}_\diredge}\right)_{\diredge \in E(\G)}$.  Now take $\lambda > 0$ large enough so that $\lambda + b_u > 0$ for all $u \in V(\G)$.  Then the entries of $A_\G + \lambda$ are nonnegative,  the eigenvector for $\rho_r + \lambda$ has positive entries, and $A_\G$ is irreducible (since as a directed graph $\G$ is strongly connected),  so by the Perron-Frobenius theorem  $\rho_r + \lambda$ is in fact the maximum eigenvalue of $A_\G + \lambda$.  Thus $\rho_r$ is the maximum eigenvalue of $A_\G$, as desired.
\end{proof}

The following theorem of Sy and Sunada will be useful.  We rephrase it in the language of Jacobi matrices on graphs.  Their original formulation is in terms of what they call a \emph{discrete Schrodinger operator}, but every graph Jacobi matrix with positive edge weights $a_\diredge$ can be represented in their framework.
\begin{theorem}[\cite{sy1992discrete}] \label{thm:sysunada}
Let $\G_1$ be a finite connected graph with no loops or multi-edges, and let $A_{\G_1}$ be a Jacobi matrix on $\G_1$ with $a_e > 0$ for all $e \in E(\G_1)$.  Let $\G_2$ be a graph covering $\G_1$, and let $A_{\G_2}$ be the lift of $A_{\G_1}$.  Then $\rho_r(A_{\G_1}) \ge \rho_r(A_{\G_2})$ with equality if and only if the deck transformation group of the covering is amenable, where $\rho_r(\cdot)$ denotes the right edge of the spectrum.
\end{theorem}

We may now prove our main result on the edge of the spectrum: \\

\noindent \textbf{Restatement of Theorem \ref{thm:vanishingdensity}}. \emph{
Let $A_\G$ be a Jacobi matrix on a finite graph $\mathcal{G}$ with at least two cycles, and let $A_\T$ be its pullback to the universal cover $\T(\G).$ Furthermore assume that $a_\diredge >0$ for all $\diredge \in E(\G)$.   Let $\mu$ and $\rho_r$ be the density of states and the maximum element of the spectrum of $A_\T$, respectively.  Then $\mu$ is absolutely continuous in a neighborhood of $\rho_r$ and $\lim_{x\to \rho_r} \frac{d\mu}{dx}(x) =0$. }
%Let $A_\G$ be the adjacency matrix of a finite graph $\mathcal{G}$ with at least two cycles, and let $A_\T$ be its pullback to the universal cover $\T(\G).$ Let $\mu$ and $\rho_r$ be the density of states and spectral radius of $\A_\T$ respectively.  Then $\mu$ is absolutely continuous in a neighborhood of $\rho_r$ and $\lim_{x\to \rho_r} \frac{d\mu}{dx}(x) =0$. }

%\todo{substitute in Sunada-Sy?}
\begin{proof}
Since $\G$ has at least two cycles, the deck transformation group of the universal cover contains the free group on two generators as a subgroup, and is therefore nonamenable.   If $\G$ has no loops and no multi-edges, Theorem \ref{thm:sysunada} then implies that the maximum eigenvalue of $A_\G$ is strictly greater than $\rho_r$.  Thus, applying the contrapositive of Proposition \ref{thm:perron} we have \[\lim_{\eps \to 0^+} \frac{\mu((\rho_r - \eps, \rho_r])}{\eps} = 0.\] 
In particular $\mu$ has no atoms in $(\rho_r - \eps, \rho_r]$ for $\eps$ sufficiently small.  Also, $\mu$ has no singular continuous part (\cite{avni2020periodic}) due to algebraicity of the Cauchy transforms.  Thus, the conclusion follows.

If, on the other hand, $\G$ has loops or multi-edges, we use the following workaround.  One may first take a finite cover $\G'$ of $\G$ that does not contain loops or multi-edges.  (If $j$ is the maximum number of loops at any vertex and $k$ is the maximum number of edges in any multi-edge, a cover with $\max \{2j+1, k\}$ sheets suffices.)   Since $\T(\G) = \T(\G')$, the only thing left to check is that $\rho_r(A_\G) \ge \rho_r(A_{\G'})$, where $\rho_r(\cdot)$ denotes the maximum eigenvalue.  This holds because the top eigenvector $v$ of $A_{\G'}$ can be projected down to an eigenvector $v'$ of $A_\G$ for the same eigenvalue by summing the entries of $v$ in each fiber.  Note that $v'$ is nonzero because $v$ has positive entries, by Perron-Frobenius applied to $A_\G + \lambda$ for $\lambda > 0$ sufficiently large (as the $b_u$ may be negative).  Thus $\rho_r(A_\G) \ge \rho_r(A_{\G'})$ as desired.
\end{proof}

\end{document}